\pdfoutput=1 
\documentclass{amsart}
\usepackage[svgnames]{xcolor}
\usepackage[unicode,colorlinks=true,linktocpage=true,citecolor=ForestGreen,linkcolor=MediumOrchid]{hyperref}
\usepackage{tikz-cd}
\usepackage{amssymb}
\usepackage{mathtools}
\usepackage{microtype}
\usepackage[capitalise,noabbrev]{cleveref}
\usepackage{mathrsfs}
\usepackage{enumitem}
\usepackage{pinlabel}

\DeclareRobustCommand{\SkipTocEntry}[5]{}

\newtheorem{theorem}{Theorem}[section]
\newtheorem{thmx}{Theorem}

\newtheorem{proposition}[theorem]{Proposition}
\newtheorem{corollary}[theorem]{Corollary}
\newtheorem{lemma}[theorem]{Lemma}
\theoremstyle{definition}
\newtheorem{definition}[theorem]{Definition}
\newtheorem{example}[theorem]{Example}
\newtheorem{remark}[theorem]{Remark}
\newtheorem*{spremark}{Remark}

\newtheorem{notation}[theorem]{Notation}
\newtheorem{construction}[theorem]{Construction}
\newtheorem{convention}[theorem]{Convention}

\newcommand{\rat}{\rightarrowtail}

\DeclareFontFamily{U}{min}{}
\DeclareFontShape{U}{min}{m}{n}{<-> udmj30}{}
\newcommand{\yon}{\!\text{\usefont{U}{min}{m}{n}\symbol{'210}}\!}

\newcommand{\actrm}{\mathrm{act}}
\newcommand{\intrm}{\mathrm{int}}
\newcommand{\elrm}{\mathrm{el}}
\newcommand{\oprm}{\mathrm{op}}

\newcommand{\finset}{\mathbf{FinSet}}
\newcommand{\finsetskel}{\mathbf{F}}
\newcommand{\properad}{\mathbf{Ppd}}
\newcommand{\set}{\mathbf{Set}}
\newcommand{\sset}{\mathbf{sSet}}
\newcommand{\cat}{\mathbf{Cat}}
\newcommand{\perm}{\mathbf{Perm}}
\newcommand{\gpd}{\mathbf{Gpd}}
\newcommand{\pgcat}{\mathbf{G}}
\newcommand{\csp}{\mathbf{Csp}}
\newcommand{\slcc}{\mathbf{SLCC}}
\newcommand{\lcc}{\mathbf{LCC}}
\DeclareMathOperator{\id}{id}
\DeclareMathOperator{\ob}{ob}
\DeclareMathOperator{\mor}{mor}
\DeclareMathOperator{\dom}{dom}
\DeclareMathOperator{\cod}{cod}
\DeclareMathOperator*{\colim}{colim}

\newcommand{\finsetstarop}{\finset_\ast^\oprm}

\newcommand{\name}[1]{\ulcorner #1\urcorner}

\DeclareMathOperator{\map}{Map}

\DeclareMathAlphabet\EuScript{U}{eus}{m}{n}
\SetMathAlphabet\EuScript{bold}{U}{eus}{b}{n}
\newcommand{\igpd}{\EuScript{S}}

\newcommand{\LL}{\mathbf{L}}
\newcommand{\simpcat}{\mathbf{\Delta}}
\newcommand{\LLc}{\mathbf{L}_\mathrm{c}}

\newcommand{\kong}{\mathbf{K}}
\newcommand{\WW}{\mathbf{W}}
\newcommand{\XX}{\mathbf{X}}

\DeclareMathOperator{\pre}{Psh}
\DeclareMathOperator{\seg}{Seg}
\DeclareMathOperator{\segcore}{Sc}

\DeclareMathOperator{\fun}{Fun}

\newcommand{\homcon}{\hom^{\mathrm{c}}}
\newcommand{\homred}{\hom^{\mathrm{r}}}
\newcommand{\morred}{\mor^{\mathrm{r}}}
\newcommand{\morcon}{\mor^{\mathrm{c}}}

\newcommand{\uk}{\underline{k}}
\newcommand{\uj}{\underline{j}}
\newcommand{\um}{\underline{m}}
\newcommand{\un}{\underline{n}}

\newcommand{\twplus}[1]{\underset{#1}{+}}
\newcommand{\twten}{\mathrel{\tilde\otimes}}
\newcommand{\ten}{\otimes}

\title{Labelled cospan categories and properads}

\author{Jonathan Beardsley}
\address{Department of Mathematics and Statistics, University of Nevada, Reno, USA}
\email{jbeardsley@unr.edu}
\urladdr{https://www.jonathanbeardsley.com/}
\author{Philip Hackney}
\address{Department of Mathematics, University of Louisiana at Lafayette, USA}
\email{philip@phck.net} 
\urladdr{http://phck.net}
\thanks{This work was supported by grants from the Simons Foundation (\#853272, JB; \#850849, PH). 
This material is partially based upon work supported by the National Science Foundation under Grant No. DMS-1928930 while the second author participated in a program supported by the Mathematical Sciences Research Institute. 
The program was held in the summer of 2022 in partnership with the Universidad Nacional Autónoma de México}
\date{April 27, 2023}



\begin{document}

\begin{abstract}
We prove Steinebrunner's conjecture on the biequivalence between (colored) properads and labelled cospan categories.
The main part of the work is to establish a 1-categorical, strict version of the conjecture, showing that the category of properads is equivalent to a category of strict labelled cospan categories via the symmetric monoidal envelope functor.
\end{abstract}

\maketitle

\tableofcontents

One way to encode the data of a cobordism is as a cospan $M_1\to P \leftarrow M_2$, where $P$ is a manifold and $M_1$ and $M_2$ are its ``left'' and ``right'' boundaries, respectively.
By taking connected components, each such cospan gives a cospan of finite sets.
This functor from cobordisms to cospans of finite sets precisely relates decompositions of objects and morphisms between the two categories.
An abstract version of this appears under the name \emph{labelled cospan category} in \cite{Steinebrunner}.
Therein, Steinebrunner develops the general theory of labelled cospan categories and uses it to show, among other things, that the classifying space of the category of 2-dimensional cobordisms is rationally equivalent to $S^1$. 

\begin{figure}
\labellist
\small\hair 2pt
 \pinlabel {$a$} [B] at 11 203
 \pinlabel {$b$} [B] at 45 203
 \pinlabel {$c$} [B] at 127 203
 \pinlabel {$d$} [B] at 158 203
 \pinlabel {$f$} [B] at 225 203
 \pinlabel {$g$} [B] at 271 203
 \pinlabel {$h$} [B] at 317 203
 \pinlabel {$f$} [B] at 494 237
 \pinlabel {$g$} [B] at 541 237
 \pinlabel {$h$} [B] at 587 237
 \pinlabel {$c$} [B] at 651 237
 \pinlabel {$f$} [B] at 747 203
 \pinlabel {$g$} [B] at 794 203
 \pinlabel {$h$} [B] at 839 203
 \pinlabel {$c$} [B] at 866 203
 \pinlabel {$e$} [B] at 87 28
 \pinlabel {$i$} [B] at 153 28
 \pinlabel {$j$} [B] at 208 28
 \pinlabel {$a$} [B] at 273 28
 \pinlabel {$b$} [B] at 312 28
 \pinlabel {$d$} [B] at 341 28
 \pinlabel {$i$} at 426 0
 \pinlabel {$j$} at 480 0
 \pinlabel {$e$} at 588 0
 \pinlabel {$i$} [B] at 676 28
 \pinlabel {$j$} [B] at 733 28
 \pinlabel {$e$} [B] at 796 28
 \pinlabel {$a$} [ ] at 530 115
 \pinlabel {$b$} [ ] at 555 118
 \pinlabel {$d$} [ ] at 608 118
 \pinlabel {$\rightsquigarrow$} [ ] at 394 120
 \pinlabel {$\rightsquigarrow$} [ ] at 679 120
\endlabellist
\centering
\includegraphics[width=\textwidth]{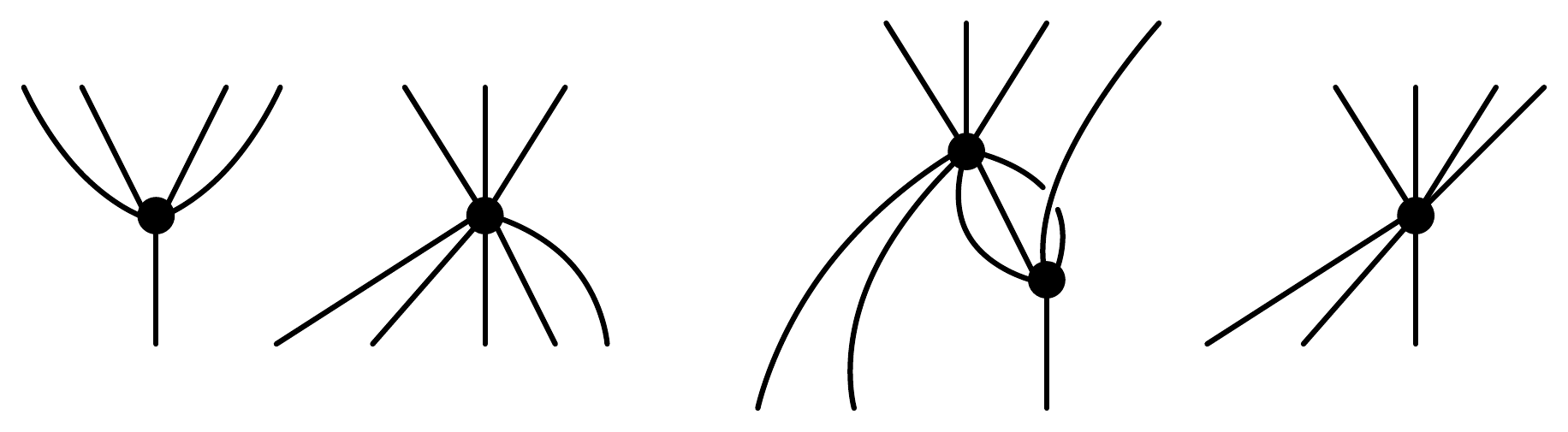}
\caption{Properadic composition of type $124;345$}
\label{fig:properadic composition}
\end{figure}
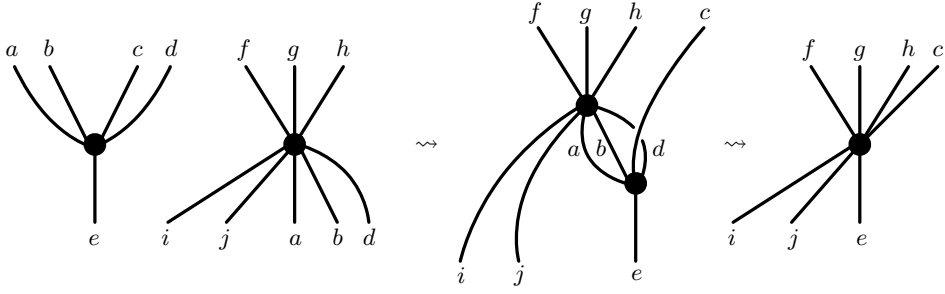

The cobordism category is a prop in the sense of Adams and Mac Lane. 
It is in fact the free prop on a \textit{properad} of connected cobordisms. 
Properads, introduced in \cite{Vallette:KDP} and independently under the name compact symmetric polycategory in \cite{Duncan:TQC}, are like props but without horizontal composition.
Operations may have multiple inputs and outputs, and can be composed by attaching some (nonzero number) of inputs of one to the outputs of another.
We do not give the precise definition (for that, see \cite[Chapter 3]{HRYbook} or \cite[11.7]{YauJohnson:FPAM}), but rather an illustration in \cref{fig:properadic composition} where the letters represent colors in the properad. 
In the example of cobordisms, the connected cobordisms form a properad and then become a prop when we allow disjoint unions.
Steinebrunner conjectured that a similar connection between properads and labelled cospan categories holds in generality.
Our main result is the following:

\begin{thmx}\label{theorem A}
The 2-category of properads is biequivalent to the 2-category of labelled cospan categories.
\end{thmx}

This appears as \cref{main theorem as corollary} below, and establishes the first part of Conjecture~2.31 of \cite{Steinebrunner} 
(we do not address the second part of the conjecture which concerns $\infty$-properads).
The underlying functor of 1-categories from properads to labelled cospan categories is not an equivalence, so 2-categorical structure is essential.
It is also the case (see \cref{prop uniqueness of lcc structure}) that if a symmetric monoidal category admits the structure of a labelled cospan category, then this structure is unique up to equivalence.
Thus another interpretation of \cref{theorem A} is that it provides a faithful inclusion of properads into symmetric monoidal categories and identifies its image.

The main effort of this paper consists of a careful analysis of the symmetric monoidal envelope, which takes properads to symmetric monoidal categories; this functor extends the classical envelope of an operad \cite[5.4.1]{LodayVallette:AO}, which goes back to Boardman--Vogt.
The envelope of the terminal properad is the category of cospans of finite sets, and applying the envelope to the unique map from a properad to the terminal properad yields a labelled cospan category.
This construction actually produces a more rigid kind of object, which we have called a \textit{strict labelled cospan category}.
In the last section we establish a biequivalence (\cref{theorem slcc lcc biequiv}) between labelled cospan categories and strict labelled cospan categories, which we combine with the following to obtain the main result.

\begin{thmx}[\cref{cor main theorem,thm properad vs slcc 2-cat}]
There is a strict 2-equivalence between the 2-category of properads and the 2-category of strict labelled cospan categories.
\end{thmx}

In particular, this yields an equivalence of categories between the underlying 1-categories.
Though our focus in this introduction has been on 2-categories, it is this 1-categorical equivalence (\cref{cor main theorem}) that is the core of the result.

A key ingredient in this paper is an efficient description of the symmetric monoidal envelope of a properad $P$.
We use an alternative description of properads as Segal presheaves on a category of level graphs $\LL$.
That is, we use an equivalence of categories $\properad \simeq \seg(\LL)$; which appears below as \cref{prop properads as presheaves} (the hard work of establishing this equivalence was already done in \cite{HRYbook} and \cite{ChuHackney}).
The upshot is that a Segal $\LL$-presheaf looks much more like a symmetric monoidal category than a properad does (in particular, it is important that the category $\LL$ contains disconnected graphs, as these account for the monoidal structure).
To form the envelope is then a relatively straightforward quotienting process, though managing symmetry isomorphisms still requires some care.

It is much more involved to go back in the other direction.
Namely, given a strict labelled cospan category $C$, we would like to produce a Segal $\LL$-presheaf that it comes from.
There is an active-inert factorization system on $\LL$, and there is essentially only one possible choice of an $\LL_{\intrm}$-presheaf that could be the restriction of our desired $\LL$-presheaf.
We must then extend the $\LL_{\intrm}$-presheaf by defining its action on active maps.
Here we make use of additional simplicial structure: by using heights of level graphs, $\LL$ is fibered over the simplicial category $\simpcat$.
The main thing that is missing are the inner face and degeneracy operators arising from this additional simplicial structure, and coskeletality arguments allow us to extend from low simplicial degree (using the composition and identities in $C$) to arbitrary simplicial degree.

\begin{spremark}[Related work]
As mentioned above, we can describe a properad by giving its envelope, along with map to the envelope of the terminal properad. 
The same thing can be done for operads, though we are not aware of any classical literature which pursues this idea. 
However, this is precisely the approach that is taken by recent work of Haugseng--Kock \cite{HaugsengKock:IOSMIC} for $\infty$-operads (see also \cite{BarkanHaugsengSteinebrunner:EAP}).
Following the initial public version of this paper, the preprints \cite{KaufmannMonaco:PCPUFC} and \cite{Barkan-Steinebrunner} appeared and offered related theorems to ours. 
The work by Barkan--Steinebrunner on $\infty$-properads is related to the Haugseng--Kock approach for $\infty$-operads.
Some important results of \cite{Barkan-Steinebrunner} are $\infty$-categorical versions our theorems, and by restricting to discrete objects they recover the $(2,1)$-categorical version  of \cref{theorem A} (which is also an immediate consequence of \cref{theorem A}).
Their elegant theory relies on the idea of `equifibered maps of $\mathbb{E}_\infty$-monoids,' and we highly recommend \cite{Barkan-Steinebrunner} for an important alternative perspective.
They also give information concerning the relationship with the `hereditary unique factorization categories' of Kaufmann--Monaco \cite{KaufmannMonaco:PCPUFC}.
\end{spremark}

\addtocontents{toc}{\SkipTocEntry}
\subsection*{Acknowledgements}
We are grateful to Jan Steinebrunner for feedback on an earlier version of this manuscript, as well as a number of interesting suggestions.
We are thankful for several suggestions and comments from anonymous referees, which have helped us improve the presentation and correct several oversights.
We also thank Joe Moeller, Marcy Robertson, and Donald Yau for helpful discussions.

\section{Background}

The category $\finset$ is the category of finite sets and arbitrary functions between them.
Let $\finsetskel \subseteq \finset$ be the full subcategory whose objects are the ordered sets $\underline{k} = \{ 1, 2, \dots, k \}$ for $k\geq 0$.
We will write $\simpcat$ for the (topologist's) simplicial category, whose objects are the ordered sets $[n] = \{ 0 < 1 < \dots < n \}$ for $n\geq 0$ and morphisms preserve the $\leq$ relation.
Note the shift in cardinality if you regard $\simpcat$ as a subcategory of $\finsetskel$.

The category $\csp$ has objects the same as $\finset$, and morphisms from $A$ to $B$ are equivalence classes of cospans $A \to C \leftarrow B$, where two are identified if there is an isomorphism on the middle term:
\[ \begin{tikzcd}[sep=tiny]
& C \ar[dd,"\cong"] \\
A \ar[ur] \ar[dr] & & B \ar[ul] \ar[dl] \\
& C'
\end{tikzcd} \]
Composition of morphisms in $\csp$ is given by pushout
\[ \begin{tikzcd}[sep=tiny]
A \ar[dr] \ar[ddrr, bend right] & & B \ar[dl] \ar[dr] \ar[dd,phantom, "\rotatebox{-45}{$\ulcorner$}" very near end]& & C \ar[dl] \ar[ddll, bend left] \\
& D \ar[dr] & & E \ar[dl]\\
& & \bullet
\end{tikzcd} \]
Choices of coproducts of finite sets yields a monoidal structure on $\csp$ with $\varnothing$ as the monoidal unit.
This is not a cocartesian monoidal category, however.

\subsection{Labelled cospan categories}
This work is concerned with Steinebrunner's notion of labelled cospan categories from \cite{Steinebrunner}, which are certain symmetric monoidal categories living over $\csp$.
We briefly recall some definitions.

\begin{definition}\label{def connected reduced}
Let $\pi \colon C \to \csp$ be a symmetric monoidal functor.
For the moment, it is convenient to write 
\[ \begin{tikzcd}[sep=small]
\pi c \ar[dr,"l f"'] & & \pi d \ar[dl,"r f"] \\
& mf
\end{tikzcd} \]
for $\pi(f \colon c \to d)$.
\begin{itemize}
\item An object $c\in C$ is \emph{connected} if $\pi(c)$ has cardinality one.
\item A morphism $f\colon c \to d$ is \emph{connected} if $m(f)$ has cardinality one.
\item A morphism $f\colon c \to d$ is \emph{reduced} if the cospan $\pi(f)$ is jointly surjective.
\item $\homcon(c,d) \subset \hom(c,d)$ denotes the connected morphisms and $\homred(c,d) \subset \hom(c,d)$ denotes the reduced morphisms.
\end{itemize}
\end{definition}

By examining the diagrams
\[ \begin{tikzcd}[sep=tiny]
\underline{0} \ar[dr] & & \underline{0} \ar[dl] \ar[dr] & &  \underline{0} \ar[dl] \\
& \underline{1} \ar[dr] & & \underline{1} \ar[dl]\\
& & \underline{2}
\end{tikzcd} \qquad \& \qquad
\begin{tikzcd}[sep=tiny]
\underline{0} \ar[dr] & & \underline{1} \ar[dl] \ar[dr] & &  \underline{0} \ar[dl] \\
& \underline{1} \ar[dr] & & \underline{1} \ar[dl]\\
& & \underline{1}
\end{tikzcd} 
\]
whose diamonds are pushouts, we see that the sets of connected morphisms and reduced morphisms are not closed under composition.

If $A \rightarrow X \leftarrow B$ represents an isomorphism in $\csp$, then both legs of the cospan are bijections.
Thus every isomorphism in $C$ is reduced, and every isomorphism involving a connected object is connected.

\begin{definition}[Steinebrunner]\label{def lcc}
A \emph{labelled cospan category} is a symmetric monoidal functor $\pi \colon C \to \csp$ satisfying the following:
\begin{enumerate}
\item If $\pi(c)$ has cardinality $n$, then we can find $n$ objects $c_1, \dots, c_n$ which are connected so that $c$ is isomorphic to $c_1 \otimes \dots \otimes c_n$.
\label{lcc: obj decomp}
\item If $\mathbf{1}$ is the tensor unit of $C$, then the abelian monoid $\hom(\mathbf{1},\mathbf{1})$ is freely generated by the set $\homcon(\mathbf{1},\mathbf{1})$ of connected morphisms.
\label{lcc: free gen}
\item The map
\[ \begin{tikzcd}
  \homred(c,d) \times \hom(\mathbf{1},\mathbf{1}) \rar{\otimes} & \hom(c,d)
\end{tikzcd} \]
is a bijection.
\label{lcc: reduced and free}
\item For each four objects $c,d,c',d'$ in $C$, the following square is cartesian.
\[ \begin{tikzcd}
\homred(c,d) \times \homred(c',d') \rar{\otimes} \dar{\pi} \arrow[dr, phantom, "\lrcorner" very near start] & \homred(c\otimes c', d\otimes d') \dar{\pi} \\
\homred_\csp(\pi c,\pi d) \times \homred_\csp(\pi c',\pi d') \rar{\ten} & \homred_\csp(\pi c \ten \pi c',\pi d \ten \pi d') 
\end{tikzcd} \]
\label{lcc: pullback}
\end{enumerate}
A map of labelled cospan categories is a symmetric monoidal functor and a choice of monoidal natural isomorphism making the triangle over $\csp$ commute.
\[ \begin{tikzcd}[row sep=small, column sep=tiny]
C \ar[rr] \ar[ddr] & & C' \ar[ddl] \\
\ar[rr,"\text{\scriptsize $\cong$}", phantom] & & {} \\
& \csp & 
\end{tikzcd} \]
\end{definition}

\begin{remark}\label{rmk unique ni}
For a fixed symmetric monoidal functor $f \colon C \to C'$, there is at most one monoidal natural isomorphism as displayed above; in other words, the forgetful functor from labelled cospan categories to symmetric monoidal categories is faithful.
Uniqueness of the monoidal natural isomorphism follows from \cref{def lcc}\eqref{lcc: obj decomp} and the fact that there are unique isomorphisms in $\csp$ between objects of cardinality zero or one.
Given two such natural isomorphisms, the following diagram of isomorphisms commutes for each of them, where $c \cong c_1 \otimes \dots \otimes c_n$ is the decomposition into connected objects guaranteed by \eqref{lcc: obj decomp}.
\[ \begin{tikzcd}
\pi' f c_1 \otimes \dots \otimes \pi'f c_n \rar \dar{\cong} & \pi c_1 \otimes \dots \otimes \pi c_n  \dar{\cong} \\
\pi' f (c_1 \otimes \dots \otimes c_n) \rar \dar{\cong} & \pi' f (c_1 \otimes \dots \otimes c_n) \dar{\cong}  \\
\pi' f c \rar & \pi c
\end{tikzcd} \]
Since the top map is the same for both natural isomorphisms, so too is the bottom map.
\end{remark}

\begin{remark}[2-categorical structure]\label{rmk: 2-cat struct}
We can regard the collection of labelled cospan categories as a 2-category, where a map between morphisms $(f,\alpha)$ and $(g,\beta)$ having the same source and target
\[ \begin{tikzcd}[row sep=small, column sep=tiny]
C \ar[rr,"f"] \ar[ddr] & & C' \ar[ddl] 
&&[+2em] &
C \ar[rr,"g"] \ar[ddr] & & C' \ar[ddl]
\\
\ar[rr,"\text{\scriptsize $\alpha \cong$}", phantom] & & {} 
&\rar[color=blue] &{} &
\ar[rr,"\text{\scriptsize $\beta \cong$}", phantom] & & {} 
\\
& \csp & 
&&&
& \csp & 
\end{tikzcd} \]
is a monoidal natural transformation $\gamma \colon f \Rightarrow g$ so that the composite natural transformation
\[ \begin{tikzcd}
C \ar[rr,"g"', bend right=20] \ar[rr,"f", bend left=20] \ar[rr,"\text{\scriptsize $\gamma \Downarrow$}",phantom] \ar[ddr] & & C' \ar[ddl]
\\
\ar[rr,"\text{\scriptsize $\beta \cong$}", phantom, bend right=10] & & {} 
\\
& \csp & 
\end{tikzcd} \]
is an isomorphism. 
By \cref{rmk unique ni}, we then have this composite is $\alpha$.
Steinebrunner only needs this when $\gamma$ is an isomorphism, that is, he considers labelled cospan categories as forming a $(2,1)$-category.
\end{remark}

We will return to this 2-categorical structure of labelled cospan categories in \cref{section 2-cat structures}.

\subsection{The category of level graphs}\label{subsec level graphs}
In this section we recall the category of level graphs $\LL$ from \cite[\S 2.1]{ChuHackney}. 
As mentioned below \cite[Lemma 2.1.9]{ChuHackney}, this category is closely related to the \emph{double} category of cospans, an enhancement of $\csp$.
In \cref{properad to categories} will see how to rederive $\csp$ (up to equivalence) from $\LL$.

The category $\mathscr{L}^n$ has, as its objects, pairs $(i,j)$ with $0 \leq i \leq j \leq n$, and as morphisms, unique maps $(i,j) \to (k,\ell)$ whenever $0\leq k \leq i \leq j \leq \ell \leq n$.
This presentation is as in \cite[Definition 2.16]{ChuHackney}, though it is also true that $\mathscr{L}^n$ is isomorphic to the twisted arrow category of $[n] = \{ 0 \to 1 \to \dots \to n \}$.
Note that every square in $\mathscr{L}^n$ commutes and is a pushout.
Here is $\mathscr{L}^3$:
\[ \begin{tikzcd}[column sep=tiny, row sep=tiny]
(0,0) \arrow[dr] & & (1,1) \arrow[dr]\arrow[dl] & & (2, 2) \arrow[dr]\arrow[dl] & & (3,3) \arrow[dl] \\
& (0,1) \arrow[dr]& &  (1,2)  \arrow[dr]\arrow[dl] & & (2,3)\arrow[dl] \\
& & (0,2)\arrow[dr] & & (1,3)\arrow[dl] \\
& & & (0,3)
\end{tikzcd} \]

A \emph{level graph of height $n$} is a functor $G\colon \mathscr{L}^n \to \finsetskel$ that sends every square to a pushout.\footnote{This is slightly different from \cite{ChuHackney}, where the target category was all finite sets and the functor only concerned the top layers $\mathscr{L}_0^n \to \finset$. This makes no substantial difference, and we arrive at an equivalent category of level graphs below; compare with \cite[2.1.19]{ChuHackney}.}
In particular, level graphs of height $0$ are just objects of $\finsetskel$, and level graphs of height $1$ are just cospans in $\finsetskel$.
We write $G_{ij}$ or $G_{i,j}$ for the value of $G$ on the object $(i,j)$ and do not label the structural maps $G_{ij} \to G_{k\ell}$.
\cref{figure height 2} gives a pictorial example of a particular height 2 level graph the following shape:
\[ \begin{tikzcd}[sep=tiny]
\underline{6} \ar[dr] & & \underline{6} \ar[dl] \ar[dr] & &  \underline{7} \ar[dl] \\
& \underline{4} \ar[dr] & & \underline{3} \ar[dl]\\
& & \underline{2}
\end{tikzcd} \]

\begin{figure}
\includegraphics[width=0.6\textwidth]{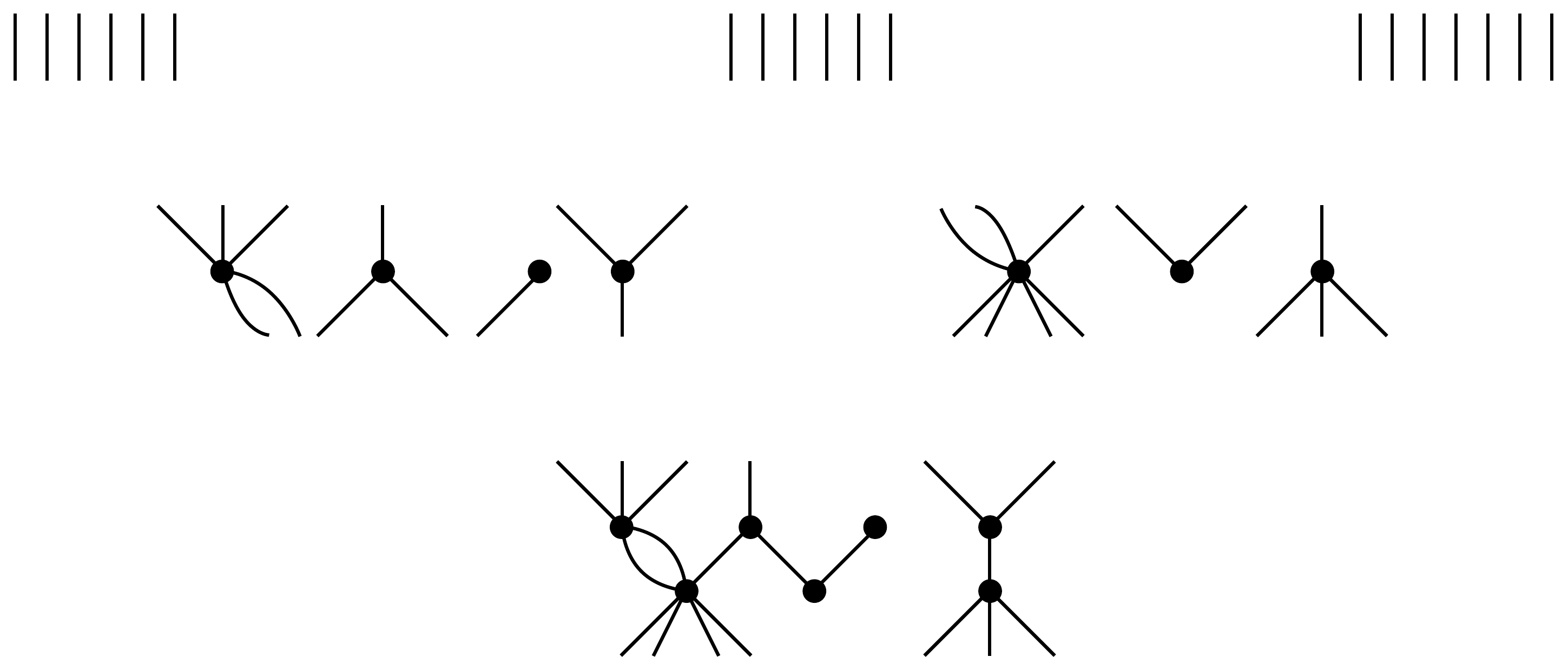}
\caption{A level graph of height 2}\label{figure height 2}
\end{figure}

We consider the subcategory $\LL_n$ of $\fun(\mathscr{L}^n, \finsetskel)$ whose objects are the level graphs and whose morphisms are the natural transformations $G\Rightarrow H$ satisfying the following two properties for each $0\leq i \leq j \leq n$:
\begin{enumerate}
\item The map $G_{ij} \to H_{ij}$ is a monomorphism.
\item The naturality square
\[ \begin{tikzcd}
G_{ij} \rar \dar \arrow[dr, phantom, "\lrcorner" very near start] & H_{ij} \dar  \\
G_{0n} \rar & H_{0n}
\end{tikzcd} \]
is cartesian.
\end{enumerate}
By the pasting law for pullbacks, the second condition is equivalent to the natural transformation being cartesian, as was required in \cite[Definition 2.1.16]{ChuHackney}.
This defines a functor $\LL_\bullet \colon \simpcat^\oprm \to \cat$, and we write $\LL \to \simpcat$ for the associated Grothendieck fibration. (The version of $\LL$ appearing in \cite{ChuHackney} is a skeleton of this one.)

Some remarks on the simplicial category $\LL_{\bullet}$ will make later proofs easier to understand (see \cite{ChuHackney} for complete details).
First we note that, in fact, the categories $\mathscr{L}^n$ assemble into a cosimplicial category.
Given a morphism $\alpha\colon \left[m\right]\to\left[n\right]$ in $\simpcat$ the functor $\alpha^\ast\colon\mathscr{L}^m\to\mathscr{L}^n$ takes the pair $(i,j)$ to the $(\alpha i, \alpha j)$.
For instance all of the induced functors $\mathscr{L}^0\to \mathscr{L}^n$ take the unique object $(0,0)$ to some object $(i,i)$ in $\mathscr{L}^n$ (and never objects of the form $(i,j)$ with $i\neq j$).
One can also consider the three coface morphisms $\mathscr{L}^1\to\mathscr{L}^2$.
These can be visualized with color as follows:

\begin{center}
\vspace{0.5cm}
\begin{tikzcd}[sep=small]
  \textcolor{red}{(0,0)} && \textcolor{red}{(1,1)} &[1cm]&& \textcolor{red}{(0,0)} && \textcolor{red}{(1,1)} && {(2,2)} \\
  & \textcolor{red}{(0,1)} &&&&& \textcolor{red}{(0,1)} && {(1,2)} \\
  &&&&&&& {(0,2)}
  \arrow[color={red}, from=1-1, to=2-2]
  \arrow[color={red}, from=1-3, to=2-2]
  \arrow[color={red}, from=1-6, to=2-7]
  \arrow[color={red}, from=1-8, to=2-7]
  \arrow[from=2-7, to=3-8]
  \arrow[from=1-10, to=2-9]
  \arrow[from=1-8, to=2-9]
  \arrow[from=2-9, to=3-8]
  \arrow[shift right=7, shorten <=15pt, shorten >=15pt, maps to, from=1-3, to=1-6, "d_2"]
\end{tikzcd}

\vspace{0.25cm}

\begin{tikzcd}[sep=small]
  \textcolor{red}{(0,0)} && \textcolor{red}{(1,1)} &[1cm]&& \textcolor{red}{(0,0)} && {(1,1)} && \textcolor{red}{(2,2)} \\
  & \textcolor{red}{(0,1)} &&&&& \textcolor{red}{(0,1)} && \textcolor{red}{(1,2)} \\
  &&&&&&& \textcolor{red}{(0,2)}
  \arrow[color={red}, from=1-1, to=2-2]
  \arrow[color={red}, from=1-3, to=2-2]
  \arrow[color={red}, from=1-6, to=2-7]
  \arrow[from=1-8, to=2-7]
  \arrow[color={red}, from=2-7, to=3-8]
  \arrow[color={red}, from=1-10, to=2-9]
  \arrow[from=1-8, to=2-9]
  \arrow[color={red}, from=2-9, to=3-8]
  \arrow[shift right=7, shorten <=15pt, shorten >=15pt, maps to, from=1-3, to=1-6, "d_1"]
\end{tikzcd}

\vspace{0.25cm}

\begin{tikzcd}[sep=small]
  \textcolor{red}{(0,0)} && \textcolor{red}{(1,1)} &[1cm]&& {(0,0)} && \textcolor{red}{(1,1)} && \textcolor{red}{(2,2)} \\
  & \textcolor{red}{(0,1)} &&&&& {(0,1)} && \textcolor{red}{(1,2)} \\
  &&&&&&& {(0,2)}
  \arrow[color={red}, from=1-1, to=2-2]
  \arrow[color={red}, from=1-3, to=2-2]
  \arrow[from=1-6, to=2-7]
  \arrow[from=1-8, to=2-7]
  \arrow[from=2-7, to=3-8]
  \arrow[color={red}, from=1-10, to=2-9]
  \arrow[color={red}, from=1-8, to=2-9]
  \arrow[from=2-9, to=3-8]
  \arrow[shift right=7, shorten <=15pt, shorten >=15pt, maps to, from=1-3, to=1-6, "d_0"]
\end{tikzcd}

\vspace{0.5cm}

\end{center}

These induce functors $\LL^2\to \LL^1$ that, respectively, truncate level 2 of a height 2 level graph, contract level 1 of a height 2 level graph, and truncate level 0 of a height 2 level graph.
More generally, on a height $n$ level graph $G\colon \mathscr{L}^n\to \finsetskel$, a morphism $\alpha\colon [m]\to [n]$ in $\simpcat$ induces a functor $\LL^n\to \LL^m$ which precomposes with $\mathscr{L}^m\to\mathscr{L}^n$, i.e.~$(\alpha^\ast G)_{ij}=G_{\alpha i,\alpha j}$ for each $(i,j)\in\mathscr{L}^m$.

If $G \to H$ is a map in $\LL_n$, then it is uniquely determined by $\{G_{ii} \to H_{ii}\}_{0\leq i \leq n}$ and $\{ G_{i-1,i} \to H_{i-1,i} \}_{1\leq i \leq n}$, since the functors $G$ and $H$ are pushout-preserving.
This implies that $\LL_n \to \LL_k \times_{\LL_0} \LL_{n-k}$ is fully faithful.
Further, this functor is surjective on objects, hence an equivalence of categories.
It is not injective on objects (unless $k=0$ or $n$).

\subsection{Properads as Segal \texorpdfstring{$\LL$}{L}-presheaves}\label{subsec properads as L presheaves}

If $G$ is a level graph of height $m$ and $H$ is a level graph of height $n$, then an \emph{active map} $G\to H$ is one whose image in $\simpcat$ is active (preserves the top and bottom elements) and where $G_{0m}\to H_{0n}$ is a bijection \cite[Remark 2.1.24]{ChuHackney}.
\emph{Inert maps} $G \rat H$ are precisely those whose image in $\simpcat$ is inert (distance preserving).

\begin{definition}[Elementary objects]\label{def elementary}
The category $\LL$ contains several \emph{elementary objects}.
These consist of the edge $\mathfrak{e} \in \LL_0$, which corresponds to $\underline{1}$ using $\ob(\LL_0) = \ob(\finsetskel)$, and a collection of $m,n$-corollas $\mathfrak{c}_{m,n} \in \LL_1$ for non-negative integers $m$ and $n$.
The graph $\mathfrak{c}_{m,n}$ is the following:
\[ \begin{tikzcd}[sep=tiny]
\underline{m} \ar[dr] & & \underline{n} \ar[dl] \\
& \underline{1}
\end{tikzcd} \]
\end{definition}

\begin{definition}\label{def Lint Lel and Lact}
The above allows to define three subcategories of $\LL$:
\begin{itemize}
    \item Write $\LL_\intrm$ for the wide subcategory of $\LL$ on the inert morphisms.
    \item Write $\LL_\elrm$ for the full subcategory of $\LL_\intrm$ spanned by the elementary graphs of \cref{def elementary}. 
    \item Write $\LL_\actrm$ for the wide subcategory of $\LL$ on the active morphisms. 
\end{itemize}
\end{definition}

\begin{remark}
The three subcategories of \cref{def Lint Lel and Lact} make $\LL^\oprm$ into an algebraic pattern in the sense of \cite[Definition 2.1]{ChuHaugseng}.
\end{remark}

\begin{definition}[Segal objects]\label{def Segal}
Let $\LL^\elrm_{/G}$ denote the category whose objects are inert maps $E\rat G$ with $E$ elementary, and whose morphisms are commutative triangles with all maps inert.
A presheaf $X \in \pre(\LL)$ is said to be \emph{Segal} if
\[
X_G \to \lim_{E\in (\LL^\elrm_{/G})^\oprm} X_E
\]
is a bijection for all $G\in \LL$.
We write $\seg(\LL) \subseteq \pre(\LL)$ for the full subcategory consisting of the Segal presheaves.
\end{definition}

\begin{definition}[Segal core]
Given $G\in \LL$, the \emph{Segal core} is the following colimit in $\pre(\LL)$
\[
\segcore(G) \coloneqq \colim_{E \in \LL^\elrm_{/G}} \yon(E) \to \yon(G)
\]
which comes equipped with a map to the representable object $\yon(G) = \hom_{\LL}(-,G)$.
\end{definition}

A presheaf $X \in \pre(\LL)$ is Segal if and only if it is local with respect to the Segal core inclusions, that is, if and only if 
\[
  \hom(\yon(G), X) \to \hom(\segcore(G), X)
\]
is a bijection for all graphs $G$.

\begin{remark}\label{examples segality}
Suppose $X\in \pre(\LL)$ is Segal.
We list explicit consequences. 
The first two are related to the `simplicial' direction of $\LL$; the second of these will play a role in the proof of \cref{prop SX segal}.
The third of these will be used repeatedly, and is about the behavior in the fibers.
\begin{itemize}[leftmargin=\parindent]

\item If $G\in \LL_n$ and $0 < k < n$, let $\alpha \colon [k] \rat [n]$ be the inert inclusion into the first part of the interval ($\alpha(t) = t$), $\beta \colon [n-k] \rat [n]$ the inert inclusion into the last part of the interval ($\beta(t) = t + k$), and $k\colon [0] \rat [n]$ pick out $k$.
Then \[
  X_G \to X_{\alpha^*(G)} \times_{X_{k^*(G)}} X_{\beta^*(G)}
\]
is an isomorphism.
At the level of presheaves, we have a commutative square
\[ \begin{tikzcd}
\segcore(\alpha^*(G)) \amalg_{\segcore(k^*(G))} \segcore(\beta^*(G)) \rar \dar{\cong} & \yon(\alpha^*(G)) \amalg_{\yon(k^*(G))} \yon(\beta^*(G))\dar \\
\segcore(G) \rar & \yon(G)
\end{tikzcd} \]
whose left edge is an isomorphism.
The claimed statement follows by applying $\hom(-,X)$ (which transforms pushouts to pullbacks) and using 2-of-3 for isomorphisms.

\item 
We can iterate the previous observation.
Suppose $G\in \LL$ is a height $n$ level graph and for $1\leq i \leq n$ that $\rho_i \colon [1] \rat [n]$ in $\simpcat$ is the inert map picking out $i-1,i$, and for $0\leq i \leq n$ the map $\kappa_i \colon [0] \rat [n]$ picks out $i$.
If $\mathfrak{c}_{p,q} \rat G$ is an inert map in $\LL$, then it factors uniquely through $\rho_i^*G$ (which has height 1) for some $i$.
Likewise, any $\mathfrak{e} \rat G$ factors uniquely through some $\kappa_i^*G$ (which has height 0). 
Further,
\[
  X_G \to X_{\rho_1^*G} \times_{X_{\kappa_1^*G}} X_{\rho_2^*G} \times_{X_{\kappa_2^*G}} \dots \times_{X_{\kappa_{n-1}^*G}} X_{\rho_n^*G}
\]
is a bijection. See \cite[Proposition 3.2.9]{ChuHackney}.

\item
Suppose $G, H \in \LL_n$ are two height $n$ level graphs.
We can define a new height $n$ level graph $G+H$ with $(G+H)_{ij} = G_{ij} + H_{ij}$, which is a coproduct in the fiber $\LL_n$ (see \cref{subsec monoidal} where this will be used more).
Then $\LL^\elrm_{/(G+H)} = \LL^\elrm_{/G} \amalg \LL^\elrm_{/H}$, so it follows that the left hand map in the following square is an isomorphism:
\[ \begin{tikzcd}
\segcore(G) \amalg \segcore(H) \dar{\cong} \rar & \yon(G) \amalg \yon(H) \dar \\
\segcore(G+H) \rar & \yon(G+H).
\end{tikzcd} \]
Applying $\hom(-,X)$ and using 2-of-3 for isomorphisms, we see that $X_{G+H} \to X_G \times X_H$ is a bijection:
\[ \begin{tikzcd}
X_{G+H} \dar{\cong} \rar & X_G \times X_H \dar{\cong} \\
\lim\limits_{E\in (\LL^\elrm_{/(G+H)})^\oprm} X_E  \rar{\cong} & \left(\lim\limits_{E\in (\LL^\elrm_{/G})^\oprm} X_E\right) \times \left(\lim\limits_{E\in (\LL^\elrm_{/H})^\oprm} X_E\right).
\end{tikzcd} \]
\end{itemize}
\end{remark}

We also have the full subcategory $\LLc \subset \LL$ whose objects are the connected level graphs (a level graph $G$ of height $m$ is connected just when $G_{0m}$ is a point).
This subcategory inherits an active-inert factorization system from the larger category $\LL$.
The elementary objects in $\LLc$ are precisely the elementary objects of $\LL$ and there is a corresponding notion of Segal $\LLc$-presheaf.
Finally, there is the graphical category defined in \cite{HRYbook} whose objects are certain connected  graphs (see \cite[Definition 6.46]{HRYbook} or \cite[Definition 2.2.11]{ChuHackney}).
Following \cite{ChuHackney}, we call this category $\pgcat$, and it has its own notion of Segal presheaf.
We will not need any low-level details about this category or its algebraic pattern structure.

There is a zig-zag of functors over $\finsetstarop$
\[ \begin{tikzcd}
\LL \ar[dr] & \LLc \rar{\tau} \lar[swap]{\iota} \dar & \pgcat \ar[dl]\\
& \finsetstarop
\end{tikzcd} \]
where the downward functors take sets of vertices.

\begin{proposition}\label{prop equiv of segal}
The functors $\iota^*\colon \pre(\LL) \to \pre(\LLc)$ and $\tau^* \colon \pre(\pgcat) \to \pre(\LLc)$ restrict to functors $\seg(\LL) \to \seg(\LLc) \leftarrow \seg(\pgcat)$ which are equivalences of categories.
\end{proposition}

One can prove that $\iota^*$ is an equivalence on Segal presheaves by imitating the proof of Proposition 3.2.29 of \cite{ChuHackney}, which is an $\infty$-categorical analogue.
Similarly, one can imitate the proof of Theorem 5.1.4 of \cite{ChuHackney} to show that $\tau^* \colon \seg(\pgcat) \to \seg(\LLc)$ is an equivalence.
However, the latter proof is very involved.
Instead, we deduce both of these results by recognizing that $\seg(\XX)$ is equivalent to the subcategory of \emph{discrete objects} of the $\infty$-categorical version $\seg^\infty(\XX) \subseteq \pre^\infty(\XX) \coloneqq \fun(\XX^\oprm, \igpd)$ when $\XX \in \{ \LL, \LLc, \pgcat \}$.
Recall from \cite[5.5.6.2]{LurieHTT} that an object $x$ in a quasi-category is called \emph{discrete} if, for every object $y$, the Kan complex $\map(y,x)$ is equivalent to a set.
We utilize the more general notion of algebraic pattern from \cite[Definition 2.1]{ChuHaugseng}; in that source $\seg^\infty(\XX)$ would be denoted by $\seg_{\XX^\oprm}(\igpd)$ and $\seg(\XX)$ would be denoted by $\seg_{\XX^\oprm}(\set)$.

\begin{lemma}\label{lemma:discreterestriction}
    Let $\XX$ be a category equipped with an active-inert factorization system and a class of elementary objects so that $\XX^\oprm$, together with this structure, is an algebraic pattern. 
    There is a fully faithful functor $\seg(\XX)\hookrightarrow \seg^\infty(\XX)$ whose essential image is the full subcategory of discrete objects.
\end{lemma}
\begin{proof}
Let $\set \hookrightarrow \igpd$ be the fully-faithful inclusion of the $\infty$-category of sets into that of spaces.
As postcomposition with a fully faithful functor is fully faithful (see \cref{rmk postcomp} below), there is a fully faithful composite $\seg(\XX)\hookrightarrow\pre(\XX)\hookrightarrow\pre^\infty(\XX).$
This functor factors through $\seg^\infty(\XX)$ since $\set \hookrightarrow \igpd$ preserves limits \cite[5.5.6.18]{LurieHTT}.
It is automatic that $\seg(\XX) \hookrightarrow \seg^\infty(\XX)$ is fully faithful.
The essential image of the functor consists of those Segal objects $F \colon \XX^\oprm \to \igpd$ so that $F(x)$ is discrete for all $x\in \XX$.
We now show this is the same thing as $F\in \seg^\infty(\XX)$ being a discrete object.

Suppose $G \colon \XX^\oprm \to \igpd$ is an arbitrary presheaf, and write $G \simeq \colim_\alpha \yon(x_\alpha)$, where $\yon \colon \XX \to \pre^\infty(\XX)$ is the Yoneda embedding.
Then there are equivalences
\[
  \lim_\alpha F(x_\alpha) \simeq \lim_\alpha \map(\yon(x_\alpha), F) \simeq \map(G,F) \simeq \map(LG,F)
\]
where $L\colon \pre^\infty(\XX) \to \seg^\infty(\XX)$ is the localization functor guaranteed by \cite[Lemma 2.11(iii)]{ChuHaugseng}.
If $F(x)$ is discrete for all $x \in \XX$, then since $\set$ is closed under limits, we see that $\map(G,F)$ is discrete hence $F \in \seg^\infty(\XX)$ is discrete.
Conversely, if $F$ is a discrete object of $\seg^\infty(\XX)$ and $x\in \XX$ is arbitrary, taking $G = \yon(x)$ in the string of equivalences above implies that $F(x)$ is discrete.
We have thus identified the essential image of $\seg(\XX) \hookrightarrow \seg^\infty(\XX)$ with the discrete objects of $\seg^\infty(\XX)$.
\end{proof}

\begin{remark}\label{rmk postcomp}
It is well known that if $\mathcal{C} \to \mathcal{D}$ is a fully faithful functor of quasi-categories, then the postcomposition functor $\fun(\mathcal{B},\mathcal{C}) \to \fun(\mathcal{B},\mathcal{D})$ is also fully faithful.
One way to verify involves first observing that \cite[5.6]{RiehlVerity:construction} and \cite[3.5.6(iv)]{RiehlVerity:EOICT} together imply that being fully faithful as a functor of quasi-categories is equivalent to being fully faithful as a functor in the $\infty$-cosmos of quasi-categories.
Since this $\infty$-cosmos is cartesian closed, the assertion follows from the characterization given in \cite[3.5.6(iii)]{RiehlVerity:EOICT}.
\end{remark}

\begin{proof}[Proof of \cref{prop equiv of segal}]
As a result of \cite[5.5.6.28]{LurieHTT}, an equivalence of quasi-categories preserves and reflects discrete objects.
Therefore it restricts to an equivalence of categories between the full subcategories of its domain and codomain. 
Proposition 3.2.29 and Theorem 5.1.4 of \cite{ChuHackney} show that $\iota$ and $\tau$ induce equivalences $\seg^\infty(\LL)\to\seg^\infty(\LLc)\leftarrow\seg^\infty(\pgcat)$. 
The result now follows from \cref{lemma:discreterestriction}.
\end{proof}

Let $\chi_3 \colon \pgcat \to \properad$ be the functor which takes a graph to a properad freely generated by it; see \cite[\S5.1.2]{HRYbook}. 
We let $\chi_2 = \chi_3 \circ \tau$, and now describe $\chi_1$ appearing in the following diagram.
\[ \begin{tikzcd}
\LL \ar[dr,"\chi_1"'] & \LLc \rar{\tau} \lar[swap]{\iota} \dar{\chi_2} & \pgcat \ar[dl, "\chi_3"]\\
& \properad
\end{tikzcd} \]
If $G$ is a height $m$ level graph, we can decompose $G$ into connected components $G \cong \coprod_{x \in G_{0m}} G_x$ (in \cref{def canonical splitting} we will fix a particular such isomorphism). 
The functor $\chi_1 \colon \LL \to \properad$ is defined by sending $G$ to $\coprod \chi_2(G_x) = \coprod \chi_3 \tau (G_x)$.\footnote{Let us describe $\chi_1$ applied to a morphism $f \colon G \to H \cong \coprod_{y\in H_{0n}} H_y$ lying over $\alpha \colon [m] \to [n]$ in $\simpcat$. For $x\in G_{0m}$ write $\bar{x}$ for the image of $x$ under the composite $G_{0m} \to H_{\alpha(0)\alpha(m)} \to H_{0n}$.
Then the composite $G_x \to G \to H$ factors through a map between connected level graphs $f_x \colon G_x \to H_{\bar x}$, and $\chi_1(f) \colon \coprod \chi_2(G_x) \to \coprod \chi_2(H_y)$ is induced from the properad maps $\chi_2(f_x)$.}


We obtain from these functors the following commutative diagram:
\[ \begin{tikzcd}
\pre(\LL) \ar[from=dr,"N_1"] \rar{\iota^*} & \pre(\LLc) \ar[from=d, "N_2"'] & \pre(\pgcat) \ar[from=dl, "N_3"'] \lar[swap]{\tau^*} \\
& \properad
\end{tikzcd} \]
with $N_i(P)_G = \hom(\chi_i(G),P)$; this commutes since if $G$ is a connected level graph, then
\[
\iota^*N_1(P)_G = N_1(P)_{\iota G} = \hom(\chi_1(\iota G), P) = \hom( \chi_2(G),P) = N_2(P)_G 
\]
and likewise for the other triangle.
By \cite[Lemma 7.38]{HRYbook}, $N_3$ takes values in the subcategory of Segal objects.

Since we know that $\tau^*$ takes Segal objects to Segal objects and $N_2 = \tau^* \circ N_3$, we have that $N_2$ lands in the subcategory of Segal objects. 
This same argument does not apply to $N_1$, so we must check the following:

\begin{lemma}\label{lem N1 segal}
If $P$ is a properad, then $N_1(P)$ is Segal.
\end{lemma}
\begin{proof}
Let $G\in \LL$ be a height $m$ level graph, and for $x\in G_{0m}$ let $G_x \in \LLc$ be the corresponding connected level graph with $\coprod_x G_x = G$ in the fiber $\LL_m$.
Notice that the canonical map $N_1(P)_G \to \prod_x N_1(P)_{G_x}$ is an isomorphism, using our description of $\chi_1$:
\[
  N_1(P)_G = \hom(\chi_1(G),P) = \hom\left(\coprod_x \chi_2(G_x), P\right) 
  = \prod_x N_1(P)_{G_x}.
\]
The map $\coprod_{x} \segcore(G_x) \to \segcore(G)$ is an isomorphism of presheaves, where $\segcore(G)$ is the Segal core of $G$.
In the commutative square
\[ \begin{tikzcd}
\coprod\limits_x \segcore(G_x) \rar{=} \dar & \segcore(G) \dar \\
\coprod\limits_x \yon(G_x) \rar & \yon(G) 
\end{tikzcd} \]
we know that the bottom map becomes an isomorphism after applying $\hom(-,N_1P)$.
Hence to see that $N_1(P)$ is Segal it is enough to check that $\hom(\yon(H), N_1P) \to \hom(\segcore H, N_1P)$ is an isomorphism for connected graphs $H$.
But this is true: if $H \in \LLc$ is connected, then $\iota_!\yon(H) = \yon(\iota H)$ and $\iota_!\segcore(H) = \segcore(\iota H)$, and by Segality of $N_2(P) = \iota^*N_1(P)$ the bottom map in the following square
\[ \begin{tikzcd}
\hom(\iota_! \yon H, N_1 P)  \rar \dar{=} & \hom(\iota_! \segcore H, N_1 P) \dar{=} \\
\hom(\yon H, \iota^*N_1 P) \rar & \hom(\segcore H, \iota^*N_1 P)
\end{tikzcd} \] 
is an isomorphism.
\end{proof}

\begin{proposition}\label{prop properads as presheaves}
The functors $N_1$ and $N_2$ induces equivalences of categories
\begin{align*}
N_1 \colon \properad &\simeq \seg(\LL) \\
N_2 \colon \properad &\simeq \seg(\LLc).
\end{align*}
\end{proposition}
\begin{proof}
Each $N_i$ lands in the full subcategory of Segal objects (using \cref{lem N1 segal} for $N_1$), hence we have the commutative diagram
\[ \begin{tikzcd}
\seg(\LL) \ar[from=dr,"N_1"] \rar["\iota^*"] & \seg(\LLc) \ar[from=d, "N_2"'] & \seg(\pgcat) \ar[from=dl, "N_3"'] \lar[swap]{\tau^*} \\
& \properad.
\end{tikzcd} \]
By \cref{prop equiv of segal}, $\iota^*$ and $\tau^*$ are equivalences, and $N_3$ is an equivalence by \cite{HRYbook}.
By 2-of-3, $N_1$ and $N_2$ are equivalences as well.
\end{proof}

We will have no further need for $\LLc$ or $\pgcat$ in this paper.

\section{The symmetric monoidal envelope of a properad}\label{properad to categories}

The envelope of a properad is the prop freely generated by it.
In this section we give a detailed description of this structure, in a way that will make our later comparisons more transparent.
We take as an input a Segal $\LL$-presheaf, relying on the equivalence of categories from \cref{prop properads as presheaves}.

\begin{remark}
In \cite[\S4.A]{HackneyRobertsonYau:SMIP}, a left adjoint to the forgetful functor from props to properads is exhibited.
We caution the reader that this is not quite the same as what we are doing here, since the kind of prop used in that paper is slightly weaker than the original version \cite{MacLane:CA,HackneyRobertson:OCP}.
See \cite[Remark 10.5]{BataninBerger:HTAPM} and \cite[Remark 3.5]{HackneyRobertson:HTSP} for details.
\end{remark}

\subsection{Congruences of level graphs}

\begin{definition}
A \emph{congruence} in $\LL_n$ is an isomorphism $G \to H$ so that for $0 \leq i \leq n$, the map $G_{ii} \to H_{ii}$ is an identity.
We denote such a congruence by
\[
  G \xrightarrow{\equiv} H.
\]
We say that $G$ and $H$ are \emph{congruent}, denoted $G\equiv H$, if there is a congruence between them.
\end{definition}

If $G\in \LL_0$, then the only congruence involving $G$ is the identity.
The same is true for the corollas $\mathfrak{c}_{n,m} \in \LL_1$.

\begin{lemma}
If $\alpha \colon [m] \to [n]$ is a map in $\simpcat$ and $G \to H$ is a congruence in $\LL_n$, then $\alpha^*G \to \alpha^*H$ is a congruence in $\LL_m$.
\end{lemma}
\begin{proof}
If $0 \leq i \leq j \leq m$, then $(\alpha^*G)_{ij} \to (\alpha^*H)_{ij}$ is equal to $G_{\alpha i, \alpha j} \to H_{\alpha i, \alpha j}$. 
When $i=j$ this map is an identity by definition of congruence.
\end{proof}

\begin{remark}\label{rmk congruence height 1}
Suppose $G\in \LL_1$ is a graph where every vertex has an input or output.
Then the only congruence $G \xrightarrow{\equiv} G$ is the identity, as the hypothesis tells us the horizontal morphisms in the following diagram are epimorphisms
\[ \begin{tikzcd}
G_{00} + G_{11} \rar[two heads] \dar{=} & G_{01} \dar{\cong} \\
G_{00} + G_{11} \rar[two heads] & G_{01} 
\end{tikzcd} \]
which implies that $G_{01} \to G_{01}$ is the identity as well.
This is not to say that $G$ is not involved in any congruence at all, only that if one permutes the elements of $G_{01}$ then one must also alter at least one of the functions $G_{00}\to G_{01}$ and $G_{11}\to G_{01}$. Indeed, if $G_{01} = \un$, then there are $n!$ different congruences with domain $G$.
\end{remark}

\begin{definition}[Congruence category]
Let $\kong_n \subseteq \LL_n$ denote the wide subcategory consisting of the congruences.
This defines a functor $\kong_\bullet \colon \simpcat^\oprm \to \gpd \subseteq \cat$, and we write $\kong \to \simpcat$ for the associated (right) fibration. 
\end{definition}

It will be useful to collapse further:

\begin{definition}
Let $\WW \to \simpcat$ be the discrete fibration associated to the composite
\[ \begin{tikzcd}
\simpcat^\oprm \rar{\kong_\bullet} & \gpd \rar{\pi_0} & \set.
\end{tikzcd} \]
That is, $\WW_n = \pi_0(\kong_n)$ is the set of height $n$ level graphs modulo congruence. 
Note that $\WW_n$ is a discrete category, but $\WW$ itself is not.
\end{definition}

\begin{remark}\label{Remark WW as nerve Csp}
The simplicial set $\WW_\bullet$ turns out to be isomorphic to the nerve of a skeleton of $\csp$.
Elements of $\WW_0$ can be identified with the set $\{ \uk \}$ for $k\in \mathbb{N}$.
Elements of $\WW_1$ are congruence classes of height 1 level graphs, and this relation means these are exactly the same thing as morphisms in $\csp$ between the sets appearing in $\{ \uk \}_{k\in \mathbb{N}}$.
In \cref{segality W} below we will formally show that $\WW_\bullet$ is Segal, and by inspection one can see that the compositions coincide in $\WW_\bullet$ and $N\csp$. 
See \cref{example C star Csp} and \cref{example C star monoidal} for more details.
\end{remark}

The zig-zags $\LL_n \hookleftarrow \kong_n \twoheadrightarrow \pi_0(\kong_n) = \WW_n$ as $n$ varies constitute natural transformations
\[ \begin{tikzcd}[column sep=large]
\simpcat^\oprm 
\rar["\kong_\bullet" description] \rar[bend left=50,"\LL_\bullet"] \rar[bend right=50, "\WW_\bullet"'] 
\rar[bend left=25, phantom,"\Uparrow"] \rar[bend right=25, phantom, "\Downarrow"']
&
\cat
\end{tikzcd} \]
(considering $\set \subseteq \gpd \subseteq \cat$).
This yields the following commutative diagram of fibrations over $\simpcat$.
\[ \begin{tikzcd}
\WW \ar[dr,"q"'] & \kong \lar[swap]{\pi} \rar{i} \dar &  \LL \ar[dl,"p"] \\
& \simpcat
\end{tikzcd} \]

We saw above that $\LL_\bullet \colon \simpcat^\oprm \to \cat$ is Segal, meaning that the Segal morphisms $\LL_n \to \LL_1 \times_{\LL_0} \dots \times_{\LL_0} \LL_1$ are all equivalences of categories.
The same holds for $\kong_\bullet$. 

\begin{proposition}\label{segality K}
The simplicial category $\kong_\bullet$ is Segal.
Moreover, the Segal morphisms are surjective on objects.
\end{proposition}
\begin{proof}
Let $n > 1$, $1 < k < n$, and consider the pushout square in $\simpcat_\intrm$ whose vertical maps preserve last elements and horizontal maps preserve first elements.
\[ \begin{tikzcd}
{[0]} \ar[d,"k",tail] \ar[r,"0",tail]  \ar[dr, phantom, "\ulcorner" very near end] & {[n-k]} \ar[d,"\beta", tail] \\
{[k]} \ar[r,"\alpha", tail] &{[n]}
\end{tikzcd} \]
We must show
\[
\alpha^* \times \beta^* \colon \kong_n \to \kong_{k} \times_{\kong_{0}} \kong_{n-k}
\]
is an equivalence of categories.
We already know that the map is surjective on objects since $\LL_n \to \LL_{k} \times_{\LL_{0}} \LL_{n-k}$ is.
Since $\LL_n \to \LL_{k} \times_{\LL_{0}} \LL_{n-k}$ is fully faithful, it is enough to show that if $G, H \in \kong_n$ are height $n$ level graphs and $\alpha^*G \to \alpha^*H$ and $\beta^*G \to \beta^*H$ are congruences, then they define a congruence $G\to H$.
But this is clear, since
\[
  (G_{ii} \to H_{ii}) = 
  \begin{cases}
  (\alpha^*G)_{ii} \to (\alpha^*H)_{ii} & \text{if } 0 \leq i \leq k \\
  (\beta^*G)_{i-k,i-k} \to (\beta^*H)_{i-k,i-k} & \text{if } k \leq i \leq n
  \end{cases}
\]
which are identities.
\end{proof}

Though $\pi_0 \colon \gpd \to \set$ does not preserve pullbacks in general, we still have the following result.

\begin{corollary}\label{segality W}
The simplicial set $\WW_\bullet$ is Segal.
\end{corollary}
\begin{proof}
We use the same notation as in the previous proof.
We are to show that the bottom map in the following square is a bijection.
\[ \begin{tikzcd}
\kong_n \rar["\alpha^* \times \beta^*"',"\simeq"] \dar & \kong_{k} \times_{\kong_{0}} \kong_{n-k} \dar \\
\pi_0(\kong_n) \rar["\alpha^* \times \beta^*"] & \pi_0(\kong_{k}) \times_{\pi_0(\kong_{0})} \pi_0(\kong_{n-k})
\end{tikzcd} \]
The bottom map is surjective, since the top-right composite is surjective on objects.
Suppose $w,v\in \WW_n$ map to the same element in $\WW_{k} \times_{\WW_{0}} \WW_{n-k}$. 
Let $G,H\in \kong_n$ be graphs in these path components. 
The assumption is that $\alpha^*G \equiv \alpha^*H$ and $\beta^*G \equiv \beta^*H$.
Choose congruences witnessing these relations, giving a morphism $(\alpha^*G, \beta^*G) \to (\alpha^*H, \beta^*H)$ in $\kong_{k} \times_{\kong_{0}} \kong_{n-k}$. Since the top map in the square above is an equivalence, this comes from a congruence in $\kong_n$.
Hence $G \equiv H$, that is, $w=v$.
\end{proof}

\subsection{The category \texorpdfstring{$CX$}{CX}}
Our next task is to describe the envelope of a properad.
For now, we content ourselves in producing this as a category, rather than as a symmetric monoidal category.
More specifically, we produce a functor $S$ as in the following diagram by using restriction along $i$ and left Kan extension along $q\pi = pi$,
\[ \begin{tikzcd}
\seg(\LL) \dar[hook] \ar[rr,dashed] &  &\seg(\simpcat) \dar[hook] & \cat \lar["N","\simeq"'] \\
\pre(\LL) \rar{i^*} \ar[rr,bend right=20, "S"] &  \pre(\kong) \rar{q_!\pi_!} & \pre(\simpcat)  
\end{tikzcd} \]
and show that it takes Segal presheaves to Segal presheaves. We write $C \colon \seg(\LL) \to \cat$ for the composite of $S$ with the left adjoint of the nerve functor $N$.

\begin{lemma}
Given a commutative triangle of small categories
\[ \begin{tikzcd}[column sep=small]
E \ar[rr,"p"] \ar[dr,"q"'] & & B \ar[dl,"r"] \\
& A 
\end{tikzcd} \]
If $q$ is a fibration and $r$ is a discrete fibration, then $p$ is a fibration.
\end{lemma}
\begin{proof}
If $b\in B$ is any object, then $B_{/b} \to A_{/rb}$ is an isomorphism. Indeed, it is bijective on objects since any $a \to rb$ has a unique lift $\tilde a \to b$. It is also fully faithful since a string $a_0 \to a_1 \to rb$ has a unique lift $\tilde a_0 \to \tilde a_1 \to b$ in $B$; but the composite $\tilde a_0 \to b$ is the unique lift of $a_0 \to rb$.

To show that $p$ is a fibration it is enough to exhibit a right adjoint right inverse to $E_{/e} \to B_{/pe}$ for every object $e$ (see \cite[Theorem 2.10]{Gray:FCC} or \cite[Theorem 2.2.2]{LoregianRiehl:CNF}).
But since $q$ is a fibration, for each $e$ the functor $E_{/e} \to A_{/qe}$ admits a right adjoint right inverse.
So choose one and this defines the desired functor $B_{/pe} \cong A_{/rpe} \to E_{/e}$.
\end{proof}

Since $q\pi \colon \kong \to \simpcat$ is a fibration and $q$ is a discrete fibration, we have:

\begin{lemma}
The functor $\pi \colon \kong \to \WW$ is a Grothendieck fibration.
\qed
\end{lemma}

It is classical that left Kan extension along an opfibration may be computed by taking the colimit over the fibers (see \cite{nlabKan}), which gives the following.

\begin{lemma}\label{lem pushforwards}
If $A$ is a $\kong$-presheaf, then its left Kan extension $\pi_!A$ is given by
\[
(\pi_!A)_w = \colim_{G \in \kong_w^\oprm} A_G.    
\] 
where $\kong_w \subseteq \kong$ is the relevant connected component of $\kong_{q(w)}$.
Likewise, if $B$ is a $\WW$-presheaf, then \[ (q_!B)_n = \sum_{w\in\WW_n} B_w.\] 
\end{lemma}

\begin{definition}\label{def category SX}
Let $S$ denote the composition
\[ \begin{tikzcd}
\pre(\LL) \rar{i^*} & \pre(\kong) \rar{q_!\pi_!} & \pre(\simpcat).
\end{tikzcd} \]
That is, if $X\in \pre(\LL)$, define $SX \in \pre(\simpcat)$ as
\[
  SX = (q\pi)_!i^*X.
\]
\end{definition}
In light of \cref{lem pushforwards}, this means that
\begin{equation}\label{EQ SX levels}
  SX_n = \sum_{w\in \WW_n}  \colim_{G \in \kong_w^\oprm} X_G.
\end{equation}
We write $\overline{X}_G$ for $\colim_{G \in \kong_w^\oprm} X_G$, which is a quotient of $X_G$ by the action of self-congruences of $G$. 
We occasionally write $\overline{X}_w$ for this set, where $w = [G] = \pi(G)\in \WW$.

\begin{remark}\label{reduced graphs remark}
If every vertex of $G \in \LL_1$ has an input or output, then the canonical map $X_G \to \overline{X}_G = \colim_{G \in \kong_w^\oprm} X_G$ is an isomorphism by \cref{rmk congruence height 1}.
\end{remark}

\begin{proposition}\label{prop SX segal}
If $X \in \seg(\LL)$, then $SX \in \seg(\simpcat)$.
\end{proposition}
\begin{proof}
Suppose we have a pushout square of non-identity maps in $\simpcat_\intrm$ whose vertical maps preserve last elements and horizontal maps preserve first elements.
\[ \begin{tikzcd}
{[0]} \ar[d,"k",tail] \ar[r,"0",tail]  \ar[dr, phantom, "\ulcorner" very near end] & {[\ell]} \ar[d,"\beta", tail] \\
{[k]} \ar[r,"\alpha", tail] &{[n]}
\end{tikzcd} \]
It is enough to show that $SX$ sends pushouts of this form to pullbacks.

We use the description of $SX_n$ from \eqref{EQ SX levels} above.
Suppose we have two elements $s,s'$ of $SX_n$, represented by $(w,G,x)$ and $(w',G',x')$, which are sent to the same element of $SX_k \times_{SX_0} SX_\ell$
(Here $G$ is a graph in the path component $w$ of $\kong_n$ and $x$ is an element of $X_G$ with $s$ the image of $x$ under $X_G \to \colim X_H \subseteq SX_n$).
Then $w=w'$ by \cref{segality W}.
Write $\alpha^*x \in X_{\alpha^*G}$, $\beta^*x \in X_{\beta^*G}$, $\alpha^*x' \in X_{\alpha^*G'}$, and $\beta^*x' \in X_{\beta^*G'}$ for the images of $x$ and $x'$.
Since $\alpha^*x$ and $\alpha^*x'$ represent the same element in the colimit
\[
  \colim_{H \in \kong_{\alpha^* w}^\oprm} X_H
\]
there is a congruence $\gamma \colon \alpha^*G \to \alpha^*G'$ with $\gamma^*(\alpha^*x') = \alpha^*x$.
Likewise, there is a congruence $\delta \colon \beta^*G \to \beta^*G'$ with $\delta^*(\beta^*x') = \beta^*x$.
By \cref{segality K}, there is a unique congruence $\chi \colon G \to G'$ with $\alpha^*(\chi) = \gamma$ and $\beta^*(\chi) = \delta$.
This implies, in particular, that $G_{kk} = G_{kk}'$, and we write $\underline{m}$ for this common value.
The following diagram commutes and consists of bijections since $X$ is a Segal $\LL$-presheaf (see \cref{examples segality}).
\[ \begin{tikzcd}
X_{G'} \rar \dar["\chi^*"'] & X_{\alpha^*G'} \times_{X_{\underline{m}}} X_{\beta^*G'} \dar \\
X_G \rar & X_{\alpha^*G} \times_{X_{\underline{m}}} X_{\beta^*G}
\end{tikzcd} \]
It follows that $\chi^*x' = x$.
Since $x\in X_G$ and $x'\in X_{G'}$ represent the same element in the colimit $\colim_{H \in \kong_w^\oprm} X_H$, it follows that $s=s'$ in $SX_n$.
We have thus shown the Segal map is injective.

We now turn to surjectivity.
Suppose $(s,s') \in SX_k \times_{SX_0} SX_\ell$, with $s$ represented by $(w,G,x)$ and $s'$ represented by $(w',G',x')$.
We have $G_{kk} = G'_{00} = \underline{m}$.
Choose a height $n$ level graph $H$ with $\alpha^*H = G$ and $\beta^*H = G'$.
Since $X$ is a Segal $\LL$-presheaf, there is a (unique) element $y \in X_H$ which maps to $(x,x')$ under
\[ \begin{tikzcd}
\alpha^* \times \beta^* \colon X_H \xrightarrow{\cong} X_G \times_{X_{\underline{m}}} X_{G'}.
\end{tikzcd} \]
Now the element of $SX_n$ represented by $(\pi(H), H, y)$ is sent to $(s,s')$ by the Segal map.
\end{proof}

\begin{definition}
If $X \in \seg(\LL)$, we write $CX = \tau_1 (SX) \in \cat$ for the associated category (where $\tau_1 \colon \sset \to \cat$ is left adjoint to $N$). This defines a functor $C \colon \seg(\LL) \to \cat$ with $NCX = SX$.
\end{definition}

The following example has previously been discussed in \cref{Remark WW as nerve Csp}, but it is extremely important for the rest of the paper and there is no harm in revisiting it.
We will return to the monoidal structure in \cref{example C star monoidal} below.

\begin{example}[Canonical inclusion]\label{example C star Csp}
If $\ast \in \pre(\LL)$ is the terminal presheaf, then there is an injective-on-objects equivalence $C(\ast) \to \csp$.
From the formula following Definition \ref{def category SX}, we see $S(\ast)$ is isomorphic to the simplicial set $\WW_\bullet$.
Then the objects of $C(\ast)$ are identified with the height 0 level graphs, that is, with $\ob (\finsetskel) = \{ \underline{0}, \underline{1}, \underline{2}, \dots \}$.
The morphisms of $C(\ast)$ are the elements of $\WW_1$, which under the identification on objects are precisely the same as the morphisms between these objects in $\csp$.
Finally, compositions in both categories are given by pushouts of equivalence classes of cospans.
\end{example}

\subsection{Monoidal structure}\label{subsec monoidal}
In this section, we explain why $CX$ is a strict monoidal category.
It turns out to also be symmetric monoidal (see \cref{thm permutative}), but this is more difficult to show since it requires working at several simplicial levels of $SX$.
The main result of this section is that if $X$ is Segal, then $SX$ is a (strict) monoid in $\sset$.
This means (see \cref{prop strict monoidal cat}) that $CX$ is a monoid in $\cat$, that is, a strict monoidal category.

For each $n\geq 0$, the category $\LL_n \subseteq \fun(\mathscr{L}^n, \finsetskel)$ is monoidal, using the ordinal sum in $\finsetskel$. 
More specifically, if $G$ and $H$ are height $n$ level graphs, then $G+H$ is the height $n$ level graph with $(G+H)_{ij} = G_{ij} + H_{ij}$.
The maps $G_{ij} \hookrightarrow G_{ij} + H_{ij} \hookleftarrow H_{ij}$ are order-preserving, and every element in $G_{ij}$ appears before every element in $H_{ij}$.
In fact, $\LL_n$ is actually a permutative category since $\finsetskel$ is (i.e.~a symmetric strict monoidal category \cite[Definition 4.1]{May74}).
From \cref{examples segality} we know that a Segal presheaf $X$ takes sums of height $n$ level graphs to products of sets, which we use in the following.

\begin{definition}\label{def tensor before S}
Suppose $X \in \seg(\LL)$, $G$ and $H$ are height $n$ level graphs, and let
$i \colon G \to G+H$ and $j\colon H \to G+H$ be the inclusions. 
Then define
\[
  {-}\ten{-} \colon X_G \times X_H \to X_{G+H}
\]
to be the inverse of the bijection $i^* \times j^* \colon X_{G+H} \to  X_G \times X_H$.
\end{definition}

We will make use of the following lemma in \cref{sec symmetry}.
\begin{lemma}\label{lem symmetry}
Let $G,H \in \LL_n$ and let $\sigma \colon H+G \to G+H$ be the flip map.
If $x\in X_G$ and $y\in X_H$, then $\sigma^*(x\ten y) = y\ten x$.
\end{lemma}
\begin{proof}
The diagram
\[
\begin{tikzcd}
H \rar{i_L} \ar[dr,"i_R"'] & H + G \dar["\sigma" description] \\
& G + H & G \lar["i_L"] \ar[ul,"i_R"']
\end{tikzcd} 
\]
commutes in $\LL_n$, which implies the diagram
\[ \begin{tikzcd}
X_G \times X_H \dar{\tau} & X_{G + H} \lar["\cong" swap,"i_L^*\times i_R^*"]  \dar{\sigma^*}  \\
X_H \times X_G & X_{H + G} \lar["\cong" swap,"i_L^*\times i_R^*"] 
\end{tikzcd} \]
commutes, where $\tau$ is the flip map. The result follows.
\end{proof}

\begin{lemma}\label{lem assoc unit before S}
Let $X\in \seg(\LL)$ and suppose that $G,H,K\in \LL_n$.
If $x\in X_G$, $y\in X_H$, and $z\in X_K$, then \[(x\ten y) \ten z = x \ten (y\ten z)\] in $X_{(G+H)+K} = X_{G+(H+K)}$.
If $\varnothing_n$ is the empty height $n$ level graph and $\ast$ is the unique element of $X_{\varnothing_n}$, then 
\[
\ast \ten x = x = x\ten \ast
\]
in $X_{\varnothing_n + G} = X_G = X_{G+ \varnothing_n}$.
\end{lemma}
\begin{proof}
Associativity follows from commutativity of the following rectangle.
\[ \begin{tikzcd}
(X_G \times X_H) \times X_K \dar & X_{G+H} \times X_K \lar["\cong" swap]  & X_{(G+H) + K} \lar["\cong" swap] \dar{=} \\
X_G \times (X_H \times X_K) & X_G \times X_{H+K} \lar["\cong" swap] & X_{G+ (H + K)} \lar["\cong" swap]
\end{tikzcd} \]
This uses that the monoidal structure on $\LL_n$ is strict, so $(G+H)+K = G+(H+K)$.

For the second statement, the Segal map at $\varnothing_n$
\[
  X_{\varnothing_n} \xrightarrow{\cong} \lim_{K \in (\LL_\elrm^\oprm)_{\varnothing_n/}} X_K
\]
identifies $X_{\varnothing_n}$ with a limit over the empty category.
Hence there is indeed a unique element $\ast$ in $X_{\varnothing_n}$.
The Segal maps are
\[ \begin{tikzcd}[row sep=0]
X_G = X_{\varnothing_n + G} \rar & X_{\varnothing_n} \times X_G & X_G = X_{ G + \varnothing_n} \rar &  X_G \times X_{\varnothing_n} \\
x \rar[mapsto] & (\ast, x) & x \rar[mapsto] & (x,\ast)
\end{tikzcd} \]
which establishes the result.
\end{proof}

\begin{lemma}\label{lem ten descends}
Suppose $X\in \seg(\LL)$ and $n\geq 0$.
The collection of functions $\ten$ from \cref{def tensor before S} descend to a function
\[
  {-}\ten{-} \colon SX_n \times SX_n \to SX_n.
\]
\end{lemma}
\begin{proof}
Recall the description of $SX_n$ from \eqref{EQ SX levels} on page~\pageref{EQ SX levels}.
Let $\bar{x} \in \overline{X}_w \subseteq SX_n$ and $\bar{x}' \in \overline{X}_{w'} \subseteq SX_n$.
Choose a representative $x\in X_G$ and $x'\in X_{G'}$ for these elements and define $\bar{x} \ten \bar{x}' \in SX_n$ to be the image of $(x,x')$ under the following composite:
\[ \begin{tikzcd}
X_G \times X_{G'} & X_{G + G'} \lar["\cong" swap] \rar & \overline{X}_{G+G'} \rar[hook] & SX_n.
\end{tikzcd} \]
The leftward map is a bijection since $X$ is Segal (\cref{examples segality}).
This definition does not depend on our choice of representatives: indeed, suppose we have $(y,y') \in X_H \times X_{H'}$ along with congruences $f \colon H \to G$ and $g \colon H' \to G'$ so that $f^*(x) = y$ and $g^*(x') = y'$. We then have a commutative diagram
\[ \begin{tikzcd}
X_G \times X_{G'} \dar{f^* \times g^*} & X_{G + G'} \lar["\cong" swap] \rar \dar{(f + g)^*}  & \overline{X}_{G+G'} \rar[hook] \dar{=} & SX_n \dar{=} \\
X_H \times X_{H'} & X_{H + H'} \lar["\cong" swap] \rar & \overline{X}_{H+H'} \rar[hook] & SX_n
\end{tikzcd} \]
which establishes that $\bar x \ten \bar x'$ does not depend on choice of representative.
\end{proof}

\begin{proposition}\label{prop strict monoidal cat}
If $X \in \seg(\LL)$, then $CX$ is a strict monoidal category. 
\end{proposition}
\begin{proof}
The assignment $\ten$ on objects and morphisms from \cref{lem ten descends} of $CX$ is a functor $CX\times CX \to CX$.
This follows because the simplicial structure maps $\alpha^* \colon \LL_m \to \LL_n$ preserve $+$.
The tensor unit is the image of the unique element of $X_{\varnothing_0}$ in $SX_0 = \ob CX$, and the tensor is strictly associative and unital by \cref{lem assoc unit before S}.
\end{proof}

Essentially by definition of the monoidal structure, if $X \to Y$ is a map in $\seg(\LL)$ then $CX \to CY$ is a strict monoidal functor.

\subsection{The symmetry}\label{sec symmetry}
In order to most efficiently produce the symmetry isomorphism and prove its properties, we first produce a more general twisted tensor of any two morphisms in $CX$, whose codomain is the opposite from the normal tensor. 
That is, if $f\colon x \to x'$ and $g\colon y \to y'$ are two morphisms, then we construct $f\twten g \colon x\ten y \to y' \ten x'$.
The symmetry isomorphism is then given in \cref{def symmetry iso} as $\id_x \twten \id_y \colon x\ten y \to y \ten x$.

As an auxiliary useful construction, we first define twisted sums of height $n$ graphs.
Though these are defined in general, we only need them for $n\leq 2$ which can be illustrated directly.
For $G,H \in \LL_n$ height $n$ level graphs and an integer $t$, define a new graph $G \twplus{t} H \in \LL_n$ by
\[
  (G \twplus{t} H)_{ij} = \begin{cases}
    G_{ij} + H_{ij} & i < t \\
    H_{ij} + G_{ij} & i \geq t.
  \end{cases}
\]
The structure maps are constructed as sums and flips of sums of those appearing in $G$ and $H$. 

\begin{example}
Special cases include $G \twplus0 H = H+G$ and $G \twplus{n+1} H = G + H$.
This covers all cases when $n=0$; when $n=1$ there is one interesting case $G \twplus1 H$
\[ \begin{tikzcd}[row sep=small, column sep=tiny]
G_{00} + H_{00} \ar[dr] & & H_{11} + G_{11} \ar[dl] \\
& G_{01} + H_{01}
\end{tikzcd} \]
while when $n=2$ we have $G \twplus1 H$
\[ \begin{tikzcd}[row sep=small, column sep=tiny]
G_{00} + H_{00} \ar[dr] & & H_{11} + G_{11} \ar[dl] \ar[dr] & &  H_{22} + G_{22} \ar[dl] \\
& G_{01} + H_{01} \ar[dr] & & H_{12} + G_{12} \ar[dl]\\
& & G_{02} + H_{02}
\end{tikzcd} \]
and $G \twplus2 H$
\[ \begin{tikzcd}[row sep=small, column sep=tiny]
G_{00} + H_{00} \ar[dr] & & G_{11} + H_{11} \ar[dl] \ar[dr] & &  H_{22} + G_{22} \ar[dl] \\
& G_{01} + H_{01} \ar[dr] & & G_{12} + H_{12} \ar[dl]\\
& & G_{02} + H_{02}
\end{tikzcd} \]
\end{example}
Define the isomorphism
\[ \begin{tikzcd}
G \twplus{t} H \rar["\sigma_t", "\cong" swap] & G+H
\end{tikzcd} \]
which is the identity in the $ij$ component for $i < t$ and is the flip map for $i \geq t$.
We denote the inclusions by $i_t$ and $j_t$, and the following diagram commutes in $\LL_n$.
\[ \begin{tikzcd}
G \rar{i_t} \ar[dr,bend right,"i"'] & G \twplus{t} H  \dar["\sigma_t", "\cong"'] & \lar[swap]{j_t} H  \ar[dl, bend left, "j"] \\
& G + H
\end{tikzcd} \]

The following is an essential lemma about the behavior of this structure with respect to coface maps. 
\begin{lemma}\label{twplus sigma d interchange}
If $k < t$, then $d_k(G\twplus{t} H) = (d_kG) \twplus{t-1} (d_kH)$ and the following diagram in $\LL$ commutes.
\[ \begin{tikzcd}
(d_kG) \twplus{t-1} (d_kH) \dar{\sigma_{t-1}} \rar{=} & d_k(G\twplus{t} H) \dar{d_k(\sigma_t)}  \rar{d^k} & G\twplus{t} H \dar{\sigma_t}  \\
d_kG + d_kH \rar{=} & d_k(G + H)  \rar{d^k} & G+H
\end{tikzcd} \]
If $k\geq t$, then $d_k(G\twplus{t} H) = (d_kG) \twplus{t} (d_kH)$ and the following diagram commutes.
\[ \begin{tikzcd}
(d_kG) \twplus{t} (d_kH) \dar{\sigma_{t}} \rar{=} & d_k(G\twplus{t} H) \dar{d_k(\sigma_t)}  \rar{d^k} & G\twplus{t} H \dar{\sigma_t}  \\
d_kG + d_kH \rar{=} & d_k(G + H)  \rar{d^k} & G+H
\end{tikzcd} \]
More concisely,
\[
  \sigma_t d^k = \begin{cases}
  d^k \sigma_{t-1} & \text{if } k < t \\
  d^k \sigma_t & \text{if } k \geq t.
  \end{cases}
\]
\end{lemma}
\begin{proof}
This is a direct verification from the definitions. 
\end{proof}

Next is a special case which will be used in \cref{lemma once twisted}.

\begin{lemma}\label{lemma d-one}
If $G,H \in \LL_n$ for $n \geq 1$, then the following diagram commutes.
\[ \begin{tikzcd}[column sep=tiny]
& d_1G \twplus1 d_1H \ar[dl,"d^1"'] \ar[dr,"d^1"] \\
G \twplus1 H \ar[rr, dashed] \ar[dr,"\sigma_1"'] & & G\twplus2 H \ar[dl,"\sigma_2"] \\
& G + H
\end{tikzcd} \]
\end{lemma}
The dashed map is the isomorphism $\sigma_2^{-1}\sigma_1$.
\begin{proof}
The first part of \cref{twplus sigma d interchange} with $k=1 < 2 = t$ tells us that $d^1 \sigma_1 = \sigma_2 d^1$, while the second part with $t=k=1$ tells us that $d^1 \sigma_1 = \sigma_1 d^1$, and both maps having domain $d_1G \twplus1 d_1H$.
\end{proof}

\begin{definition}\label{def rep twten}
Suppose $G,H \in \LL_1$, $x \in X_G$, and $y\in X_H$. 
Let $x \twten y \in X_{G\twplus{1} H}$ be the element defined by the following diagram
\[ \begin{tikzcd}
X_G \times X_H \dar{=} & X_{G + H} \lar["i^* \times j^*", "\cong" swap] \dar{\sigma_1^*} \\
X_G \times X_H & X_{G\twplus{1} H} \lar["i_1^* \times j_1^*", "\cong" swap] 
\end{tikzcd} \] 
That is, $x \twten y = (i_1^* \times j_1^*)^{-1}(x,y) = \sigma_1^*(x\ten y)$.
\end{definition}

\begin{proposition}
The preceding definition descends to a well-defined function
\[
  {-}\twten{-} \colon \mor CX \times \mor CX \to \mor CX.
\]
\end{proposition}
\begin{proof}
If $G' \xrightarrow{\equiv} G$ and $H' \xrightarrow{\equiv} H$ are congruences in $\kong_1$, then there is an associated congruence $G' \twplus{1} H' \xrightarrow{\equiv} G \twplus1 H$ as well, and the following diagram commutes.
\[ \begin{tikzcd}[row sep=tiny]
& X_G \times X_H \ar[dd,"\cong"] \ar[dl] & X_{G\twplus{1} H} \lar["i_1^* \times j_1^*", "\cong" swap] \ar[dd,"\cong"] \ar[dr] \\ 
SX_1 \times SX_1 & & &  SX_1
\\
& X_{G'} \times X_{H'} \ar[ul] & X_{G'\twplus{1} H'} \lar["i_1^* \times j_1^*", "\cong" swap]  \ar[ur]
\end{tikzcd} \]
Thus we may define $[x] \twten [y] \coloneqq [x \twten y]$ with no ambiguity about choice of representatives.
\end{proof}

\begin{lemma}\label{lemma on dom and codom}
If $f \colon x_0 \to x_1$ and $g \colon y_0 \to y_1$ are morphisms in $CX$, then 
$f\twten g$ is a map $x_0\ten y_0 \to y_1 \ten x_1$.
\end{lemma}
\begin{proof}
We may prove the result by working with representatives.
The diagram
\[ \begin{tikzcd}
X_{G+H} \ar[dr,"d_0"] \ar[rrr,"\cong"] \ar[dd,"\sigma_1^*"] & & & X_G \times X_H \dar{d_0 \times d_0} \\
& X_{d_0(G+H)} \rar{=}  
& X_{d_0G + d_0H} \rar{\cong} \dar["\sigma_0^*=\sigma^*"] & X_{d_0G} \times X_{d_0H} \dar{\tau} \\
X_{G\twplus1 H} \rar{d_0} \dar & X_{d_0(G\twplus1 H)} \rar{=} & X_{d_0H + d_0G} \rar{\cong} \dar & X_{d_0H} \times X_{d_0G} \dar \\
SX_1 \ar[rr,"d_0"] & & SX_0 & SX_0 \times SX_0 \lar["-\ten-" swap] 
\end{tikzcd} \]
commutes, using $d_0(G\twplus1 H) = d_0G \twplus0 d_0H = d_0H + d_0G$ along with \cref{lem symmetry} and \cref{twplus sigma d interchange}.
The two outer maps $X_G \times X_H \to SX_0$ are equal, implying that $\cod[ x\twten y ] = (\cod [y]) \ten (\cod [x])$.
A similar, but simpler diagram can be constructed for $d_1$ to see  $\dom[ x\twten y ] = (\dom [x]) \ten (\dom [y])$; it is simpler since $d_1(G\twplus1 H) = d_1G \twplus1 d_1H = d_1G + d_1H$.
\end{proof}

\begin{lemma}\label{lemma once twisted}
Given morphisms
\[ \begin{tikzcd}[row sep=tiny, column sep=small]
x_0 \rar{f_1} & x_1 \rar{f_2} & x_2 \\
y_0 \rar{g_1} & y_1 \rar{g_2} & y_2
\end{tikzcd} \]
in $CX$, we have
\[
(f_2\twten g_2) \circ (f_1 \ten g_1) = (f_2f_1) \twten (g_2g_1) = (g_2\ten f_2) \circ (f_1 \twten g_1).
\]
\end{lemma}
\begin{proof}
The composable pairs given are elements in $SX_2$, and hence we can find height 2 level graphs $G$ and $H$ and elements $x\in X_G$ and $y\in X_H$ so that $d_2x \in X_{d_2G}$ represents $f_1$ and $d_0x \in X_{d_0G}$ represents $f_2$, and likewise for $y$ and $g_1,g_2$.
We have $d_1(G\twplus1 H) = (d_1 G) \twplus 1 (d_1 G) = d_1(G\twplus2 H)$.
Let $\sigma_1 \colon G \twplus 1 H \to G + H$ and $\sigma_2 \colon G \twplus 2 H \to G + H$  be the partial flipping morphisms from above.
By \cref{lemma d-one} we have 
\[
  d_1\sigma_1^*(x \otimes y) = d_1 \sigma_2^*(x\otimes y) \in X_{d_1(G \twplus 1 H)} = X_{d_1(G\twplus2 H)}.
\]
We claim that this element represents all three maps; we use freely the identities from \cref{twplus sigma d interchange}. 
First, the equalities
\[
  d_1 \sigma_1^*(x\ten y) = \sigma_1^*d_1 (x\ten y) = \sigma_1^* ((d_1 x) \ten (d_1 y))
\]
show this element represents $(f_2f_1) \twten (g_2g_1)$.
The equations (the first of which uses $\sigma_0 = \sigma$ and \cref{lem symmetry})
\[
\begin{gathered}
d_0 \sigma_1^*(x\ten y) = \sigma_0^* d_0(x\ten y) = \sigma_0^*(d_0x \ten d_0y) = \sigma^*(d_0x \ten d_0y) = d_0y \ten d_0 x \\
d_2 \sigma_1^*(x\ten y) = \sigma_1^* d_2(x\ten y) = \sigma_1^*(d_2x \ten d_2y) = d_2x \twten d_2 y
\end{gathered}
\]
tell us our element represents $(g_2\ten f_2) \circ (f_1 \twten g_1)$.
Finally, we have 
\[
\begin{gathered}
d_0 \sigma_2^*(x\ten y) = \sigma_1^* d_0(x\ten y) = \sigma_1^*(d_0x \ten d_0y) = d_0x \twten d_0 y \\
d_2 \sigma_2^*(x\ten y) = \sigma_2^* d_2(x\ten y) = \sigma_2^*(d_2x \ten d_2y) = d_2x \ten d_2 y
\end{gathered}
\]
so the element represents $(f_2\twten g_2) \circ (f_1 \ten g_1)$.
\end{proof}

\begin{lemma}\label{lemma twice twisted}
Given morphisms
\[ \begin{tikzcd}[row sep=tiny, column sep=small]
x_0 \rar{f_1} & x_1 \rar{f_2} & x_2 \\
y_0 \rar{g_1} & y_1 \rar{g_2} & y_2
\end{tikzcd} \]
in $CX$, we have
\[
(g_2\twten f_2) \circ (f_1 \twten g_1) = (f_2f_1) \ten (g_2g_1).
\]
\end{lemma}
\begin{proof}
Use the same notation from the beginning of the previous proof.
We introduce a new graph $G\tilde{+}H$ as displayed:
\[ \begin{tikzcd}[row sep=small, column sep=tiny]
G_{00} + H_{00} \ar[dr] & & H_{11} + G_{11} \ar[dl] \ar[dr] & &  G_{22} + H_{22} \ar[dl] \\
& G_{01} + H_{01} \ar[dr] & & H_{12} + G_{12} \ar[dl]\\
& & G_{02} + H_{02}
\end{tikzcd} \]
Define $\tilde \sigma \colon G \tilde{+} H \to G + H$ which is the flip map in components $11$ and $12$, and the identity elsewhere.
Our strategy is to show that
\[
  d_1 \tilde\sigma^* (x\ten y) \in X_{d_1(G \tilde{+} H)}
\]
represents both maps in the statement of the lemma.

Notice that we have $d_0(G\tilde{+}H) = (d_0H)\twplus1 (d_0G)$, $d_1(G\tilde{+}H) = d_1G + d_1H$, and $d_2(G\tilde{+}H) = (d_2G)\twplus1 (d_2H) $.
Furthermore, the following diagrams commute:
\[ \begin{tikzcd}
(d_0H) \twplus1 d_0G \rar{\sigma_1} \dar{d^0} & d_0H + d_0G \rar{\sigma} & d_0G + d_0H\dar{d^0}  \\ 
G\tilde{+}H \ar[rr,"\tilde \sigma"] & & G+ H
\end{tikzcd} 
\]
\[ \begin{tikzcd}[column sep=small]
& d_1G + d_1 H \ar[dr,"d^1"] \ar[dl,"d^1"'] \\
G \tilde{+} H \ar[rr,"\tilde\sigma"] & & G+H
\end{tikzcd} \qquad 
\begin{tikzcd}
d_2G\twplus1 d_2H \rar{\sigma_1} \dar{d^2} & d_2G + d_2H \dar{d^2} \\
G\tilde{+} H \rar{\tilde\sigma} & G+H
\end{tikzcd}
\]
Now this comes down to computation:
from the above diagrams and \cref{lem symmetry},
\[
\begin{gathered}
d_0 \tilde\sigma^*(x\ten y) = \sigma_1^*\sigma^* d_0(x\ten y) = \sigma_1^*\sigma^*(d_0x \ten d_0y) = \sigma_1^* (d_0y \ten d_0 x) = d_0y \twten d_0x \\
d_2 \tilde\sigma^*(x\ten y) = \sigma_1^* d_2(x\ten y) = \sigma_1^*(d_2x \ten d_2y) = d_2x \twten d_2 y
\end{gathered}
\]
so the element indeed represents $(g_2\twten f_2) \circ (f_1 \twten g_1)$.
On the other hand, 
\[
  d_1\tilde\sigma^* (x\otimes y) = d_1(x\otimes y) = d_1x \otimes d_1 y
\]
so the element represents $(f_2f_1) \ten (g_2g_1)$.
\end{proof}

\begin{definition}\label{def symmetry iso}
If $x,y\in \ob CX$, define $\tau_{x,y} \colon x\ten y \to y \ten x$ to be $\id_x \twten \id_y$.
\end{definition}

\begin{proposition}\label{prop symmetry naturality}
For any $x,y \in \ob CX$, we have $\tau_{y,x} \circ \tau_{x,y} = \id_{x\ten y}$.
If $f \colon x \to x'$ and $g \colon y \to y'$ are two morphisms in $CX$, then the diagram
\[ \begin{tikzcd}
x \ten y \rar{\tau_{x,y}} \dar[swap]{f\ten g} & y \ten x \dar{g \ten f} \\
x' \ten y' \rar[swap]{\tau_{x',y'}} & y' \ten x'
\end{tikzcd} \]
commutes.
\end{proposition}
\begin{proof}
The first statement follows from \cref{lemma twice twisted}, since  
\[
  \tau_{y,x} \circ \tau_{x,y} = (\id_y \twten \id_x) \circ (\id_x \twten \id_y) = (\id_x \id_x) \ten (\id_y \id_y) = \id_{x\ten y}.
\]
The second statement similarly follows from \cref{lemma once twisted} since both maps are equal to $f \twten g$.
Indeed, we first have \[ \tau_{x',y'} (f\ten g) = (\id_{x'} \twten \id_{y'}) \circ (f\ten g) = (\id_{x'} f) \twten (\id_{y'} g)\]
utilizing the first equality, and then we have 
\[
  (g \ten f) \tau_{x,y} = (g\ten f) \circ (\id_x \twten \id_y) = (f\id_x) \twten (g \id_y)
\]
using the second equality.
\end{proof}

\begin{theorem}\label{thm permutative}
The functor $C \colon \seg(\LL) \to \cat$ factors through the category $\perm$ of small permutative categories and strict symmetric monoidal functors.
More precisely, this uses the monoidal structure on $CX$ from \cref{prop strict monoidal cat} and the symmetry isomorphism from \cref{def symmetry iso}.
\end{theorem}
\begin{proof}
By \cref{prop strict monoidal cat} and \cref{prop symmetry naturality}, $CX$ is a permutative category for every $X\in \seg(\LL)$.
But the definitions imply that both the monoidal structure and the symmetry are natural in $X$.
\end{proof}

\begin{example}\label{example C star monoidal}
The equivalence $C(\ast) \to \csp$ from \cref{example C star Csp} is a symmetric monoidal functor.
To see this, recall that objects of $C(\ast)$ have been identified with objects of $\finsetskel$ and morphisms are equivalence classes of cospans in $\finsetskel$.
Direct inspection shows that the symmetry isomorphisms in $C(\ast)$ and $\finsetskel$ coincide.
The monoidal constraint for $C(\ast) \to \csp$ is then inherited from that for $\finsetskel \to \finset$, and we conclude that $C(\ast) \to \csp$ is a symmetric monoidal functor.
That $\csp$ is equivalent to a skeletal permutative category is well known; see \cite[5.4]{Lack:CP}.
Lack shows that $C(\ast)$ is the prop for special Frobenius monoids \cite[Proposition 6.1]{CoyaFong:CPECFM}, also known as commutative separable algebras \cite{Carboni:MRGR}. 
The introduction of \cite{RSW:GCSACG} contains a nice overview of these structures.
\end{example}

It seems to be a folklore result that special Frobenius monoids are also the algebras for the terminal properad. 
This is consistent with the fact that $C$ gives the free prop generated by a properad.

\section{Properads give labelled cospan categories}\label{sec from properad to lcc}

In this section, we show that the envelope of a properad is a  labelled cospan category.
More precisely, if $P$ is a properad then the composite
\[
  C(N_1(P)) \to C(\ast) \xrightarrow{\simeq} \csp
\]
is a labelled cospan category.
Since $N_1$ is an equivalence of categories between $\properad$ and $\seg(\LL)$, it is enough to prove the following.
Below, we will verify the axioms from \cref{def lcc} in a series of lemmas.
\begin{proposition}\label{prop CX is lcc}
If $X \in \seg(\LL)$, then $C(X) \to \csp$ is a labelled cospan category.
\end{proposition}
\begin{proof}
Combine \cref{lem lcc obj decomp}, \cref{lem lcc free gen}, \cref{lem lcc reduced and free}, and \cref{lem lcc pullback}.
\end{proof}

Suppose $X\in \seg(\LL)$, and consider functor $CX \to C(\ast)$.
The set \[ X_\mathfrak{e} = X_{\underline{1}} \subseteq \sum_{k \in \mathbb{N}} X_{\uk} = \sum_{w\in \WW_0} \overline{X}_w = \ob CX\] is the set of \emph{connected objects} of $CX$ in the sense of \cref{def connected reduced}.
Likewise, the set
\[
\sum_{n,m \geq 0} X_{\mathfrak{c}_{n,m}} \subseteq \sum_{w \in \WW_1} \overline{X}_w = \mor CX
\]
is the set of \emph{connected morphisms}.
These displays use that the component of $\uk \in \kong_0$ and the component of $\mathfrak{c}_{n,m} \in \kong_1$ are discrete.
The \emph{reduced morphisms} are those on (congruence classes) of height 1 level graphs so that each vertex has either an input or an output.
Alternatively, these graphs are those without a $\mathfrak{c}_{00}$-summand.

\begin{lemma}\label{lem lcc obj decomp}
Let $c$ be an object of $CX$ so that $\pi(c) = \underline{n}$.
Then $c = c_1 \otimes \cdots \otimes c_n$ for some connected objects $c_1, \dots, c_n$.
\end{lemma}
\begin{proof}
Recall that $CX$ is a strict monoidal category, and on objects the $n$-fold tensor product is given as the inverse of the Segal map
\[
  X_{\un} \xrightarrow{\cong} X_{\underline{1}} \times \dots \times X_{\underline{1}}.
\]
Thus given an object $c$ with $\pi(c) = \underline{n}$, we see that $c$ is the tensor product of $n$ connected objects $c_1, \dots, c_n$.
\end{proof}

\begin{remark}[Uniqueness of decomposition]\label{rmk uniqueness decomp}
We emphasize that as the $n$-fold tensor product is given by the inverse of the Segal map $X_{\un} \xrightarrow{\cong} X_{\underline{1}} \times \dots \times X_{\underline{1}}$
there is actually a \emph{unique} list of connected objects $c_1, \dots, c_n$ with $c = c_1 \ten \dots \ten c_n$ in the previous lemma.
\end{remark}

See also \cref{def slcc} below.

\begin{lemma}\label{lem lcc free gen}
The abelian monoid $\hom(\mathbf{1},\mathbf{1})$ is freely generated by the set $\homcon(\mathbf{1},\mathbf{1})$ of connected morphisms.
\end{lemma}
\begin{proof}
The set of objects of $\LL_1$ with empty input and empty output is in bijection with $\mathbb{N}$, and we write temporarily abbreviate the object corresponding with $n$ by
\[
  n\mathfrak{c}_{00} \coloneqq \underbrace{\mathfrak{c}_{00} + \dots + \mathfrak{c}_{00}}_n.
\]
This object is not isomorphic, hence not congruent, to any other object in $\LL_1$.
Every automorphism of $n\mathfrak{c}_{00}$ is a congruence, and the group of automorphisms is the symmetric group $\Sigma_n$.
For $\gamma \in \Sigma_n$, the left diagram below commutes
\[ \begin{tikzcd}[row sep=small]
& n\mathfrak{c}_{00} \ar[dd,"\gamma"]  & X_{n\mathfrak{c}_{00}} \rar{\cong} & \prod\limits_{i=1}^n X_{\mathfrak{c}_{00}} \ar[dr,bend left,"\pi_k"] \\
\mathfrak{c}_{00} \ar[ur,"k"] \ar[dr,"\gamma(k)"'] & & & & X_{\mathfrak{c}_{00}} 
\\
& n\mathfrak{c}_{00} & X_{n\mathfrak{c}_{00}} \rar{\cong} \ar[uu,"\gamma^*"]  & \prod\limits_{i=1}^n X_{\mathfrak{c}_{00}}  \ar[ur,bend right,"\pi_{\gamma(k)}"']
\end{tikzcd} \]
hence so too does the right diagram, where $\pi_t$ is the $t$th projection.
This implies that the following commutes
\[ \begin{tikzcd}
X_{n\mathfrak{c}_{00}} \rar{\cong} & \prod\limits_{i=1}^n X_{\mathfrak{c}_{00}} \\
X_{n\mathfrak{c}_{00}} \rar{\cong} \ar[u,"\gamma^*"]  & \prod\limits_{i=1}^n X_{\mathfrak{c}_{00}} \uar["{-}\cdot \gamma"']
\end{tikzcd} \]
where the right-hand map is the standard right action on lists, that is $(x_1, \dots, x_n) \mapsto (x_1, \dots, x_n)\cdot \gamma = (x_{\gamma(1)}, \dots, x_{\gamma(n)})$.
Further, the monoidal structure comes from the bottom right isomorphism in the following diagram.
\[ \begin{tikzcd}
X_{(n+m)\mathfrak{c}_{00}} \dar{=}\ar[rr,"\cong"] &   & \prod\limits_{i=1}^{n+m} X_{\mathfrak{c}_{00}} \dar \\
X_{n\mathfrak{c}_{00} + m\mathfrak{c}_{00} } \rar{\cong} & X_{n\mathfrak{c}_{00}}  \times X_{m\mathfrak{c}_{00} } \rar{\cong} & \prod\limits_{i=1}^{n} X_{\mathfrak{c}_{00}} \times \prod\limits_{i=1}^{m} X_{\mathfrak{c}_{00}}
\end{tikzcd} \]
Thus we see that $\hom(\mathbf{1},\mathbf{1})$ as a monoid is isomorphic to
\[
\sum_{n\geq 0} \left(\prod_{i=1}^n X_{\mathfrak{c}_{00}}\right) / 
\Sigma_n^{\oprm}
\]
which is the free abelian monoid on $X_{\mathfrak{c}_{00}} = \homcon(\mathbf{1},\mathbf{1})$.
\end{proof}

Let $R = \morred C(\ast) \subseteq \WW_1$ be the congruence classes of graphs where each vertex has an input or output.
The following lemma, concerning the connected components of the $q$-fibers in the congruence category $\kong$ appearing in \cref{lem pushforwards}, will be used in the proofs of \cref{lem lcc reduced and free} and \cref{auxillary to lcc pullback}.
If $r\in R$, then by \cref{rmk congruence height 1} the category $\kong_r$ is a preorder, and $\overline{X}_r \cong X_G$ if $[G] = r$ (see \cref{reduced graphs remark}).

\begin{lemma}\label{lem tensor lemma}
If $r\in R \subseteq \WW_1$ is reduced and $w\in \WW_1$ is arbitrary, then
\begin{equation}\label{proposed lemma eq}
  \kong_r \times \kong_w \to \kong_{r\ten w}
\end{equation}
sending $(G,H)$ to $G+H$ is an equivalence of categories.
It follows that if $X\in \seg(\LL)$, then $\overline{X}_r \times \overline{X}_w \to \overline{X}_{r\ten w}$ is a bijection.
\end{lemma}
\begin{proof}
By the defintion in \cref{lem ten descends}, the class $r\ten w \in \WW_1$ contains some graph $G+H$ where $[G] = r$ and $[H] = w$.
Thus \eqref{proposed lemma eq} is essentially surjective.
It is also clearly faithful.
We now prove that it is full.
We must show that if $G, G'\in r$ and $H,H'\in w$ then any congruence $f \colon G+H \equiv G' + H'$ restricts to congruences $f^L \colon G\equiv G'$ and $f^R \colon H \equiv H'$ with $f = f^L + f^R$.
For $i=0,1$ we know that $\id_{G_{ii}} \colon G_{ii} \to G'_{ii}$ and $\id_{H_{ii}} \colon H_{ii} \to H'_{ii}$ sum to $f_{ii} = \id_{G_{ii} + H_{ii}} \colon (G+H)_{ii}  \to (G+H')_{ii}$, so this part restricts.
Since every vertex of $G$ or $G'$ has either an input or an output, the first horizontal maps in the following commutative rectangle are surjective.
\[ \begin{tikzcd}
G_{00} + G_{11} \rar[two heads] \dar["=", "f^L_{00} + f^L_{11}"swap] & G_{01} \rar[hook] & G_{01} + H_{01} \dar{f_{01}} \\
G'_{00} + G'_{11} \rar[two heads] & G'_{01} \rar[hook] & G'_{01} + H'_{01}.
\end{tikzcd} \]
It follows that $f_{01}$ restricts to a function $f^L_{01} \colon G_{01} \to G'_{01}$.
We must still check that $f^L$ is a congruence, and for this simply note that $(f^{-1})_{01}$ restricts to a function $G'_{01} \to G_{01}$ by the same argument, providing an inverse to $f^L_{01}$.
Since $f^L$ is an isomorphism, it is a congruence.

Now the composite $H \to G+H \to G'+H'$ must land in $H'$, since $f$ is a monomorphism and $f^L \colon G \to G'$ is an epimorphism.
Thus $f$ restricts to a map $f^R \colon H \to H'$, which must be an epimorphism since $f$ is.
We conclude that \eqref{proposed lemma eq} is full, hence an equivalence.

For the conclusion, suppose $X$ is Segal.
In the following diagram, we use the result to deduce the isomorphism at the bottom of the triangle; the others come from Segality.
\[ \begin{tikzcd}
X_G \times X_H \dar & & X_{G+H} \ar[ll,"\cong"'] \dar \ar[dl] \\
\colim\limits_{G \in \kong_r^\oprm} X_G \times \colim\limits_{H\in \kong_w^\oprm} X_H \dar{=} & \colim\limits_{(G,H)\in \kong_r^\oprm\times \kong_w^\oprm} X_{G+H} \lar["\cong"']  & \colim\limits_{K\in \kong_{r\ten w}^\oprm} X_K \dar{=} \lar["\cong"'] \\
\overline{X}_r \times \overline{X}_w \ar[rr,dashed] 
\dar[hook] & &  \overline{X}_{r\ten w} \dar[hook]\\
SX_1 \times SX_1 \ar[rr] & & SX_1
\end{tikzcd} \]
We conclude that the dotted arrow is an isomorphism.
\end{proof}

\begin{remark}
The proof of \cref{lem tensor lemma} truly depended on one of the elements being reduced. 
The simplest issue was apparent in \cref{lem lcc free gen} and its proof, since \[ 
X_{\mathfrak{c}_{00}} \times X_{\mathfrak{c}_{00}} =
\overline{X}_{\mathfrak{c}_{00}} \times \overline{X}_{\mathfrak{c}_{00}} \to \overline{X}_{\mathfrak{c}_{00} + \mathfrak{c}_{00}} = X_{\mathfrak{c}_{00} + \mathfrak{c}_{00}} / \Sigma_2^\oprm \cong (X_{\mathfrak{c}_{00}} \times X_{\mathfrak{c}_{00}})/ \Sigma_2^\oprm
\]
is not typically a bijection.
But this also occurs for other congruence classes. 
For example, if $G = \mathfrak{c}_{10} + \mathfrak{c}_{00}$ and $H = \mathfrak{c}_{01} + \mathfrak{c}_{00}$, then the non-trivial endo-congruence on $G+H = \mathfrak{c}_{10} + \mathfrak{c}_{00} + \mathfrak{c}_{01} + \mathfrak{c}_{00}$ does not restrict in the way of the proof.
Further, the automorphism group of $G+H$ in $\kong_1$ is $\Sigma_2$, whereas $G$ and $H$ have trivial automorphism groups. 
\end{remark}

\begin{lemma}\label{lem lcc reduced and free}
The map
\[ \begin{tikzcd}
  \homred(c,d) \times \hom(\mathbf{1},\mathbf{1}) \rar{\otimes} & \hom(c,d)
\end{tikzcd} \]
is a bijection.
\end{lemma}
\begin{proof}
Let $R = \morred C(\ast) \subseteq \WW_1$ be the congruence classes of graphs where each vertex has an input or output, and $T  = \{ n \mathfrak{c}_{00}  \mid n \in \mathbb{N} \} \subseteq \WW_1$ be the set of graphs with empty inputs and outputs (see proof of \cref{lem lcc free gen}).
Every element $w\in \WW_1$ decomposes uniquely as $w = r \ten t$ where $r\in R$ and $t\in T$.
Indeed, to get this decomposition, if $G$ is an arbitrary height 1 level graph, then we can form new graphs $G^r$ and $G^t$ and morphisms in $\LL_1$ 
\[ \begin{tikzcd}
G^r \rar[hook,"f"] & G  & G^t \lar[hook', "g"']
\end{tikzcd} \]
so that every vertex of $G^r$ has an input or output, $f_{ii} \colon G^r_{ii} \to G_{ii}$ is an identity, $g_{ii} \colon \underline{0} \to G_{ii}$ is the unique map, and $f_{01}$, $g_{01}$ are order-preserving.
See \cref{fig: graph decomp} for an illustration.
There is a unique congruence $G^r + G^t \equiv G$ taking $f$ and $g$ to the usual inclusions into the sum.
If $H^r + H^t \equiv G$ is some other congruence with $[H^r] \in R$ and $[H^t] \in T$, then the composite $H^r + H^t \to G \to G^r + G^t$ takes every vertex in $H^r$ to a vertex in $G^r$ and every vertex in $H^t$ to a vertex in $G^t$, hence gives congruences $H^r \equiv G^r$ and $H^t \equiv G^t$.
This establishes uniqueness of the decomposition, and we conclude that $\WW_1 \cong R \times T$.
Coupling this with \cref{lem tensor lemma}, we conclude that the following 
\[ 
\morred C(X) \times \hom(\mathbf{1},\mathbf{1}) = \sum_R \overline{X}_r \times \sum_T \overline{X}_t = \sum_{R\times T} \overline{X}_r  \times \overline{X}_t \xrightarrow{\cong} \sum_{\WW_1} \overline{X}_w = \mor C(X)
\]
is a bijection.

\begin{figure}
\labellist
\small\hair 2pt
  \pinlabel {$G^r$} [t] at 169 0
  \pinlabel {$G$} [t] at 550 0
  \pinlabel {$G^t$} [t] at 890 0
  \pinlabel {$\hookrightarrow$} [t] at 359 0
  \pinlabel {$\hookleftarrow$} [t] at 760 0
\endlabellist
\centering
\includegraphics[width=0.7\textwidth]{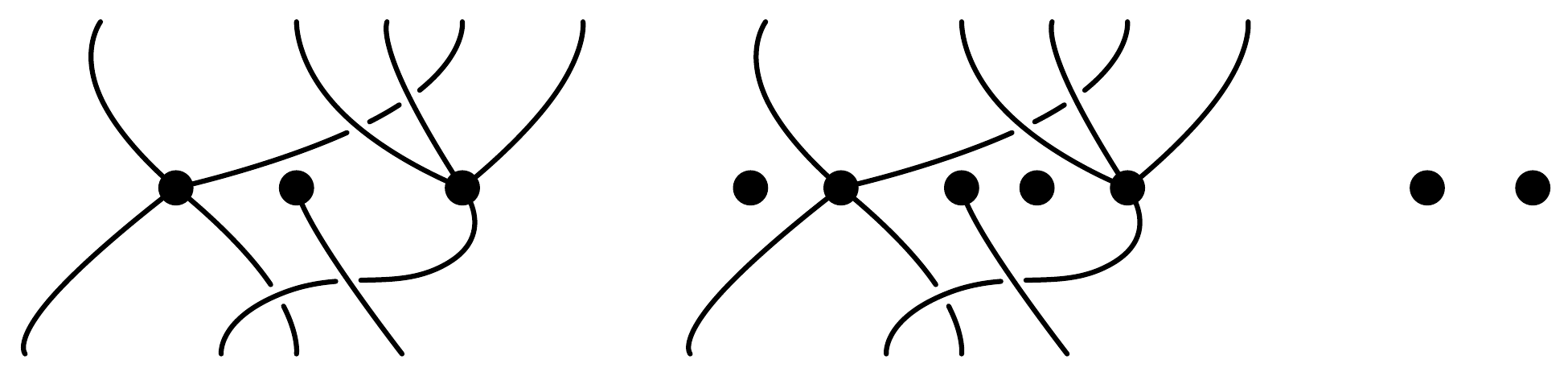}
\caption{Graph decomposition from \cref{lem lcc reduced and free}}
\label{fig: graph decomp}
\end{figure}
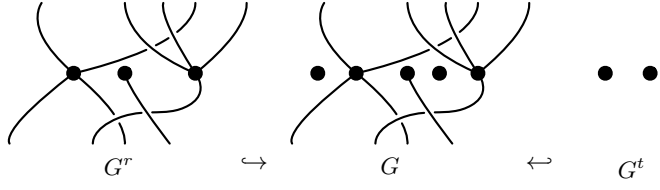

The large rectangle and bottom square in the following are pullbacks,
\[ \begin{tikzcd}
\homred(c,d) \times \hom(\mathbf{1},\mathbf{1}) \rar \dar & \morred C(X) \times \hom(\mathbf{1},\mathbf{1}) \dar{\cong} \\
\hom(c,d) \rar \dar & \mor C(X) \dar \\
\{ (c,d) \} \rar & \ob C(X) \times \ob C(X)
\end{tikzcd} \]
hence so is the upper square. The result follows.
\end{proof}

\begin{lemma}\label{lem lcc pullback}
For each four objects $c,d,c',d'$ in $C$, the following square is cartesian.
\[ \begin{tikzcd}
\homred(c,d) \times \homred(c',d') \rar{\otimes} \dar{\pi} \arrow[dr, phantom, "\lrcorner" very near start] & \homred(c\otimes c', d\otimes d') \dar{\pi} \\
\homred_\csp(\pi c,\pi d) \times \homred_\csp(\pi c',\pi d') \rar{\ten} & \homred_\csp(\pi c \ten \pi c',\pi d \ten \pi d') 
\end{tikzcd} \]
\end{lemma}

This follows from a different result.

\begin{proposition}\label{auxillary to lcc pullback}
If $X\to Y$ is a morphism in $\seg(\LL)$, then 
\[ \begin{tikzcd}
\morred C(X) \times \morred C(X) \rar \dar \arrow[dr, phantom, "\lrcorner" very near start] & \morred C(X) \dar \\
\morred C(Y) \times \morred C(Y) \rar & \morred C(Y)
\end{tikzcd} \]
is cartesian.
\end{proposition}
\begin{proof}
Let $R = \morred C(\ast) \subseteq \WW_1$ be the congruence classes of height 1 level graphs where each vertex has an input or output, so that
\[
  \morred C(X) = \sum_{r\in R} \overline{X}_r.
\]
By \cref{lem tensor lemma}, there is a canonical bijection
\[
  \morred C(X) \times \morred C(X) = \sum_{r\in R} \overline{X}_r \times \sum_{r'\in R} \overline{X}_{r'} \cong \sum_{(r, r') \in R\times R} \overline{X}_{r} \times \overline{X}_{r'} \cong \sum_{(r, r') \in R \times R} \overline{X}_{r \ten r'}.
\]
Thus it suffices to observe that the following square, 
\[ \begin{tikzcd}
\sum\limits_{r, r' \in R} \overline{X}_{r \ten r'} \rar \dar \arrow[dr, phantom, "\lrcorner" very near start] & \sum\limits_{s\in R} \overline{X}_s \dar \\
\sum\limits_{r, r' \in R} \overline{Y}_{r \ten r'} \rar & \sum\limits_{s\in R} \overline{Y}_s
\end{tikzcd} \]
whose horizontal legs are induced by $\ten \colon R \times R \to R$, is cartesian.
\end{proof}

\begin{proof}[Proof of \cref{lem lcc pullback}]
Apply the previous lemma in the case when $Y = \ast$, using that $C(\ast) \to \csp$ is fully faithful.
Suppose $f\colon \pi c \to \pi d$, $g\colon \pi c' \to \pi d'$, and $h\colon c\otimes c' \to d\otimes d'$ are reduced morphisms with $\pi(h) = f \otimes g$.
By the previous lemma, we know there are unique maps $\tilde f\colon a \to b$ and $\tilde g \colon a' \to b'$ with $\tilde f \otimes \tilde g = h$, $\pi(\tilde f) = f$, and $\pi(\tilde g) = g$.
The only thing to verify is that $a = c$, $b = d$, $a' = c'$ and $b' = d'$, but this follows from the uniqueness of decompositions of objects of $CX$ into connected objects (as was observed in \cref{rmk uniqueness decomp}).
\end{proof}

\begin{proposition}\label{prop C is faithful}
The functor $C \colon \seg(\LL) \to \cat$ is faithful.
\end{proposition}
\begin{proof}
If $X, Y \in \seg(\LL)$, then a map $X \to Y$ is uniquely determined by its action on elementary objects $X_{\mathfrak{e}} \to Y_{\mathfrak{e}}$ and $X_{\mathfrak{c}_{p,q}} \to Y_{\mathfrak{c}_{p,q}}$ (as $p$ and $q$ vary).
If $f_0, f_1 \colon X \to Y$ are two maps in $\seg(\LL)$, then the following diagrams commute for $i=0,1$.
\[ \begin{tikzcd}
X_{\mathfrak{c}_{p,q}} \dar[hook] \rar{f_i} & Y_{\mathfrak{c}_{p,q}}  \dar[hook] & X_{\mathfrak{e}} \dar[hook] \rar{f_i} & Y_{\mathfrak{e}} \dar[hook] \\
\morcon C(X) \rar{C(f_i)} & \morcon C(Y) & \ob C(X) \rar{C(f_i)} & \ob C(Y)
\end{tikzcd} \]
Thus if $C(f_0) = C(f_1)$, then $f_0 = f_1$.
\end{proof}
This proposition implies that the functors from $\seg(\LL)$ to permutative categories and labelled cospan categories are also faithful.

\section{Strict labelled cospan categories}
In this section we identify a special kind of labelled cospan category, and show that to each we may associate a Segal $\LL$-presheaf (and hence a properad).
This will provide an inverse to our previous construction: the colors of this properad will be the connected objects, and the operations in the associated properad will be the connected morphisms.
Unfortunately, actually proving this gives a properad comes with substantial challenges.

We begin with an explanatory example.

\begin{example}
From the category $C(\ast)$, create a new category $C$ with a single additional object $\underline{1}'$ which is isomorphic to $\underline{1} \in C(\ast)$.
As a permutative category, it will satisfy the equation $\underline{1}' \ten \un = \underline{1} \ten \un$ for $n\geq 1$ and the equation $\underline{1}' \ten \underline{1}' = \underline{2}$.
There is an evident functor $C \to C(\ast) \to \csp$ sending the new object to $\underline{1}$.
By \cref{rmk uniqueness decomp} this labelled cospan category is not $C(X)$ for any $X\in \seg(\LL)$, since $\underline{2} = \underline{1} \ten \underline{1} = \underline{1}'\ten \underline{1} = \underline{1}\ten \underline{1}' = \underline{1}' \ten \underline{1}'$ admits several distinct decompositions into connected objects.
\end{example}

Thus not all labelled cospan categories arise from properads (though they do up to equivalence).
We now isolate the image of the functor $C$.
Recall that a strict map of permutative categories is a strict monoidal functor which is compatible with the symmetries.

\begin{definition}\label{def slcc}
A \emph{strict labelled cospan category} is a strict map of permutative categories $\pi \colon C \to C(\ast)$ satisfying the conditions of \cref{def lcc} along with 
\begin{enumerate}[label={(\arabic*')}, ref={\arabic*'}]
\item If $\pi(c) = \un$, then there are unique connected objects $c_1, \dots, c_n$ with $c = c_1 \otimes \dots \otimes c_n$.
\label{slcc: obj decomp}
\end{enumerate}
Maps are strict maps of permutative categories so that the triangle
\[ \begin{tikzcd}[row sep=small, column sep=tiny]
C \ar[rr] \ar[dr] & & C' \ar[dl] \\
& C(\ast)
\end{tikzcd} \]
commutes.
We write $\slcc$ for the associated category.
\end{definition}

Alternatively, \eqref{slcc: obj decomp} can be rephrased as follows, which means that each strict labelled cospan category is a (colored) prop of a special type \cite{HackneyRobertson:OCP}.
\begin{enumerate}[label={(\arabic*'')}, ref={\arabic*''}]
\item The monoid $\ob(C)$ is the free monoid on the set of connected objects $\ob^{\mathrm{c}}(C)$. \label{slcc: obj decomp redo}
\end{enumerate}

One can show (see \cref{prop uniqueness of lcc structure}) that being a strict labelled cospan category is a \emph{property} of a permutative category, rather than a \emph{structure}.
Indeed, if $C$ is a permutative category, there is at most one functor $\pi \colon C \to C(\ast)$ making $C$ into a strict labelled cospan category. 
Nevertheless, we will often require the map $\pi$, so we do not attempt a definition free from it.

\cref{cor main theorem} states that the category of properads is equivalent to $\slcc$.
We know that $\properad \simeq \seg(\LL)$ by \cref{prop properads as presheaves}, and we have already exhibited a functor $\seg(\LL) \to \slcc$ by \cref{prop CX is lcc} and \cref{rmk uniqueness decomp}.
\cref{prop C is faithful} implies that this functor is faithful.
It remains to show this functor is full and essentially surjective.

Suppose $C \xrightarrow{\pi} C(\ast)$ is in $\slcc$.
We will define a presheaf $P = P(\pi) \in \pre(\LL)$ from this data, and later show that $CP \to C(\ast)$ is isomorphic to $C \to C(\ast)$, proving that $\seg(\LL) \to \slcc$ is essentially surjective.

\subsection{The inert part}\label{subsec inert}
We first define the restriction of the presheaf to $\LL_\elrm$; it is slightly more convenient to define this also on arbitrary height zero level graphs.
We define $P_{\un}$ and $P_{\mathfrak{c}_{m,n}}$ as subsets of the objects of $C$ and the \emph{connected} morphisms of $C$ so that the following diagrams are pullbacks.
\begin{equation}\label{diagram defining Pn and Pcmn}
\begin{tikzcd}
P_{\un} \rar \arrow[dr, phantom, "\lrcorner" very near start] \dar & \ob C \dar 
& &
P_{\mathfrak{c}_{m,n}}  \rar \arrow[dr, phantom, "\lrcorner" very near start] \dar & \morcon C \dar["s \times t"]
\\
\{ \un \} \rar & \ob C(\ast)  
& & 
P_{\um} \times P_{\un} \rar \arrow[dr, phantom, "\lrcorner" very near start]\dar  & \ob C \times \ob C \dar
\\
& & & \{ (\um,\un) \} \rar & \ob C(\ast) \times \ob C(\ast)  
\end{tikzcd} \end{equation}
Notice that $P_{\mathfrak{e}} = P_{\underline{1}}$ is precisely the set of connected objects in $C$, and there is a bijection
\begin{equation}\label{P objects projections}
  P_{\un} \xrightarrow{\cong} P_{\underline{1}} \times \dots \times P_{\underline{1}}
\end{equation}
that takes an object to the list of connected objects from \eqref{slcc: obj decomp} in \cref{def slcc}.

We must still define the action of our restriction on inert morphisms.
Between the elementary objects of $\LL$, there are the following inert morphisms:
\begin{itemize}
\item For $1 \leq i \leq m$ the following composite
\[
\begin{tikzcd}
\underline{1} \rar{\name{i}} & \um = d_1(\mathfrak{c}_{m,n}) \rar{d^1} & \mathfrak{c}_{m,n}
\end{tikzcd}
\]
which we call $\ell_i$.
\item For $1 \leq j \leq n$ the following composite
\[
\begin{tikzcd}
\underline{1} \rar{\name{j}} & \un = d_0(\mathfrak{c}_{m,n}) \rar{d^0} & \mathfrak{c}_{m,n}
\end{tikzcd}
\]
which we call $r_j$.
\item Automorphisms $f = (\gamma,\chi) \colon \mathfrak{c}_{m,n} \to \mathfrak{c}_{m,n}$, where $\gamma = d_1(f) \colon \um \to \um$ and $\chi = d_0(f) \colon \un \to \un$ are bijections.
\end{itemize}

The only relations among the above these morphisms are the following:
\begin{itemize}
\item $(\gamma, \chi) \circ \ell_i = \ell_{\gamma(i)}$
\item $(\gamma,\chi) \circ r_j = r_{\chi(j)}$
\item 
$(\gamma, \chi) \circ (\gamma',\chi') = (\gamma\gamma', \chi\chi')$.
\end{itemize}

\begin{notation}
If $\sigma \colon \underline{k} \to \underline{k}$ is a bijection, and $c_1, \dots, c_k$ are connected objects of $C$ then we write $\hat \sigma \colon c_1 \ten \cdots \ten c_k \to c_{\sigma^{-1}(1)} \ten \cdots \ten c_{\sigma^{-1}(k)}$ for the map determined by the symmetry isomorphism.
Then $\hat \gamma \circ \hat \sigma = \widehat{\gamma \circ \sigma}$.
\end{notation}

\begin{definition}\label{def corolla iso}
Suppose that $p \colon c_1 \ten \dots \ten c_m \to b_1 \ten \dots \ten b_n$ is in $P_{\mathfrak{c}_{m,n}} \subseteq \morcon C$.
We define $\ell_i^*(p) = c_i \in P_{\mathfrak{e}}$ and $r_j^*(p) = b_j \in P_{\mathfrak{e}}$.
If $f = (\gamma, \chi) \colon \mathfrak{c}_{m,n} \to  \mathfrak{c}_{m,n}$ is an isomorphism,
we define $f^*p = (\gamma,\chi)^*(p)$ to be the following composite of $p$ with symmetry isomorphisms.
\[ \begin{tikzcd}
c_{\gamma(1)} \ten \dots \ten c_{\gamma(m)} \dar{\hat \gamma} & b_{\chi(1)} \ten \dots \ten b_{\chi(n)}
\\
c_1 \ten \dots \ten c_m \rar{p} & b_1 \ten \dots \ten b_n \uar{\hat \chi^{-1}}
\end{tikzcd} \]
that is, \[ f^*p = (\gamma, \chi)^*(p) = \hat\chi^{-1} \circ p \circ \hat\gamma = \widehat{d_0(f)}\vphantom{d_0(f)}^{-1} \circ p \circ \widehat{d_1(f)}. \]
\end{definition}

We then have
\[
  (\gamma',\chi')^* (\gamma,\chi)^*(p) =(\gamma',\chi')^* (\hat \chi^{-1} p \hat \gamma) = (\hat \chi')^{-1} \hat \chi^{-1} p\hat\gamma\hat \gamma' = \widehat{\chi \chi'}\vphantom{\chi\chi'}^{-1}  p \widehat{\gamma \gamma'}
\]
which is $ (\gamma\gamma',\chi\chi')^*(p)$, as desired.
Likewise, 
\[
\begin{gathered}
\ell_i^* (\gamma, \chi)^*(p) = c_{\gamma(i)}  = \ell_{\gamma(i)}^*(p) \\
r_j^* (\gamma, \chi)^*(p) = b_{\chi(j)}  = r_{\chi(j)}^*(p)
\end{gathered}
\]
so we conclude that $P|_{\LL_\elrm} \colon \LL^\oprm_\elrm \to \set$ is indeed a functor.

Notice that our definition of $\ell_i^*$  actually factors through a map $d_1 \colon P_{\mathfrak{c}_{m,n}} \to P_{\um}$ taking $p$ to $c_1 \ten \cdots \ten c_m$ which appeared in \eqref{diagram defining Pn and Pcmn}, and then following by the $i$th projection from \eqref{P objects projections}.
A similar situation occurs for $r_j^*$.
We cannot do the same for $(\gamma,\chi)^*$.

\begin{definition}\label{def P restriction}
We define $P|_{\LL_\intrm}$ to be the right Kan extension of $P|_{\LL_\elrm}$ along the fully faithful inclusion $\LL_\elrm^\oprm \hookrightarrow \LL_\intrm^\oprm$.
\end{definition}
Since the functor is fully faithful, we can take the restriction of $P|_{\LL_\intrm}$ to be exactly equal to $P|_{\LL_\elrm}$. 
Further, by \eqref{P objects projections} we can arrange things so our previously-defined $P_{\un}$ agrees with the new one.
By \cite[Lemma 2.9]{ChuHaugseng}, if we can show that $P|_{\LL_\intrm}$ is the restriction of some $\LL$-presheaf $P$, then $P$ will automatically be Segal.

\subsection{The active part}
We have defined $P|_{\LL_\intrm}$ in \cref{def P restriction}.
In particular, this values of this presheaf are defined on every level graph $G$ and on every isomorphism.
Since $(\LL_\actrm, \LL_\intrm)$ is an orthogonal factorization system on $\LL$, it remains to define $P$ on every active map.
But each active map is isomorphic to $\alpha^*G \to G$ for $\alpha\colon [m] \to [n]$ an active map of $\simpcat$, so we will only extend to maps of this form.
Further if $G$ is a height $n$ level graph, we can focus on the maps $d^i \colon d_iG \to G$ for $0< i < n$ and $s^i \colon s_i G \to G$ for $0 \leq i \leq n$, since $\simpcat_\actrm$ is generated by such maps.

In this section it will be helpful to have the following explicit version of splitting a graph $G$ into its connected components.
A looser version of this was used in \cref{subsec properads as L presheaves} in proving \cref{lem N1 segal}.

\begin{definition}[Canonical splitting of $G$]\label{def canonical splitting}
If $G$ is a height $m$ level graph with $G_{0m} = \uk$, then for $x\in G_{0m}$ we define $G_x$ to be the connected height $m$ level graph where for each $i,j$ the following is a pullback whose top map $\iota_x$ is order-preserving. 
\[ \begin{tikzcd}
(G_x)_{ij} \arrow[dr, phantom, "\lrcorner" very near start] \rar[hook,"\iota_x"] \dar & G_{ij} \dar \\
\underline{1} \rar[hook,"x"] & G_{0m}
\end{tikzcd} \]
Taking the sum of the $\iota_x$, we have an isomorphism $\displaystyle \iota \colon \sum_{x=1}^k G_x \to G$ in $\LL_m$.
\end{definition}

\begin{construction}\label{constr muG}
For each height 1 level graph $G\in \LL_1$, we define a function $\mu = \mu_G \colon P_G \to \mor C$ so that the following diagram commutes.
\[ \begin{tikzcd}
P_{d_1G} \dar{\subseteq} & P_G \lar["d_1"'] \rar["d_0"] \dar["\mu_G"] & P_{d_0G} \dar{\subseteq} \\
\ob C & \mor C \lar["s"'] \rar["t"] & \ob C.
\end{tikzcd} \]
When $G = \mathfrak{c}$ is a corolla, $\mu$ is just the standard inclusion $P_{\mathfrak{c}} \subseteq \morcon C \subseteq \mor C$ from the beginning of \cref{subsec inert}.
More generally, suppose that $G$ has $n$ input-edges ($d_1G = G_{00} = \un$), $m$ output-edges ($d_0G = G_{11} = \um$), and $k$ vertices ($G_{01} = \uk$).
If $p\in P_G$, then we must define $\mu_G(p)$ to be a map
\[
  d_1(p) = c_1 \ten \cdots \ten c_n \to e_1 \ten \cdots \ten e_m = d_0(p)
\]
where $c_i$ and $e_i$ are connected objects.
Recall the canonical splitting
\[
  \iota \colon \sum_{x=1}^k G_x \xrightarrow{\cong} G
\]
from \cref{def canonical splitting}. 
We write $\iota^0 = d_0(\iota) = \iota_{11}$ and $\iota^1 = d_1(\iota) = \iota_{00}$.
If $p \in P_G$, write $p_x \coloneqq \iota_x^*(p) \in P_{G_x} \subseteq \morcon C$, which is a map
\[
  p_x \colon c_{i_1} \ten \cdots \ten c_{i_r} \to e_{j_1} \ten \cdots \ten e_{j_{r'}}
\]
where $\{ i_1, \dots, i_r \} = \iota_x((G_x)_{00})$  and $\{ j_1, \dots, j_{r'} \} = \iota_x((G_x)_{11})$.
Tensoring these maps together, we have
\[
  p_1 \ten \cdots \ten p_k \colon c_{\iota^1(1)} \ten \cdots \ten c_{\iota^1(n)} \to e_{\iota^0(1)} \ten \cdots \ten e_{\iota^0(m)}.
\]
Then $\mu_G(p)$ is defined as the following composite:
\[ \begin{tikzcd}[column sep=2cm]
c_{\iota^1(1)} \ten \cdots \ten c_{\iota^1(n)} \rar{p_1 \ten \cdots \ten p_k} & e_{\iota^0(1)} \ten \cdots \ten e_{\iota^0(m)} \dar["\widehat{\iota^0}"] \\
c_1 \ten \cdots \ten c_n \uar["(\widehat{\iota^1})^{-1}"'] \rar[dashed,"\mu_G(p)"] &  e_1 \ten \cdots \ten e_m
\end{tikzcd} \]
\end{construction}

As a special case, we have the following.

\begin{lemma}\label{lem isomorphism}
Let $G \in \LL_1$ be a graph with both legs isomorphisms, as follows.
\[
\begin{tikzcd}[sep=tiny]
\un \ar[dr,"\alpha"',"\cong"] & & \un \ar[dl,"\beta","\cong"'] \\
& \un 
\end{tikzcd}
\]
Suppose $c_1, \dots, c_n$ are connected objects of $C$ and $p\in P_G$ is the element which goes to $(\id_{c_1}, \id_{c_2}, \dots, \id_{c_n})$ under the Segal map $ P_G \to \prod_{i=1}^n P_{\mathfrak{c}_{11}}$.
Then $\mu_G(p)$ is the following composite
\[ 
\pushQED{\qed}
\begin{tikzcd}
c_{\alpha^{-1}(1)} \ten \cdots \ten c_{\alpha^{-1}(n)} \rar{\widehat{\alpha}} &  c_{1} \ten \cdots \ten c_{n} \rar{\widehat{\beta}\vphantom{\beta}^{-1}} & c_{\beta^{-1}(1)} \ten \cdots \ten c_{\beta^{-1}(n)}.
\end{tikzcd} 
\qedhere
\popQED
\] 
\end{lemma}

\begin{lemma}\label{lemma mu splitting}
Suppose that $G$ and $G'$ are in $\LL_1$ and $i\colon G \to G+G'$, $j\colon G' \to G+G'$ are the inclusions.
If $p \in P_{G+G'}$, then $\mu_{G+G'}(p) = \mu_G(i^*p) \ten \mu_{G'}(j^*p)$. 
\end{lemma}
\begin{proof}
The canonical decomposition for $H= G+G'$
\[
  \sum_{z=1}^{k+\ell} H_z = \sum_{x=1}^k G_x + \sum_{y=1}^\ell G_y' \xrightarrow{\iota = \alpha + \beta} G+G' = H\]
  is the sum of the canonical decompositions (denoted by $\alpha$ and $\beta$) for $G$ and $G'$.
Write $\iota^1 \coloneqq d_1(\iota)$ and $\iota^0\coloneqq d_0(\iota)$ and likewise for $\alpha$ and $\beta$.
Then if $p_z = \iota_z^*p$, we have
\begin{align*}
  \mu_H(p) &= (\widehat{\iota^1})^{-1} \circ (p_1 \ten \cdots \ten p_k \ten p_{k+1} \ten \cdots \ten p_{k+\ell}) \circ \widehat{\iota^0} \\
  &= (\widehat{\alpha^1} \ten \widehat{\beta^1})^{-1} \circ (p_1 \ten \cdots \ten p_k \ten p_{k+1} \ten \cdots \ten p_{k+\ell}) \circ (\widehat{\alpha^0} \ten \widehat{\beta^0}) \\
  &= ((\widehat{\alpha^1})^{-1} \ten (\widehat{\beta^1})^{-1}) \circ (\alpha_1^*p \ten \cdots \ten \alpha_k^*p \ten \beta_1^*p \ten \cdots \ten \beta_\ell^*p) \circ (\widehat{\alpha^0} \ten \widehat{\beta^0}) \\
  &= ((\widehat{\alpha^1})^{-1} \circ (\alpha_1^*p \ten \cdots \ten \alpha_k^*p) \circ \widehat{\alpha^0}) \ten ( (\widehat{\beta^1})^{-1} \circ  ( \beta_1^*p \ten \cdots \ten \beta_\ell^*p) \circ \widehat{\beta^0}) \\
  &= \mu_G(i^*p) \ten \mu_{G'}(j^*p).  \qedhere
\end{align*}
\end{proof}

\begin{lemma}\label{lem height one iso mu}
Suppose $G$ and $G'$ are isomorphic height 1 level graphs, and let $\mu_G$ and $\mu_{G'}$ be as in \cref{constr muG}.
Let $f\colon G' \to G$ be an isomorphism. 
If $p\in P_G$, then 
\[
  \mu_{G'} f^*(p) = \widehat{d_0(f)}\vphantom{d_0(f)}^{-1} \circ \mu_G(p) \circ \widehat{d_1(f)}.
\]
\end{lemma}
In particular, if $f$ is a congruence then $\mu_{G'} f^*(p) = \mu_G(p)$, since $d_0(f)$ and $d_1(f)$ are identities.
\begin{proof}
If an isomorphism $f$ satisfies the lemma, then so does $f^{-1}$. Further, if two composable isomorphisms $f,f'$ satisfy the lemma, then so does their composition.
We use these facts to reduce the proof to several special cases.

If $G$ and $G'$ are connected, that is, if they are both a corolla $\mathfrak{c}$, then $\mu$ is the canonical inclusion into $\morcon C$ and the equation already appeared in \cref{def corolla iso}.
We can boost this up to the case where we have isomorphisms of connected graphs $f^i \colon H^i \to K^i$ for $1 \leq i \leq m$, and we consider their sum
\[
  f\colon G' = \sum_{i=1}^m H^i \to \sum_{i=1}^m K^i = G.
\]
(Note that $H^i = K^i$ since these are isomorphic corollas.)
Letting $\iota_i' \colon H^i \to G'$ and $\iota_i \colon K^i \to G$ be the inclusions 
and using \cref{lemma mu splitting} and the result for corollas, we have
\begin{align*}
\mu_{G'}(f^*(p)) &= (\iota_1')^*(f^*(p)) \ten \cdots \ten (\iota_m')^*(f^*(p)) \\
&= (f^1)^*(\iota_1^*(p)) \ten \cdots \ten (f^m)^*(\iota_m^*(p)) \\
&= [(\widehat{d_0f^1})^{-1} \circ \iota_1^*(p) \circ \widehat{d_1f^1}] \ten \cdots \ten [(\widehat{d_0f^m})^{-1} \circ \iota_m^*(p) \circ \widehat{d_1f^m}] \\
&= \widehat{d_0f}\vphantom{d_0f}^{-1} \circ \big( \iota_1^*(p) \ten \cdots \ten \iota_m^*(p)  \big) \circ \widehat{d_1f} \\
&= \widehat{d_0f}\vphantom{d_0f}^{-1} \circ \mu_G(p) \circ \widehat{d_1f}.
\end{align*}

On the other hand, suppose we have 
\[
  f\colon G' = \sum_{i=1}^m G'_i \to \sum_{k=1}^m G_i = G.
\]
with $f$ acting as an automorphism on $\um$ but as the identity on each component, that is, $f  \circ \iota'_i = \iota_{fi} \colon G_{fi} =  G'_i  \to G' \to G$.
Then 
\begin{align*}
  \mu_{G'} f^*(p) &= (\iota'_1)^* f^*(p) \ten \cdots \ten (\iota'_m)^* f^*(p) \\
  &= \iota_{f1}^*(p) \ten \cdots \ten \iota_{fm}^*(p) \\
  &= \widehat{d_0f}\vphantom{d_0f}^{-1} \circ (\iota_{1}^*(p) \ten \cdots \ten \iota_{m}^*(p)) \circ \widehat{d_1f}
\end{align*}
where these are now block permuations.

In the special case when $f$ is the canonical decomposition  from \cref{def canonical splitting}
\[
  \iota \colon G' = \sum G_x \to G,
\]
the required formula follows from the definition of $\mu_G$ in \cref{constr muG}.
That is, if $p_k = \iota_k^*p$ then we have 
\[
  \mu_G(p) = \widehat{d_0\iota} \circ (p_1 \ten \cdots \ten p_k) \circ \widehat{d_1\iota}\vphantom{d_1\iota}^{-1} =  \widehat{d_0\iota} \circ  (\mu_{G'} \iota^*(p)) \circ \widehat{d_1\iota}\vphantom{d_1\iota}^{-1}.
\]

We can combine all of these cases as follows
\[ \begin{tikzcd}
\sum\limits_{x=1}^m G_x' \rar \dar &  \sum\limits_{y=1}^m G_{fy}' \rar & \sum\limits_{y=1}^m G_y \dar \\
G' \ar[rr, "f"] & &  G. 
\end{tikzcd} \]
The vertical maps are canonical decompositions, the first map on the top rearranges factors, and the second map on top is a sum of isomorphisms between corollas.
The result now follows for $f$ by the facts in the first paragraph of the proof.
\end{proof}

\begin{proposition}\label{prop image in csp}
Let $G\in \LL_1$ and $\pi \colon C \to C(\ast)$. Then $\pi(\mu_G(p)) \in \WW_1 = \mor C(\ast)$ is the congruence class $[G]$ of $G$.
\end{proposition}
\begin{proof}
In the connected case, that is, when $G = \mathfrak{c}_{n,m}$ is a corolla, this is true by definition of $P_{\mathfrak{c}_{n,m}}$.
If $G$ is empty, then $\mu_G(p) = \mathbf{1}$, so the result follows.
If the result is true for $G$ and $G'$, then applying \cref{lemma mu splitting} we have 
\[ \pi \mu_{G+G'}(p) = \pi(\mu_G(i^*p) \ten \mu_{G'}(j^*p)) = \pi\mu_G(i^*p) \ten \pi\mu_{G'}(j^*p) = [G] \ten [G'] = [G + G']. \]

Finally, suppose the result is true for $G'$ and that $f\colon G' \to G$ is an isomorphism. 
Let $H \in \LL_3$ denote following the height 3 level graph
\[ \begin{tikzcd}[row sep=small, column sep=tiny]
G_{00} \ar[dr,"="] & & G'_{00} \ar[dl,"\cong"'] \ar[dr] & &  G'_{11} \ar[dl] \ar[dr,"\cong"]  & &  G_{11} \ar[dl,"="']\\
& G_{00} \ar[dr] & & G'_{01} \ar[dl,"\cong"']  \ar[dr,"\cong"] & & G_{11} \ar[dl] \\
& & G_{01} \ar[dr,"="] & & G_{01} \ar[dl,"="'] \\
& & & G_{01}
\end{tikzcd} \]
where the unlabelled maps are the structure maps.
This graph has the property that $d_1 d_2 H = G$ and $d_0 d_3 H = G'$,
while $d_2d_3G$ and $d_0d_0G$ are of the form in \cref{lem isomorphism}.
Let $p\in P_G$ be an element with $d_1(p) = c_1 \ten \cdots \ten c_n$ and $d_0(p) = e_1 \ten \cdots \ten e_m$.
Using the Segal condition, define an element $q$ of $P_H$ so that
\[ \begin{tikzcd}
P_H \rar{d_2d_3} & P_{d_2d_3H} \rar{\cong} & \prod\limits_{G_{00}} P_{\mathfrak{c}_{11}}
\end{tikzcd} \]
sends $q$ to $(\id_{c_1}, \dots, \id_{c_n})$, 
\[ \begin{tikzcd}
P_H \rar{d_0d_0} & P_{d_0d_0H} \rar{\cong} & \prod\limits_{G_{11}} P_{\mathfrak{c}_{11}}
\end{tikzcd} \]
sends $q$ to $(\id_{e_1}, \dots, \id_{e_m})$,
and $d_0d_3 \colon P_H \to P_{G'}$ sends $q$ to $f^*(p)$.
By \cref{lem isomorphism}, 
\[
  \mu_{d_2d_3H}(d_2d_3q) = \widehat{d_1(f)}\vphantom{\widehat{d_1(f)}}^{-1} \qquad \& \qquad \mu_{d_0d_0H}(d_0d_0q) = \widehat{d_0(f)},
\]
and notice that $\pi$ sends these to $[d_2d_3 H]$ and $[d_0d_0H]$, respectively.
By \cref{lem height one iso mu} we also have
\[
  \mu_G(p) = \widehat{d_0(f)} \circ \mu_{G'} f^*(p) \circ \widehat{d_1(f)}\vphantom{\widehat{d_1(f)}}^{-1},
\]
which is sent to \[
  [d_0d_0H] \circ [d_0d_3H] \circ [d_2d_3 H]
\]
by $\pi$, since our assumption is $\pi(\mu_{G'} f^*(p)) = [G']$.
But this element is equal to $[d_1d_2 H] = [G]$.
Thus $\pi(\mu_G(p)) = [G]$.

Since every graph is isomorphic to a finite sum of connected graphs, the general result follows.
\end{proof}

\begin{definition}[Inner face in lowest dimension]\label{def d one}
Suppose $G$ is a height 2 level graph.
If $G$ is connected, then $d_1G$ is a corolla. 
In this case we get a map $P_G \to P_{d_1G}$ by using composition in $C$.
\[ \begin{tikzcd}
P_G \rar{\cong} \dar[dashed] & P_{d_2G} \times_{P_{d_0d_2G}} P_{d_0G} \rar & \mor C \times_{\ob C} \mor C \dar{\circ} \\
P_{d_1G} \rar[hook] \dar & \morcon C \rar[hook] \dar & \mor C \dar \\
\{ d_1G \} \rar[hook] & \{ \mathfrak{c}_{p,q} \}_{p,q} \rar[hook] & \mor C(\ast)
\end{tikzcd} \]
For an arbitrary height 2 level graph, define $d_1 \colon P_G \to P_{d_1G}$ by examining the following square.
\[ \begin{tikzcd}
\sum\limits_{x\in G_{02}} d_1 G_x \dar{d^1} \rar & d_1 G \dar{d^1} \\
\sum\limits_{x\in G_{02}} G_x \rar & G \\
\end{tikzcd} \]
Here, each $G_x$ is connected and the bottom map is an isomorphism.
Then applying $P$ we see $d_1$ is the dashed map in the following.
\[ \begin{tikzcd}
 P_G \rar{\cong} \dar[dashed] & \prod\limits_x P_{G_x} \dar{\prod d_1} \\
 P_{d_1G} \rar{\cong} & \prod\limits_x P_{d_1G_x}
 \end{tikzcd} \] 
\end{definition}

\begin{lemma}\label{fiber maps lem}
Suppose $f \colon G' \to G$ is an isomorphism between connected graphs in $\LL_2$ and $d_1(f) = g \colon d_1 G' \to d_1 G$.
Then the diagram
\[ \begin{tikzcd}
P_G \rar{f^*} \dar{d_1} & P_{G'} \dar{d_1}  \\
P_{d_1 G} \rar{g^*} & P_{d_1G'}
\end{tikzcd} \]
commutes.
\end{lemma}
\begin{proof}
Write 
\begin{align*}
\alpha \coloneqq d_2(f) \colon H' = d_2 G' &\to d_2 G = H \\
\beta \coloneqq d_0(f) \colon K' = d_0 G' &\to d_0 G = K
\end{align*}
for the induced isomorphisms between height 1 level graphs.
For $i=0,1$, we also write $\alpha_i \coloneqq d_i(\alpha)$ and $\beta_i \coloneqq d_i(\beta)$.
If $p\in P_G$, then by definition 
\[
  d_1(p) = \mu_K(d_0p) \circ \mu_H(d_2p)
\]
and
\[
  d_1(f^*p) = \mu_{K'}(d_0f^*p) \circ \mu_{H'}(d_2f^*p) = \mu_{K'} (\beta^* d_0p) \circ \mu_{H'} (\alpha^* d_2p).
\]
By \cref{lem height one iso mu} this means 
\[
  d_1(f^*p) = \hat\beta_0^{-1} \circ \mu_{K} (d_0p) \circ \hat\beta_1 \circ \hat\alpha_0^{-1} \circ \mu_{H} (d_2p) \circ \hat\alpha_1.
\]
Now $\alpha_0 = d_0 d_2 (f) = d_1 d_0 (f) = \beta_1$, so 
\[
  d_1(f^*p) = \hat\beta_0^{-1} \circ \mu_{K} (d_0p) \circ \mu_{H} (d_2p) \circ \hat\alpha_1  = \hat\beta_0^{-1} \circ d_1(p) \circ \hat\alpha_1.
\]
The result now follows from \cref{def corolla iso} since $\alpha_1 = d_1 d_2(f) = d_1 d_1(f) = d_1(g)$ and $\beta_0 = d_0 d_0(f) = d_0 d_1(f) = d_0(g)$, implying $\hat\beta_0^{-1} \circ d_1(p) \circ \hat\alpha_1 = g^* d_1(p)$.
\end{proof}

\begin{proposition}\label{prop appropriate end behavior done}
If $G\in \LL_2$, then 
\begin{enumerate}
\item $d_0 d_1 \colon P_G \to P_{d_1G} \to P_{d_0d_1G}$ is equal to $d_0 d_0 \colon P_G \to P_{d_0G} \to P_{d_0d_0G}$.
\item $d_1 d_1 \colon P_G \to P_{d_1G} \to P_{d_1d_1G}$ is equal to $d_1 d_2 \colon P_G \to P_{d_2G} \to P_{d_2d_1G}$.
\end{enumerate}
\end{proposition}
\begin{proof}
We prove only the first equality, as the second is dual.
Suppose $G$ is connected.
\[ \begin{tikzcd}
& P_{d_0G} \rar[hook] \ar[rr,"d_0", bend left=15] & \mor C \ar[dr,"\text{target}" description] & P_{d_0d_0G} \dar[hook']\\
P_G \rar["\cong","d_2 \times d_0"'] \dar[dashed,"d_1"] \ar[ur,"d_0"] & P_{d_2G} \times_{P_{d_0d_2G}} P_{d_0G} \rar \uar & \mor C \times_{\ob C} \mor C \dar{\circ} \uar & \ob C \ar[ddl, bend left,"="] \\
P_{d_1G} \rar[hook] \dar["d_0"] & \morcon C \rar[hook] & \mor C \dar{\text{target}} \ar[ur,"\text{target}" description] \\
P_{d_0d_1G} \ar[rr,hook] & & \ob C
\end{tikzcd} \]
The above diagram commutes, where the two upward maps are projections onto the second factor, so $d_0d_1 = d_0d_0$ in this case.
The general case follows from the connected case:
\[ \begin{tikzcd}
 P_G \rar{\cong} \ar[rr,bend left,"\id"] \dar[dashed,"d_1"] & \prod\limits_x P_{G_x} \dar{\prod d_1} \ar[dd, bend left=60, "\prod d_0d_0"] & P_G \lar["\cong" swap] \ar[dd, bend left=30, "d_0d_0"] \\
 P_{d_1G} \rar{\cong} \dar["d_0"] & \prod\limits_x P_{d_1G_x} \dar["\prod d_0"] \\
 P_{d_0d_1G} \rar{\cong} \ar[rr,bend right,"\id"]& \prod\limits_x P_{d_0d_1G_x} & P_{d_0d_0G} \lar["\cong" swap]
 \end{tikzcd} \] 
\end{proof}

Now that we have defined $d^1 \colon P_G \to P_{d_1G}$ for each $G\in \LL_2$ and shown it is compatible at the ends, we can use the Segal condition to define inner faces at arbitrary height graphs.
It is convenient to use the following uniform notation for the generating outer face maps.

\begin{convention}
If $X$ is a simplicial set, we write $d_\bot \colon X_n \to X_{n-1}$ for the bottom face map $d_0$ and $d_\top \colon X_n \to X_{n-1}$ for the top face map $d_n$.
We use the same symbols when $G$ is a height $n$ level graph to write $d_\bot = d_0 \colon P_G \to P_{d_\bot G} = P_{d_0 G}$ and $d_\top = d_n \colon P_G \to P_{d_\top G} = P_{d_n G}$.
\end{convention}

If $G$ is a height $n$ level graph, then for $0 < i < n$ we can define $d_i$ by means of the following diagram, which is well-formed by \cref{prop appropriate end behavior done}.
\begin{equation}\label{diagram face} \begin{tikzcd}
P_G \rar{\cong} \dar[dashed, "d_i"] & 
P_{d_\top^{n-i+1} G} \underset{P_{d_\bot^{i-1}d_\top^{n-i+1}G}}{\times}
P_{d_\bot^{i-1}d_\top^{n-i-1}G} \underset{P_{d_\bot^{i+1}d_\top^{n-i-1}G}}{\times}
P_{d_\bot^{i+1} G} \dar{\id \times d_1 \times \id} \\
P_{d_iG} \rar{\cong} & 
P_{d_\top^{n-i}d_iG} \underset{P_{d_\bot^{i-1}d_\top^{n-i} d_iG}}{\times}
P_{d_\bot^{i-1} d_\top^{n-i-1} d_iG} \underset{P_{d_\bot^{i}d_\top^{n-i-1}d_iG}}{\times}
P_{d_\bot^id_iG} 
\end{tikzcd} \end{equation}
In other words, inner face maps are defined as follows.

\begin{definition}[Inner face operators]\label{def inner face}
If $G \in \LL_n$, $0 < i < n$, and $x\in P_G$, then $d_i(x) \in P_{d_iG}$ is the unique element so that 
\begin{itemize}
\item $d_\top^{n-i}(d_i x) = d_\top^{n-i+1}(x)$
\item $d_\bot^{i-1}d_\top^{n-i-1}(d_i x) = d_1 (d_\bot^{i-1} d_\top^{n-i-1} x)$, and
\item $d_\bot^{i} (d_i x) = d_\bot^{i+1}(x)$.
\end{itemize}
\end{definition}

In order to utilize this definition effectively, we now examine how an inner face $d_i$ interacts with arbitrary inert simplicial operators.
The identity in the following lemma also holds in an arbitrary simplicial object.

\begin{lemma}\label{lem inner face v inert}
Suppose $G$ is a height $n$ level graph and $x\in P_G$.
If $k+\ell \leq n-1$, then
\[
  d_\bot^k d_\top^\ell d_i(x) = \begin{cases}
  d_\bot^k d_\top^{\ell+1}(x) & n-\ell \leq i \\
  d_{i-k} d_\bot^k d_\top^\ell(x) & k+1 \leq i \leq n- \ell -1 \\
  d_\bot^{k+1} d_\top^\ell(x) & i \leq k.
  \end{cases}
\]
\end{lemma}
\begin{proof}
If $i \geq n- \ell$, then by \cref{def inner face}
\[
  d_\top^\ell d_i(x) = d_\top^{\ell-n+i} d_\top^{n-i} d_i(x) = d_\top^{\ell-n+i} d_\top^{n-i+1}(x) = d_\top^{\ell+1}(x),
\]
while if $i \leq k$ then 
\[
  d_\bot^k d_i(x) = d_\bot^{k-i} d_\bot^i d_i(x) = d_\bot^{k-i} d_\bot^{i+1}(x) = d_\bot^{k+1} (x),
\]
so the first and last options hold.
The interesting case is the second option.
Notice that $0 < i-k < s$, and we use the definition of the inner face map $d_{i-k}$ from \cref{def inner face}.
We make the following three computations 
\begin{align*}
d_\top^{(n-k-\ell)-(i-k)} ({ \color{purple} d_\bot^k d_\top^\ell d_i x}) 
&= d_\bot^k d_\top^{n-i} d_i(x) \\ &= d_\bot^k d_\top^{n-i+1} x = d_\top^{(n-k-\ell)-(i-k)+1} ({\color{blue} d_\bot^k d_\top^\ell(x)})
\end{align*}
\begin{align*}
  d_\bot^{(i-k)-1} d_\top^{(n-k-\ell)-(i-k)-1} ({ \color{purple} d_\bot^k d_\top^\ell d_i x})
&=
d_\bot^{i-1} d_\top^{n-i-1} d_i(x) 
\\
&= d_1(d_\bot^{i-1} d_\top^{n-i-1}x) \\
&= d_1(d_\bot^{(i-k)-1} d_\top^{ (n-k-\ell) - (i-k) - 1}({\color{blue} d_\bot^k d_\top^\ell(x)})
\end{align*}
\begin{align*}
d_\bot^{i-k} ({ \color{purple} d_\bot^k d_\top^\ell d_i x})  
&=
d_\top^\ell d_\bot^i d_i x \\
&= d_\top^\ell d_\bot^{i+1} x 
= d_\bot^{(i-k)+1} ({\color{blue} d_\bot^k d_\top^\ell(x)})
\end{align*}
and conclude that the element $d_\bot^k d_\top^\ell d_i x$ in $P_{d_\bot^k d_\top^\ell d_i G} = P_{d_{i-k} d_\bot^k d_\top^\ell G}$ is equal to $d_{i-k} (d_\bot^k d_\top^\ell x)$.
\end{proof}

\begin{proposition}\label{fiber maps prop}
Suppose $f \colon G' \to G$ is a map in $\LL_n$, and $0 < i < n$.
Let $d_i(f) = g \colon d_i G' \to d_i G$.
Then the diagram
\[ \begin{tikzcd}
P_G \rar{f^*} \dar{d_i} & P_{G'} \dar{d_i}  \\
P_{d_i G} \rar{g^*} & P_{d_iG'}
\end{tikzcd} \]
commutes.
\end{proposition}
\begin{proof}
By \cref{def inner face}, since $f$ and $g$ are inert, it suffices to prove the result when $n=2$ and $i=1$.
Further, the diagram commutes if and only if the outer rectangle in
\[ \begin{tikzcd}
P_G \rar{f^*} \dar{d_1} & P_{G'} \dar{d_1} \rar & P_{G'_x} \dar{d_1} \\
P_{d_1 G} \rar{g^*} & P_{d_1G'} \rar & P_{d_1G'_x}
\end{tikzcd} \]
commutes for every $x\in G'_{02} = (d_1G')_{01}$, since $d_1$ is defined as a product over the connected pieces.
Since the cube
\[ \begin{tikzcd}[sep=small]
d_1 G_x' \ar[rr] \ar[dd] \ar[dr] & & d_1 G_{fx} \ar[dd] \ar[dr] \\
& d_1 G' \ar[rr, crossing over] \ar[dd]& & d_1G \ar[dd]\\
G_x' \ar[rr]  \ar[dr] & & G_{fx} \ar[dr]  \\
& G' \ar[from=uu, crossing over] \ar[rr] &  & G
\end{tikzcd} \]
commutes, it is enough to show that 
\[ \begin{tikzcd}
P_G \rar \dar{d_1} & P_{G_{fx}} \dar{d_1} \rar & P_{G'_x} \dar{d_1} \\
P_{d_1 G} \rar & P_{d_1G_{fx}} \rar & P_{d_1G'_x}
\end{tikzcd} \]
commutes.
The left square commutes by definition of $d_1$ on disconnected graphs, and $G_x' \to G_{fx}$ is an isomorphism since any map in $\LL_n$ between connected graphs is an isomorphism. 
Thus it suffices to establish the result only in the case when $f \colon G' \to G$ is an isomorphism between height 2 connected graphs, which we already did in \cref{fiber maps lem}.
\end{proof}

\begin{lemma}\label{lem d1 mu thing}
If $G \in \LL_2$, then
\[ \begin{tikzcd}
P_G \rar["d_\top \times d_\bot"', "\cong"] \dar["d_1"] & P_{d_\top G} \times_{P_{d_\top d_\bot G}} P_{d_\bot G} \rar{\mu_{d_\top G} \times \mu_{d_\bot G}} &[+1cm] \mor C \times_{\ob C} \mor C = NC_2 \dar{\circ} \\
P_{d_1 G} \ar[rr,"\mu_{d_1G}"] & & \mor C  = NC_1
\end{tikzcd} \]
commutes.
\end{lemma}
\begin{proof}
The result holds when $G$ is connected; in fact, it is the definition of $d_1$.
In a moment we will induct on the number of connected components of $G$.
First, suppose the result is true for $G'$ and that $f\colon G' \to G$ is an isomorphism.
Then
\[
  \mu_{d_1G'} (d_1f^*p) = \mu_{d_\bot G'}(d_\bot f^*p) \circ \mu_{d_\top G'} (d_\top f^*p).
\]
By \cref{fiber maps prop} on the left and the fact that $P$ is a presheaf on $\LL_\intrm$, the preceding equation can be written as
\[
  \mu_{d_1G'} ((d_1f)^*d_1p) = \mu_{d_\bot G'}((d_\bot f)^* d_\bot p) \circ \mu_{d_\top G'} ((d_\top f)^* d_\top p).
\]
We now apply \cref{lem height one iso mu} to see that
\[
  \mu_{d_1G'} ((d_1f)^*d_1p) = \widehat{d_0d_1f}\vphantom{d_0d_1f}^{-1} \circ \mu_{d_1G} (d_1 p) \circ \widehat{d_1 d_1f}
\]
and 
\begin{align*}
& \mu_{d_\bot G'}((d_\bot f)^* d_\bot p) \circ \mu_{d_\top G'} ((d_\top f)^* d_\top p)  \\
=& 
  \widehat{d_0 d_\bot f}\vphantom{d_0 d_\bot f}^{-1} \circ \mu_{d_\bot G}(d_\bot p) \circ \widehat{d_1 d_\bot f} \circ \widehat{d_0d_\top f}\vphantom{d_0d_\top f}^{-1} \circ \mu_{d_\top G} (d_\top p) \circ \widehat{d_1d_\top f} \\
=& 
  \widehat{d_0 d_\bot f}\vphantom{d_0 d_\bot f}^{-1} \circ \mu_{d_\bot G}(d_\bot p) \circ  \mu_{d_\top G} (d_\top p) \circ \widehat{d_1d_\top f},
\end{align*}
using that $d_0d_\top = d_0d_2 = d_1d_0 = d_1 d_\bot$.
But also $d_0 d_\top = d_0 d_1$ and $d_1 d_\top = d_1 d_2 = d_1 d_1$, so we see that $\mu_{d_1G} (d_1 p) = \mu_{d_\bot G}(d_\bot p) \circ  \mu_{d_\top G} (d_\top p)$.

Now suppose the result is known for graphs with strictly fewer than $n$ connected components, and let $H+K \to G$ be an isomorphism where $H$ and $K$ are non-empty graphs.
By the preceding paragraph, it is enough to prove the result for the graph $H+K$. 
Let $h \colon H \to H+K$ and $k\colon K \to H+K$ be the inclusions.
Write $h_i = d_i(h)$ and $k_i = d_i(k)$.
Then using the induction hypothesis and \cref{lemma mu splitting}, we have
\begin{align*}
 &  \mu_{d_1(H+K)} (d_1p) \\
&= \mu_{d_1H} (h_1^*d_1p) \ten \mu_{d_1K} (k_1^*d_1p) \\
&= \mu_{d_1H} (d_1h^*p) \ten \mu_{d_1K} (d_1k^*p) \\
&= [(\mu_{d_0H} (d_0h^*p)) \circ (\mu_{d_2H} (d_2h^*p))]
\ten [(\mu_{d_0K} (d_0k^*p)) \circ (\mu_{d_2K} (d_2k^*p))] \\
&= [(\mu_{d_0H} (d_0h^*p)) \ten (\mu_{d_0K} (d_0k^*p))] \circ [
(\mu_{d_2H} (d_2h^*p)) \ten  (\mu_{d_2K} (d_2k^*p))] \\
&= [(\mu_{d_0H} (h_0^*d_0p)) \ten (\mu_{d_0K} (k_0^*d_0p))] \circ [
(\mu_{d_2H} (h_2^*d_2p)) \ten  (\mu_{d_2K} (k_2^*d_2p))] \\
&= (\mu_{d_0(H+K)} (d_0p)) \circ 
(\mu_{d_2(H+K)} (d_2p)). \qedhere
\end{align*}
\end{proof}

\begin{lemma}\label{close face maps lem}
If $G$ is a height 3 level graph and $x\in P_G$, then $d_1 d_2(x) = d_1 d_1 (x)$.
\end{lemma}
\begin{proof}
For space reasons, in the following two diagrams we do not notate the fiber products, nor do we label the $\mu$.
We use the notation $NC_3 = \mor C \times_{\ob C} \mor C \times_{\ob C} \mor C$ so that $d_2 = (\id \times \circ)$ and $d_1 = (\circ \times \id)$.
In each diagram the two rectangles commute by \cref{lem d1 mu thing}.
\[ \begin{tikzcd}
P_G \rar["d_\top^2 \times d_\bot"] \dar["d_2"] & P_{d_\top^2 G} \times P_{d_\bot G} \dar["\id \times d_1"] \rar["\id \times (d_\top \times d_\bot)"] &[+0.5cm] P_{d_\top^2 G} \times P_{d_\top d_\bot G} \times P_{d_\bot^2 G} \rar["\mu^{\times 3}"] & NC_3 \dar["d_2"] \\
P_{d_2G} \rar["d_\top \times d_\bot"] \dar["d_1"] & P_{d_\top^2 G} \times P_{d_1 d_\bot G} \ar[rr,"\mu^{\times 2}"] & & NC_2 \dar["d_1"] \\
P_{d_1d_2 G} \ar[rrr,"\mu"] & & & NC_1
\end{tikzcd} \]

\[ \begin{tikzcd}
P_G \rar["d_\top \times d_\bot^2"] \dar["d_1"] & P_{d_\top G} \times P_{d_\bot^2 G} \dar["d_1 \times \id"] \rar["(d_\top \times d_\bot) \times \id"] &[+0.5cm] P_{d_\top^2 G} \times P_{d_\top d_\bot G} \times P_{d_\bot^2 G} \rar["\mu^{\times 3}"] & NC_3 \dar["d_1"] \\
P_{d_1G} \rar["d_\top \times d_\bot"] \dar["d_1"] & P_{d_1d_\top G} \times P_{d_\bot^2 G} \ar[rr,"\mu^{\times 2}"] & & NC_2 \dar["d_1"] \\
P_{d_1d_1 G} \ar[rrr,"\mu"] & & & NC_1
\end{tikzcd} \]
If $G$ is connected, then so are $d_1d_2G$ and $d_1d_1G$, and the bottom map of both diagrams is an inclusion, so the result holds in the connected case.
The general case follows from the connected case and \cref{fiber maps prop}.
\end{proof}

\begin{proposition}\label{prop face relation}
If $G$ is a height $n$ level graph, $x\in P_G$, and $0 \leq i < j \leq n$, then $d_i d_j(x) = d_{j-1} d_i(x)$.
\end{proposition}
\begin{proof}
If $i \leq j-2$, then using \cref{lem inner face v inert} one can compute that $d_\bot^m d_\top^{n-3-m} d_id_j$ and $d_\bot^m d_\top^{n-3-m} d_{j-1}d_i$ are both equal to
\[
  \begin{cases}
    d_\bot^m d_\top^{n-1-m} & m \leq i-2 \\
    d_1 d_\bot^m d_\top^{n-2-m} & m = i-1 \\
    d_\bot^{m+1}d_\top^{n-2-m} & i \leq m \leq j-3 \\
    d_1 d_\bot^{m+1} d_\top^{n-3-m} & m = j-2 \\
    d_\bot^{m+2} d_\top^{n-3-m} & m\geq j-1.
  \end{cases}
\]
By the Segal condition, this shows that $d_id_j(x) = d_{j-1}d_i(x)$ in this case.

It remains to consider the case when $i = j-1$.
In this case we still have for $m \leq i-2$ that \[ d_\bot^m d_\top^{n-3-m} d_id_{i+1} = d_\bot^m d_\top^{n-1-m} = d_\bot^m d_\top^{n-3-m} d_id_i\] and for $m \geq i$ that 
\[ d_\bot^m d_\top^{n-3-m} d_id_{i+1} = d_\bot^{m+2} d_\top^{n-3-m} = d_\bot^m d_\top^{n-3-m} d_id_i.\]
For the remaining value $m=i-1 = j-2$, we must verify that $d_\bot^{i-1} d_\top^{n-2-i} d_id_{i+1} = d_\bot^{i-1} d_\top^{n-2-i} d_id_i$.
By \cref{lem inner face v inert} we have $d_\bot^{i-1} d_\top^{n-2-i} d_id_{i+1} = d_1 d_2 d_\bot^{i-1} d_\top^{n-2-i}$ and $d_\bot^{i-1} d_\top^{n-2-i} d_id_i = d_1 d_1 d_\bot^{i-1} d_\top^{n-2-i}$.
These are equal by \cref{close face maps lem}.
\end{proof}

\begin{definition}[Degeneracy in lowest dimension]\label{def s zero}
If $\un$ is height $0$ level graph, define
$s_0 \colon P_{\un} \to P_{s_0(\un)}$ by
declaring that 
\[
  P_{\un} \to P_{s_0(\un)} \xrightarrow{\cong} \prod_{i=1}^n P_{\mathfrak{c}_{11}}
\]
takes $c_1 \ten \dots \ten c_n$ to $(\id_{c_1}, \dots, \id_{c_n})$.
\end{definition}

\begin{lemma}\label{lem degen low end}
If $x\in P_{\un}$, then $d_0s_0(x) = x = d_1 s_0(x)$.
\end{lemma}
\begin{proof}
For $i=0,1$ the full composite of the diagram
\[
\begin{tikzcd}
  P_{\un} \rar{s_0} & P_{s_0(\un)} \rar{\cong} \dar{d_i} & \prod\limits_{i=1}^n P_{\mathfrak{c}_{11}} \dar{\prod d_i} \\
  & P_{\un} \rar{\cong} & \prod\limits_{i=1}^n P_{\underline{1}}
\end{tikzcd}
\]
takes $c_1 \ten \cdots \ten c_n$ to $(c_1, \dots, c_n)$.
Hence $d_is_0 = \id$.
\end{proof}

\begin{lemma}\label{lem s0 with fiber maps}
If $f \colon \un \to \um$ is a map in $\LL_0$ and $g = s_0(f)$, then \[s_0 f^* = g^* s_0 \colon P_{\um} \to P_{s_0(\un)}.\]
\end{lemma}
\begin{proof}
By Segality, it is enough to show that the the equality holds after postcomposing with $P_{s_0(\un)} \to P_{\mathfrak{c}_{11}}$ for each of the $n$ inert maps $\mathfrak{c}_{11} \rat s_0(\un)$.
Since the two squares on the right in the following diagram commute,
\[ \begin{tikzcd}
P_{\um} \rar{f^*} \dar{s_0} & P_{\un} \dar{s_0} \rar & \prod\limits_{i=1}^n P_{\mathfrak{e}} \dar \rar & P_{\mathfrak{e}} \dar   \\
P_{s_0 (\um)} \rar{g^*} & P_{s_0(\un)} \rar & \prod\limits_{i=1}^n P_{\mathfrak{c}_{11}} \rar & P_{\mathfrak{c}_{11}}
\end{tikzcd} \]
it thus enough to prove the result for $\un = \underline{1}$.
But if $f \colon \underline{1} \to \um$ hits an element $k$, then $s_0 f^*(c_1 \ten \cdots \ten c_m) = s_0(c_k) = \id_{c_k}$, and by \cref{def s zero} we have $g^* s_0 (c_1 \ten \cdots \ten c_m) = \id_{c_k}$.
\end{proof}

\begin{definition}[Degeneracy operators]\label{def degeneracies}
If $G$ is a height $n$ level graph, then $s_j \colon P_G \to P_{s_jG}$ is defined so that $s_j(x)$ is the unique element satisfying
\begin{itemize}
\item $d_\top^{n+1-j}(s_j x) = d_\top^{n-j}(x)$
\item $d_\bot^j d_\top^{n-j}(s_j x) = s_0 (d_\bot^j d_\top^{n-j} x)$, and
\item $d_\bot^{j+1} (s_j x) = d_\bot^j(x)$.
\end{itemize}
\end{definition}

In other words, $s_j$ is defined so the following diagram commutes.
\begin{equation}\label{diagram degeneracy} \begin{tikzcd}
P_G \rar{\cong} \dar[dashed, "s_j"] & 
P_{d_\top^{n-j} G} \underset{P_{d_\bot^j d_\top^{n-j}G}}{\times}
P_{d_\bot^j d_\top^{n-j}G} \underset{P_{d_\bot^j d_\top^{n-j}G}}{\times}
P_{d_\bot^j G} \dar{\id \times s_0 \times \id} \\
P_{s_jG} \rar{\cong} & 
P_{d_\top^{n+1-j}s_jG} \underset{P_{d_\bot^j d_\top^{n+1-j} s_jG}}{\times}
P_{d_\bot^j d_\top^{n-j} s_jG} \underset{P_{d_\bot^{j+1}d_\top^{n-j}s_jG}}{\times}
P_{d_\bot^{j+1}s_jG} 
\end{tikzcd} \end{equation}

In order to utilize this definition effectively, we now examine how a degeneracy $s_j$ interacts with arbitrary inert simplicial operators.
The identity in the following lemma also holds in an arbitrary simplicial object.

\begin{lemma}\label{lem degen v inert}
Suppose $G$ is a height $n$ level graph and $x\in P_G$.
If $k+\ell \leq n+1$, then
\[
  d_\bot^k d_\top^\ell s_j(x) = \begin{cases}
  d_\bot^k d_\top^{\ell-1}(x) & n-\ell + 1  \leq j \\
  s_{j-k} d_\bot^k d_\top^\ell(x) & k \leq j \leq n- \ell \\
  d_\bot^{k-1} d_\top^\ell(x) & j \leq k-1.
  \end{cases}
\]
\end{lemma}
\begin{proof}
If $j \geq n- \ell+1$, then by \cref{def degeneracies}
\[
  d_\top^\ell s_j(x) = d_\top^{\ell-n+j-1} d_\top^{n+1-j} s_j(x) = d_\top^{\ell-n+j-1}  d_\top^{n-j}(x) = d_\top^{\ell-1}(x),
\]
while if $j \leq k-1$ then 
\[
  d_\bot^k s_j(x) = d_\bot^{k-j-1} d_\bot^{j+1} s_j(x) = d_\bot^{k-j-1} d_\bot^j(x) = d_\bot^{k-1} (x),
\]
so the first and last options hold.
We now use the definition of degeneracy operators from \cref{def degeneracies}.
We make the following three computations 
\begin{align*}
d_\top^{(n-k-\ell)+1-(j-k)} ({ \color{purple} d_\bot^k d_\top^\ell s_j x}) 
&= d_\bot^k d_\top^{n-j+1} s_j(x) \\ &= 
d_\bot^k d_\top^{n-j} x = d_\top^{(n-k-\ell)-(j-k)} ({\color{blue} d_\bot^k d_\top^\ell(x)})
\end{align*}

\begin{align*}
  d_\bot^{j-k} d_\top^{(n-k-\ell)-(j-k)} ({ \color{purple} d_\bot^k d_\top^\ell s_j x})
&=
d_\bot^j d_\top^{n-j} s_j(x) 
\\
&= s_0(d_\bot^j d_\top^{n-j}x) \\
&= s_0(d_\bot^{j-k} d_\top^{ (n-k-\ell) - (j-k)}({\color{blue} d_\bot^k d_\top^\ell(x)})
\end{align*}

\begin{align*}
d_\bot^{j-k+1} ({ \color{purple} d_\bot^k d_\top^\ell s_j x})  
&=
d_\top^\ell d_\bot^{j+1} s_j x \\
&= d_\top^\ell d_\bot^j x 
= d_\bot^{j-k} ({\color{blue} d_\bot^k d_\top^\ell(x)})
\end{align*}
and conclude that the element $d_\bot^k d_\top^\ell s_j x$ in $P_{d_\bot^k d_\top^\ell s_j G} = P_{s_{j-k} d_\bot^k d_\top^\ell G}$ is equal to $s_{j-k} (d_\bot^k d_\top^\ell x)$.
\end{proof}

\begin{proposition}\label{fiber maps degen prop}
Suppose $f \colon G' \to G$ is a map in $\LL_n$, and $0 \leq j \leq n$.
Let $s_j(f) = g \colon s_j G' \to s_j G$.
Then the diagram
\[ \begin{tikzcd}
P_G \rar{f^*} \dar{s_j} & P_{G'} \dar{s_j}  \\
P_{s_j G} \rar{g^*} & P_{s_jG'}
\end{tikzcd} \]
commutes.
\end{proposition}
\begin{proof}
By Segality it is enough to show that $d_\bot^m d_\top^{n-m} s_j f^*(x) = d_\bot^m d_\top^{n-m} g^* s_j(x)$ for all $x\in P_G$ and each $0 \leq m \leq n$.
Set \[ h = d_\bot^m d_\top^{n-m} (g) = d_\bot^m d_\top^{n-m} s_j (f)  = \begin{cases}
  d_\bot^m d_\top^{n-m-1} (f) & m+1 \leq j \\
  s_0 d_\bot^m d_\top^{n-m} (f) & j = m \\
  d_\top^{m-1} d_\top^{n-m} (f) & j \leq m-1.
\end{cases} \]
By \cref{lem degen v inert}, for $x\in P_G$ we have
\[
d_\bot^m d_\top^{n-m} s_j f^*x = \begin{cases}
  d_\bot^m d_\top^{n-m-1} f^* x & m+1 \leq j \\
  s_0 d_\bot^m d_\top^{n-m} f^* x & j = m \\
  d_\top^{m-1} d_\top^{n-m} f^* x & j \leq m-1.
\end{cases}
\]
Since $P|_{\LL_\intrm}$ is a presheaf we have the first equality below, while the second follows from \cref{lem degen v inert}.
\[
  d_\bot^m d_\top^{n-m} g^* s_j x = h^* d_\bot^m d_\top^{n-m} s_j x =  \begin{cases}
  h^* d_\bot^m d_\top^{n-m-1} x & m+1 \leq j \\
  h^* s_0 d_\bot^m d_\top^{n-m} x & j = m \\
  h^* d_\top^{m-1} d_\top^{n-m} x & j \leq m-1
\end{cases}
\]
Again using that $P|_{\LL_\intrm}$ is a presheaf, we have the  desired equality for $m \neq j$.
When $m = j$, we have 
\[
   s_0 d_\bot^m d_\top^{n-m} f^* x = s_0 (d_\bot^m d_\top^{n-m}(f))^* d_\bot^m d_\top^{n-m} (x).
\] 
But 
\[
   s_0 (d_\bot^m d_\top^{n-m}(f))^* = (s_0 d_\bot^m d_\top^{n-m}(f))^* s_0 = h^* s_0
\]
by \cref{lem s0 with fiber maps}, so 
\[
   d_\bot^j d_\top^{n-j} s_j f^*(x) = d_\bot^j d_\top^{n-j} g^* s_j(x)
\]
holds as well.
\end{proof}

\begin{proposition}\label{prop degeneracies}
Suppose that $G$ is a height $n$ level graph and $x\in P_G$. 
If $i \leq j$, then $s_i s_j(x) = s_{j+1} s_i(x)$. 
\end{proposition}
\begin{proof}
We have $s_is_j G = s_{j+1}s_i G$ in $\LL_{n+2}$.
By Segality, it is enough to verify, for $0\leq m \leq n+1$ and $x\in P_G$, that the operator $d_\bot^m d_\top^{n+1-m}$ takes the two elements $s_is_j(x)$ and $s_{j+1}s_i(x)$ to the same element.
Applying \cref{lem degen v inert} several times, and also using \cref{lem degen low end}, one computes that $d_\bot^m d_\top^{n+1-m} s_i s_j (x)$ and $d_\bot^m d_\top^{n+1-m} s_{j+1} s_i (x)$ are both equal to 
\[
\begin{cases}
d_\bot^m d_\top^{n-m-1} x & m \leq i-1 \\
s_0 d_\bot^m d_\top^{n-m} x & m = i \\
d_\bot^{m-1} d_\top^{n-m} x & i+1 \leq m \leq j \\
s_0 d_\bot^{m-1} d_\top^{n+1-m}x & m = j+1 \\
d_\bot^{m-2} d_\top^{n+1-m} x & m \geq j+2.
\end{cases}
\]
Thus $s_is_j(x) = s_{j+1}s_i(x)$.
\end{proof}

\begin{proposition}\label{prop low level interchange}
If $G$ is a height 1 level graph and $x\in P_G$, then \[ d_1s_0 x = x = d_1s_1 x.\]
\end{proposition}
\begin{proof}
If we can prove the relation for a corolla $\mathfrak{c}$, then it will follow for an arbitrary $G \in \LL_1$ using \cref{fiber maps prop}, \cref{fiber maps degen prop}, and Segality.
Indeed, if $\iota \colon \mathfrak{c} \to G$ is a component of $G$, then for $x\in P_G$ and $k=0,1$, we have \[ \iota^* d_1 s_k x = d_1 (s_k \iota)^* s_k x = d_1 s_k \iota^* x = \iota^*x \]
using the corolla case, so by Segality $d_1 s_k x = x$.

Let $x \in P_{\mathfrak{c}}$ where $\mathfrak{c}$ is a corolla.
By \cref{def d one}, $d_1$ applied to the element $s_0x \in P_{s_0\mathfrak{c}}$ is
\[
  \mu_{d_0 s_0 \mathfrak{c}}(d_0 s_0 x) \circ \mu_{d_2 s_0 \mathfrak{c}} (d_2s_0 x)
\]
which is equal to 
\[
  \mu_{d_0 s_0 \mathfrak{c}}(x) \circ \mu_{d_2 s_0 \mathfrak{c}} (s_0d_1x) = x \circ \id = x
\]
by \cref{def degeneracies}.
Likewise, $d_1 s_1 x$ is 
\[
  \mu_{d_0 s_1 \mathfrak{c}}(d_0 s_1 x) \circ \mu_{d_2 s_1 \mathfrak{c}} (d_2s_1 x) = 
  \mu_{d_0 s_1 \mathfrak{c}}(s_0d_0 x) \circ \mu_{d_2 s_1 \mathfrak{c}} (x) = \id \circ x = x. \qedhere
\]
\end{proof}

The proof of the following lemma uses the same idea as the proofs of \cref{prop face relation} and \cref{prop degeneracies}. 

\begin{proposition}\label{simpop mixed}
If $G$ is a height $n$ level graph and $x\in P_G$, 
then
\[
  d_i s_j x = \begin{cases}
    s_{j-1} d_i x & i < j \\
    x & i = j, j+1 \\
    s_j d_{i-1} x & i  > j+1.
  \end{cases}
\]
\end{proposition}
\begin{proof}
We will apply the operator $d_\bot^m d_\top^{n-m-1}$ for $0\leq m \leq n-1$ to both sides of the equation, and show that our answers coincide. 
Since $P$ is Segal, that implies that equation holds.

Using \cref{lem inner face v inert}, \cref{lem degen v inert}, and \cref{prop low level interchange}, one computes that 
\[
d_\bot^m d_\top^{n-m-1} d_i s_j x = d_\bot^m d_\top^{n-m-1} x
\]
whenever $m \leq \min(i-2,j-1)$, $m\geq \max(i,j)$, $m=j=i-1$, or $m=i-1=j-1$.
If $i=j$ or $i=j+1$, this covers all possible values of $m$, so the formula holds in these cases.
Otherwise, one computes using \cref{lem inner face v inert} and \cref{lem degen v inert} that
\[
  d_\bot^m d_\top^{n-m-1} d_i s_j x = 
  \begin{cases}
    d_\bot^{m+1} d_\top^{n-m-2} x & i \leq m \leq j-2 \\
    d_1 d_\bot^m d_\top^{n-m-2} x & i-1 = m \leq j-2 \\
    s_0 d_\bot^{m+1} d_\top^{n-m-1} x & i \leq m = j-1 \\
    d_\bot^{m-1} d_\top^{n-m} x & j+1 \leq m \leq i-2 \\
    d_1 d_\bot^{m-1} d_\top^{n-1-m} x & j+1 \leq m = i-1 \\
    s_0 d_\bot^m d_\top^{n-m} x & m = j \leq i-2.
  \end{cases}
\]

When $i<j$, one computes
\[
  d_\bot^m d_\top^{n-m-1} s_{j-1} d_i x = 
  \begin{cases}
  d_\bot^m d_\top^{n-m-1}x & m\leq i-2 \text{ or } j \leq m \\
  d_1 d_\bot^m d_\top^{n-m-2}x & m = i-1 \\
  d_\bot^{m+1} d_\top^{n-m-2}x & i \leq m \leq j-2 \\
  s_0 d_\bot^{m+1} d_\top^{n-m-1} x & m = j-1
  \end{cases}
\]
which implies $d_i s_j x = s_{j-1} d_i x$ by the previous calculation.
When $i > j + 1$, one computes
\[
  d_\bot^m d_\top^{n-m-1} s_j d_{i-1}x =
  \begin{cases}
    d_\bot^m d_\top^{n-m-1}x & m \leq j-1 \text{ or } i \leq m \\
    s_0 d_\bot^m d_\top^{n-m}x & m=j \\
    d_\bot^{m-1} d_\top^{n-m}x & j+1 \leq m \leq i-2 \\
    d_1 d_\bot^{m-1} d_\top^{n-m-1}x & m = i-1,
  \end{cases}
\]
implying $d_i s_j x =  s_j d_{i-1}x$.
\end{proof}

\begin{theorem}
The $\LL_\intrm$-presheaf $P|_{\LL_\intrm} \in \seg(\LL_\intrm)$ extends to an $\LL$-presheaf $P\in \seg(\LL)$ with inner face maps and degeneracies defined as in \cref{def inner face} and \cref{def degeneracies}.
Further, the assignment $C \rightsquigarrow P$ is a functor $\slcc \to \seg(\LL)$.
\end{theorem}
\begin{proof}
Every map $G\to H$ in $\LL$ factors uniquely as follows:
\[
G \xrightarrow{f} \alpha^*H \xrightarrow{\alpha} H
\]
where $\alpha$ is a map in $\simpcat$ and $f \in \LL_n$ where $n$ is the height of $G$.
We have an existing map $f^* \colon P_{\alpha^*H} \to P_G$ using the $\LL_\intrm$-presheaf structure, and we define $\alpha^*\colon P_H \to P_{\alpha^*H}$ by decomposing into face and degeneracy operators; by \cref{prop face relation}, \cref{prop degeneracies}, and \cref{simpop mixed} the result does not depend on any choices of how to do this decomposition.
On such a composite, we define $(\alpha \circ f)^* \colon P_H \to P_G$ to be the composite $P_H \to P_{\alpha^*H} \to P_G$.
To see that composition is preserved, consider a pair of composable morphisms in $\LL$ presented in this way:
\[ \begin{tikzcd}
G \rar{f} \ar[drr,"hf"', bend right=15, dashed, color=blue]  & \alpha^* H \rar{\alpha} \ar[dr,"h"',color=blue,dashed] & H \rar{g} & \beta^*K \rar{\beta} & K 
\\
 &  & 
 {\color{blue}\alpha^*\beta^*K} \ar[ur,"\alpha"',dashed, color=blue] \ar[urr,"\beta \alpha"', bend right=15, dashed, color=blue] 
\end{tikzcd} \]
Then setting $h \coloneqq \alpha^*(g) \colon \alpha^*H \to \alpha^*\beta^*K$, the composite is presented as in the dashed maps.
Then \[ (\alpha \circ f)^* (\beta \circ g)^* = f^* \alpha^* g^* \beta^* = f^* h^* \alpha^* \beta^*
\]
where the last equality uses \cref{fiber maps prop}, \cref{fiber maps degen prop}, and that $P$ is an $\LL_\intrm$-presheaf.
But then $f^*h^* = (hf)^*$ since $P$ is an $\LL_\intrm$-presheaf and $\alpha^* \beta^* = (\beta \alpha)^*$ by \cref{prop face relation,prop degeneracies,simpop mixed}. 
Thus $(\alpha \circ f)^* (\beta \circ g)^* = ((\beta \circ g) \circ (\alpha \circ f))^*$, so $P$ is an $\LL$-presheaf.

All constructions from this section were functorial on maps in $\slcc$: the construction of the $\LL_\elrm$-presheaf, the right Kan extension $\pre(\LL_\elrm) \to \pre(\LL_\intrm)$, and finally the definition of the inner face operators (\cref{def d one,def inner face}) and degeneracy operators (\cref{def s zero,def degeneracies}).
Thus we have constructed a functor $\slcc \to \seg(\LL)$.
\end{proof}

\section{Comparison of constructions}
We now have two constructions: the first is the functor $\seg(\LL) \to \slcc$ that takes a presheaf $X$ to the image of the unique map $X \to \ast$ under the envelope functor $C$, and the second is the functor $\slcc \to \seg(\LL)$ that takes a strict labelled cospan category $\pi \colon C \to C(\ast)$ to the presheaf $P = P(\pi)$.
In this section, we will show that these are part of an equivalence of categories $\seg(\LL) \simeq \slcc$ (\cref{equivalence of presheaves and slcc}).
We first show that the composite $\slcc \to \seg(\LL) \to \slcc$ is isomorphic to the identity.

Let $\pi\colon C \to C(\ast)$ be a strict labelled cospan category, and $P = P(\pi)$ be the associated Segal $\LL$-presheaf.
Our goal is to compare $C$ with the envelope $CP$ of $P$.

The set of connected objects of $CP$ is $P_{\mathfrak{e}}$, which is exactly the set of connected objects of $C$. (See the beginning of \cref{sec from properad to lcc} and the beginning of \cref{subsec inert}).
Since these are both strict labelled cospan categories, the set of objects of each is the free monoid on this set, hence the sets of objects can be identified. Under this identification the tensor units coincide.
Moreover, the set of connected morphisms of $CP$ is $\sum P_{\mathfrak{c}_{n,m}}$, which is naturally in bijection with the set of connected morphisms of $C$ (each $P_{\mathfrak{c}_{n,m}}$ is a subset of $\morcon C$, see \eqref{diagram defining Pn and Pcmn} on page \pageref{diagram defining Pn and Pcmn}).
The abelian monoid $\hom(\mathbf{1},\mathbf{1})$ is the same in both categories, that is, it is freely generated by the identical sets $\homcon_C(\mathbf{1},\mathbf{1}) = P_{\mathfrak{c}_{00}} = \homcon_{CP}(\mathbf{1},\mathbf{1})$ of connected endomorphisms (see \cref{def lcc}\eqref{lcc: free gen}).

In \cref{constr muG}, we defined, for each $G\in \LL_1$, a function $\mu_G \colon P_G \to \mor C$.
We now identify the image of this function.

\begin{lemma}\label{injectivity after quotient}
If $G \in \LL_1$, then $\mu_G \colon P_G \to \mor C$ 
factors through $\overline{P}_G$ and the map $\overline{P}_G \to \mor C$ is injective.
\end{lemma}
\begin{proof}
If $f\colon H \xrightarrow{\equiv} G$ is a congruence between height 1 graphs, then by \cref{lem height one iso mu} we have $\mu_H f^*(p) = \mu_G(p)$ for all $p\in P_G$.
This implies that $\mu_G$ factors through $\overline{P}_G = \colim_{H \in \kong_{[G]}^{\oprm}} P_H$.
For injectivity, we may assume that $G$ has the property that all vertices without inputs or outputs come at the end; that is, we utilize a congruence
\[
H = \left( G' + \sum_{i=1}^{k-j} \mathfrak{c}_{00} \right) \xrightarrow{\equiv} G   
\] 
where $G'$ has $j$ vertices and is reduced, and $G' \to G$ is order-preserving at each level.
Now $\overline{P}_G \to \mor C$ is equal to $\overline{P}_H \to \mor C$, so we replace $G$ with $H$.

Let $p,p' \in P_G$, where $G \in \LL_1$ has $j$ vertices which have an input or output and $k-j$ vertices which have no input or output, and the former come before the latter in the order on $G_{01}$.
Assume $\mu_G(p) = \mu_G(p')$; we wish to show that $p$ and $p'$ represent the same element of $\overline{P}_G$.
We use similar notation as in \cref{constr muG}; specifically, we let
\[
  \iota = \sum \iota_x \colon \sum_{x=1}^k G_x \xrightarrow{\cong} G
\]
be the canonical splitting from \cref{def canonical splitting}, $p_x = \iota_x^*p$, $p_x' = \iota_x^*p'$, $\iota^0 = d_0(\iota)$, and $\iota^1 = d_1(\iota)$.
Then 
\begin{align*}
\mu (p) &= \widehat{\iota^0} \circ (p_1 \ten \cdots \ten p_k) \circ (\widehat{\iota^1})^{-1} \\
\mu (p') &= \widehat{\iota^0} \circ (p_1' \ten \cdots \ten p_k') \circ (\widehat{\iota^1})^{-1}.
\end{align*}
Our assumption was $\mu(p) = \mu(p')$, hence \[ p_1 \ten \cdots \ten p_k = p_1' \ten \cdots \ten p_k'\]
with all $p_x$, $p_x'$ connected morphisms of $C$.
By \cref{def lcc}\eqref{lcc: reduced and free} we have $p_1 \ten \cdots \ten p_j = p_1' \ten \cdots \ten p_j'$ and $p_{j+1} \ten \cdots \ten p_k = p_{j+1}' \ten \cdots \ten p_k'$.
By \eqref{lcc: pullback} of \cref{def lcc}, $p_x = p_x'$ for $0\leq x \leq j$.
In light of \cref{def lcc}\eqref{lcc: free gen}, there is an automorphism $\gamma'$ of the set $\{j+1, j+2, \dots, k\}$ so that $p_x = p'_{\gamma'(x)}$ for $j+1 \leq x \leq k$.
We extend $\gamma'$ to an automorphism $\gamma$ of $\uk$ by letting $\gamma(x) = x$ for $1 \leq x \leq j$, and this comprises a congruence $f \colon G \xrightarrow{\equiv} G$ with $f_{01} = \gamma$.
We next observe that $f^*(p') = p$.
By Segality, it is enough to show $\iota_x^* f^*(p') = \iota_x^*(p)$ for $1 \leq x \leq k$.
Notice that $\iota_x^* f^*(p') = \iota_{\gamma(x)}^*(p')$.
For $1 \leq x \leq j$, we have $\iota_{\gamma(x)}^*(p') = \iota_x^*(p') = p'_x = p_x = \iota_x^*(p)$.
For $j+1 \leq x \leq k$, we have $\iota_{\gamma(x)}^*(p') = \iota_{\gamma'(x)}^*(p') = p'_{\gamma(x)} = p_x = \iota_x^*(p)$.

We have shown that if $p,p'\in P_G$ are two elements with $\mu(p) = \mu(p')$, then they are identified in $\overline{P}_G$. Since $P_G \to \overline{P}_G$ is surjective, this proves that $\overline{P}_G \to \mor C$ is injective.
\end{proof}

As we saw at the beginning of the preceding proof, the functions $\mu_G \colon P_G \to \mor C$ give well defined functions $\overline{P}_w \to \mor C$ where $w\in \WW_1 = \mor C(\ast)$. 
Thus we can make the following definition.
\begin{definition}[Comparison of morphisms]
Let \[ \bar \mu \colon \mor CP = SP_1 = \sum_{w\in \WW_1} \overline{P}_w \to \mor C\]
be the function induced from the collection of functions $\mu_G$ (for $G\in \LL_1$).
\end{definition}

\begin{lemma}\label{lem source target}
The function $\bar \mu \colon \mor CP \to \mor C$ is compatible with sources and targets of morphisms.
\end{lemma}
\begin{proof}
For each graph $G\in \LL_1$, the outer square and all inner chambers except the inner rectangle of the following diagram are known to commute (including using \cref{constr muG}).
\[ \begin{tikzcd}
P_G \ar[rrrr,"d_1"] \ar[ddd] \ar[dr,two heads] \ar[drr,"\mu"] & & & &  P_{d_1G} \ar[dl,hook'] \ar[ddd,hook] \\
& \overline{P}_G \rar[hook] \ar[d,hook] & \mor C \rar{s} & \ob C \dar["="] \\ 
& \mor CP \ar[rr,"s"]\ar[dl, "="'] & & \ob CP \ar[dr,"="']
\\
SP_1 \ar[rrrr,"d_1"]  & & & &  SP_0 
\end{tikzcd} \]
It follows from surjectivity of $P_G \to \overline{P}_G$ that the inner rectangle commutes as well, so the function $\mor CP \to \mor C$ preserves sources of maps. Compatibility with targets is the evident modification of this argument.
\end{proof}

The next proof uses that $\bar \mu$ preserves identities between connected objects, which is immediate from \cref{def s zero}, since $s_0 \colon P_{\underline{1}} \to P_{s_0(\underline{1})} = P_{\mathfrak{c}_{11}}$ takes $c$ to $\id_c \in P_{\mathfrak{c}_{11}} \subseteq \mor C$.

\begin{lemma}\label{lem monoidal preservation}
The function $\bar \mu \colon \mor CP \to \mor C$ strictly preserves the monoidal product of morphisms, and also preserves the symmetry isomorphisms.
\end{lemma}
\begin{proof}
Recall from \cref{lem ten descends} and \cref{def tensor before S} that the tensor product of morphisms in $CP$ is induced as in the left rectangle in the following diagram, while the right square commutes by \cref{lemma mu splitting}.
\[ \begin{tikzcd}
SP_1 \times SP_1 \dar{\ten} \ar[rrr,"\bar \mu \times \bar \mu"', bend left=15] & \overline{P}_G \times \overline{P}_H \lar[hook'] & P_G \times P_H \lar[two heads] \rar{\mu_G \times \mu_H} &[+0.4cm] \mor C \times \mor C \dar{\ten} \\
SP_1 \ar[rrr,"\bar \mu", bend right=15] & \overline{P}_{G+H} \lar[hook'] & P_{G+H} \lar[two heads] \uar["\cong"'] \rar["\mu_{G+H}"] &  \mor C 
\end{tikzcd} \]
Thus $\bar \mu$ preserves monoidal product.

To see that $\bar\mu$ preserves the symmetry, it is enough to check that it is true for a pair of connected objects $c,d$.
Let $G$ be the following graph
\[ \begin{tikzcd}[sep=tiny]
\underline{2} \ar[dr,"="'] & & \underline{2} \ar[dl,"\tau"',"\cong"] \\
& \underline{2}
\end{tikzcd} \]
(using the notation of \cref{sec symmetry}, this means $G = \mathfrak{c}_{11} \twplus1 \mathfrak{c}_{11}$).
Let $p \in P_G$ be the unique element which is sent to $(\id_c, \id_d)$ under the isomorphism $P_G \to P_{\mathfrak{c}_{11}} \times P_{\mathfrak{c}_{11}}$.
By \cref{def rep twten} and \cref{def symmetry iso},
the symmetry $\tau_{c,d} \colon c \ten d \to d \ten c$ of $CP$ is the image of $p$ under $P_G \to \overline{P}_G \subset \mor CP$.
The definition of $\bar \mu$ implies that $\bar \mu(\tau_{c,d}) = \mu_G(p)$.
By \cref{lem isomorphism}, $\mu_G(p) \colon c \ten d \to d \ten c$ is the symmetry isomorphism of $C$. 
Hence $\bar\mu$ preserves the symmetry isomorphism. 
\end{proof}

\begin{lemma}\label{lem mu bar bijective}
The function $\bar \mu \colon \mor CP \to \mor C$ is bijective.
\end{lemma}
\begin{proof}
Suppose $G \not\equiv G'$, $p\in P_G$, and $p' \in P_{G'}$.
Then using \cref{prop image in csp} we have $\pi \mu_G(p) = [G] \neq [G'] = \pi \mu_{G'}(p')$ in $\mor C(\ast) = \WW_1$, hence $\mu_G(p) \neq \mu_{G'}(p')$ in $\mor C$.
Combining this with \cref{injectivity after quotient} we conclude that $\mor CP \to \mor C$ is injective.
We now turn to surjectivity.

We already know that $\bar \mu$ restricts to a bijection of connected morphisms $\morcon CP = \sum P_{\mathfrak{c}_{nm}} \to \morcon C$ and that $\hom_{CP}(\mathbf{1},\mathbf{1}) \to \hom_C(\mathbf{1}, \mathbf{1})$ is an isomorphism. 
From \cref{lem monoidal preservation} we know that $\bar\mu$ preserves the monoidal product,
so by \cref{def lcc}\eqref{lcc: reduced and free}, if $\morred CP \to \morred C$ is a surjection then the same is true of $\mor CP \to \mor C$.
Let $f \colon c_1 \ten \cdots \ten c_n \to e_1 \ten \cdots \ten e_m$ be a reduced morphism of $C$ other than the identity of the tensor unit, and let $G \in \LL_1$ be a graph which represents $\pi(f) \in \mor C(\ast) = \WW_1$.
For specificity, we take $G_{01} = \uk$ and let $\iota \colon G' = \sum_{x\in \uk} G_x \to G$ be the canonical splitting from \cref{def canonical splitting}. 
Define a new morphism of $C$ by 
\[
  f' = (\widehat{\iota^0})^{-1} \circ f \circ  \widehat{\iota^1} \colon c_{\iota^1(1)} \ten \cdots \ten c_{\iota^1(n)} \to e_{\iota^0(1)} \ten \cdots \ten e_{\iota^0(m)} 
\]
where $\iota^1 = d_1(\iota) \colon \un \to \un$ and $\iota^0 = d_0(\iota) \colon \um \to \um$.
For each $x \in \uk$ we also have $\iota^1_x = d_1(\iota_x) \colon \underline{n_x} \to \underline{n_x}$ and $\iota^0_x = d_0(\iota_x) \colon \underline{m_x} \to \underline{m_x}$ where $n_x$ is the number of inputs and $m_x$ is the number of outputs of $G_x$.
By \eqref{lcc: pullback} of \cref{def lcc} the following diagram is cartesian.
\[ \begin{tikzcd}
\prod\limits_{x=1}^k \homred ( \bigotimes\limits_{i=1}^{n_x} c_{\iota_x^1(i)}, \bigotimes\limits_{j=1}^{m_x} e_{\iota^0_x(j)}) \arrow[dr, phantom, "\lrcorner" very near start] \rar["\ten"] \dar & \homred( \bigotimes\limits_{i=1}^n c_{\iota^1(i)}, \bigotimes\limits_{j=1}^m e_{\iota^0(j)}) \dar \\
\prod\limits_{x=1}^k  \homred_{C(\ast)} (\underline{n_x}, \underline{m_x}) \rar["\ten"] & \homred_{C(\ast)} (\un, \um)
\end{tikzcd} \]
Thus there is an element $(f_1', \dots, f_k')$ in the upper left corner with $f'_x$ living over $[G_x]$ and $f_1' \ten \cdots \ten f_k' = f'$. 
The functor $\pi \colon C \to C(\ast)$ is strict monoidal, hence $f'$ lives over $[\sum G_x] = [G']$.
Since each $G_x$ is a corolla, by definition of $P_{G_x}$ we have $f_x' \in P_{G_x} \subseteq \morcon C$.
These assemble to an element $p' \in P_{G'} = P_{\sum G_x}$ so that $\mu_{G'} (p') = f'$.
Then the element $p = (\iota^{-1})^*(p')$ will be sent by $\mu_G$ to $f$, since \cref{lem height one iso mu} gives the last equality in the following 
\[
  (\widehat{\iota^0})^{-1} \circ f \circ  \widehat{\iota^1} = f' = \mu_{G'} (p') = \mu_{G'} \iota^*(p) = (\widehat{\iota^0})^{-1} \circ \mu_G(p) \circ \widehat{\iota^1}.
\]
Thus $\bar \mu$ is surjective.
\end{proof}

\begin{theorem}\label{essentially surjective}
The function $\bar\mu \colon \mor CP \to \mor C$ gives an isomorphism of strict labelled cospan categories $CP \cong C$.
\end{theorem}
\begin{proof}
By \cref{lem d1 mu thing}, $\bar \mu$ preserves composition.
\cref{lem source target} and \cref{lem mu bar bijective} imply that $\hom_{CP}(x,y) \to \hom_C(x,y)$ is a bijection for all $x,y$. 
Preservation of identities follows from surjectivity of these functions and the fact that $\bar\mu$ preserves composition.
Hence $CP \to C$ is a functor.
Since it is bijective on objects and fully-faithful, it is an isomorphism of categories.
By \cref{lem monoidal preservation} this is a map of permutative categories, hence an isomorphism of such.
By \cref{prop image in csp} this is a map over $C(\ast)$, hence an isomorphism in $\slcc$.
\end{proof}

\begin{lemma}\label{lem C full}
The functor $\seg(\LL) \to \slcc$ is full.
\end{lemma}
\begin{proof}
Suppose we are given $X, Y\in \seg(\LL)$ and a map of strict labelled cospan categories $f \colon CX \to CY$.
From the description near the beginning of \S\ref{sec from properad to lcc}, since the permutative functor $f$ lives over $C(\ast)$, we have induced functions $f_{\uk} \colon X_{\uk} \to Y_{\uk}$ and $f_{\mathfrak{c}_{m,n}} \colon X_{\mathfrak{c}_{m,n}} \to Y_{\mathfrak{c}_{m,n}}$; we will first show these give a map $X|_{\LL_\elrm} \to Y|_{\LL_\elrm}$ of $\LL_\elrm$-presheaves.
Three types of maps are indicated in \S\ref{subsec inert} that we must contend with, namely $\ell_i, r_j \colon \underline{1} \to \mathfrak{c}_{m,n}$ and $(\gamma, \chi) \colon \mathfrak{c}_{m,n} \to \mathfrak{c}_{m,n}$.
The diagram for $\ell_i$ is listed just below, where the top composite is $\ell_i^* \colon X_{\mathfrak{c}_{m,n}} \to X_{\underline{1}}$ and similarly for the bottom composite.
The two left squares commute since $f$ preserves domains and the monoidal product.
\[ \begin{tikzcd}
X_{\mathfrak{c}_{m,n}} \rar["d_1"',"{\color{red}\dom}"]  \dar& X_{\um} \rar["\cong", bend right=15, start anchor = {[yshift=-1.5ex]east}, end anchor = {[yshift=-1.5ex]west}]  \dar& X_{\underline{1}}^{\times m} \lar[color=red, dashed,"\ten"', bend right=15, start anchor = {[yshift=1.5ex]west}, end anchor = {[yshift=1.5ex]east}] \rar  \dar & X_{\underline{1}} \dar \\
Y_{\mathfrak{c}_{m,n}} \rar["d_1"',"{\color{red}\dom}"] & Y_{\um} \rar["\cong", bend right=15, start anchor = {[yshift=-1.5ex]east}, end anchor = {[yshift=-1.5ex]west}] & Y_{\underline{1}}^{\times m} \lar[color=red, dashed,"\ten"', bend right=15, start anchor = {[yshift=1.5ex]west}, end anchor = {[yshift=1.5ex]east}] \rar & Y_{\underline{1}}
\end{tikzcd} \]
We similarly have $f_{\underline{1}} r_j^* = r_j^* f_{\mathfrak{c}_{m,n}}$.
For the bipermutations $(\gamma, \chi)$ acting on $\mathfrak{c}_{m,n}$ it is enough to consider the cases where one of $\gamma$ or $\chi$ is an adjacent transposition and the other is an identity.
Let $H$ be the height 2 level graph
\[ \begin{tikzcd}[sep=small]
\underline{m} \ar[dr,"\cong"',"\tau"] & & \underline{m} \ar[dl,"\id"] \ar[dr] & &  \underline{n} \ar[dl] \\
& \underline{m} & & \underline{1}
\end{tikzcd} \]
where
the leftmost bijection $\tau$ interchanges $i$ and $i+1$ and acts as the identity otherwise.
There is an isomorphism $H \cong s_0 \mathfrak{c}_{m,n}$ which is $\tau \colon H_{00} \to (s_0 \mathfrak{c}_{m,n})_{00}$, and the identity elsewhere.
Then the composite
\[ \begin{tikzcd}
\mathfrak{c}_{m,n} \rar{d^1} & H \rar{\cong} & s_0 \mathfrak{c}_{m,n} \rar{s^0} & \mathfrak{c}_{m,n}
\end{tikzcd} \]
is the automorphism of $\mathfrak{c}_{m,n}$ which interchanges the $i$th and $(i+1)$st inputs.
There is a congruence as to the left of the following commutative square (see \cref{sec symmetry}).
\[ \begin{tikzcd}[column sep=large]
d_2H \dar{\equiv} \rar{\cong} & d_2 s_0 \mathfrak{c}_{m,n} \dar{=} \\
s_0(\underline{i-1}) + (\mathfrak{c}_{11} \twplus1 \mathfrak{c}_{11} ) + s_0(\underline{n-i-1}) 
  \rar{\id + \sigma_1 + \id}
&
s_0 \underline{m}
\end{tikzcd} \]
so that $X_{s_0\un} \to X_{d_2H}$ sends the element representing $\id_{x_1 \ten \cdots \ten x_m}$ to the element representing $\id_{x_1\ten \cdots \ten x_{i-1}} \ten \tau_{x_i,x_{i+1}} \ten \id_{x_{i+2}\ten \cdots \ten x_m}$ in $CX$.
Commutativity of the following diagram then gives that a connected morphism $\phi \colon x_1 \ten \cdots \ten x_m \to x_1' \ten \cdots \ten x_n'$ is sent to $\phi \circ (\id_{x_1\ten \cdots \ten x_{i-1}} \ten \tau_{x_i,x_{i+1}} \ten \id_{x_{i+2}\ten \cdots \ten x_m})$.
\[ \begin{tikzcd}
X_{\mathfrak{c}_{m,n}} & 
X_H 
  \lar[swap]{d_1} \dar["\cong"', "d_2 \times d_0"]& 
X_{s_0\mathfrak{c}_{m,n}} 
  \lar["\cong"']  \dar["\cong"',"d_2\times d_0"] & 
X_{\mathfrak{c}_{m,n}} 
  \lar["s_0"'] \ar[dl,"s_0d_1 \times \id", bend left=15] 
\\
& 
X_{d_2 H} \underset{X_{\underline{m}}}\times X_{d_0H} 
  \ar[ul, bend left=15, dashed]
& 
X_{s_0 \underline{m}} \underset{X_{\underline{m}}}\times
X_{\mathfrak{c}_{m,n}}
  \lar  
\end{tikzcd} \]
All of this structure is compatible with $f\colon CX \to CY$, so we conclude that 
\[ \begin{tikzcd}
X_{\mathfrak{c}_{m,n}} \dar  & X_{\mathfrak{c}_{m,n}} \lar \dar \\
Y_{\mathfrak{c}_{m,n}} & Y_{\mathfrak{c}_{m,n}} \lar 
\end{tikzcd} \]
commutes.
The case when an adjacent transposition acts on an output is similar.
It follows that $f$ defines a map $X|_{\LL_\elrm} \to Y|_{\LL_\elrm}$.

Since $Y$ is Segal, $Y|_{\LL_\intrm}$ is the right Kan extension of $Y|_{\LL_\elrm}$, hence we obtain a unique extension $X|_{\LL_\intrm} \to Y|_{\LL_\intrm}$.
In particular, we have functions $f_G \colon X_G \to Y_G$ for every level graph $G$.
It remains only to check that these maps are compatible with degeneracies and inner faces.
By \eqref{diagram face} above \cref{def inner face}, we only need to check compatibility with the inner faces $d_1 \colon X_G \to X_{d_1G}$ where $G \in \LL_2$, and by \eqref{diagram degeneracy} below \cref{def degeneracies} we only need to check compatibility with the degeneracies $s_0 \colon X_{\un} \to X_{s_0(\un)}$.
Again by Segality we may further reduce to the cases where $G$ is connected (for $d_1$) and $n=1$ (for $s_0$).
For $d_1$, all squares of the following cube are known to commute with the exception of the back, which then must commute since $Y_{d_1G} \to SY_1$ is injective (using that $d_1G$ is a corolla).
\[ \begin{tikzcd}[sep=small]
X_{d_1G} \ar[dr,hook]  \ar[dd]& & X_G \ar[dr] \ar[ll,"d_1"']  \ar[dd]\\
& SX_1   & & SX_2 \ar[ll,"d_1"' near end, crossing over] \ar[dd]
\\
Y_{d_1G} \ar[dr,hook] & & Y_G \ar[dr] \ar[ll,"d_1"' near start] \\
& SY_1 \ar[from=uu, crossing over] & & SY_2 \ar[ll,"d_1"']
\end{tikzcd} \]
The map $s_0 \colon X_{\mathfrak{e}} \to X_{\mathfrak{c_{1,1}}} \subset \morcon(CX)$ takes a connected object $x$ to $\id_x$, so $f(s_0 x) = f(\id_x) = \id_{f x} = s_0 (fx)$, as desired.
We conclude that $X|_{\LL_\intrm} \to Y|_{\LL_\intrm}$ is, in fact, a map $X\to Y$.
By construction, this map is sent to $f$ by $C$.
\end{proof}

\begin{theorem}\label{equivalence of presheaves and slcc}
The functor $\seg(\LL) \to \slcc$ is an equivalence of categories.
\end{theorem}
\begin{proof}
This functor is fully faithful by \cref{prop C is faithful} and \cref{lem C full}, and is essentially surjective by \cref{essentially surjective}.
\end{proof}

Combining the previous theorem with \cref{prop properads as presheaves}, we have:

\begin{corollary}\label{cor main theorem}
The category of properads is equivalent to the category of strict labelled cospan categories. \qed
\end{corollary}

\section{2-Categorical structures}\label{section 2-cat structures}
In this section we turn our attention to the biequivalence between properads and labelled cospan categories, \cref{main theorem as corollary}, which is the main result of this paper.
The functor from strict labelled cospan categories into labelled cospan categories is not an equivalence of 1-categories, as it is only surjective up to equivalence, not up to isomorphism.
Thus to connect properads to labelled cospan categories, we will need to view $\properad$ as a 2-category, rather than as a 1-category.
The following definition is introduced in \cite{Hackney:Properad2Cat}, where it is related to the Boardman--Vogt-style tensor product of properads.

\begin{definition}[Natural transformation of properad maps]\label{properad natural trans}
If $F, G \colon P \to Q$ are two maps of properads with the same source and target, then a \emph{natural transformation} from $F$ to $G$ is nothing but a polynatural transformation between the underlying map of polycategories/dioperads, in the sense of \cite[Definition 2.5.16]{JohnsonYau:2DC}.
This means that a natural transformation is given by a collection of unary morphisms $\gamma_c \colon F(c) \to G(c)$, indexed by the colors of $P$, so that, for each $p \in P(c_1, \dots, c_n; d_1, \dots, d_m)$, we have
\begin{equation}\label{eq properad nat trans}
  (Gp)(\gamma_{c_1}, \dots, \gamma_{c_n}) = (\gamma_{d_1}, \dots, \gamma_{d_m})(Fp)
\end{equation}
in $Q(Fc_1, \dots, Fc_n; Gd_1, \dots, Gd_m)$.
\end{definition}
As in \cite{JohnsonYau:2DC}, the notation in \eqref{eq properad nat trans} is meant to represent iterated dioperadic composition, see \cref{fig: nat trans rel}.

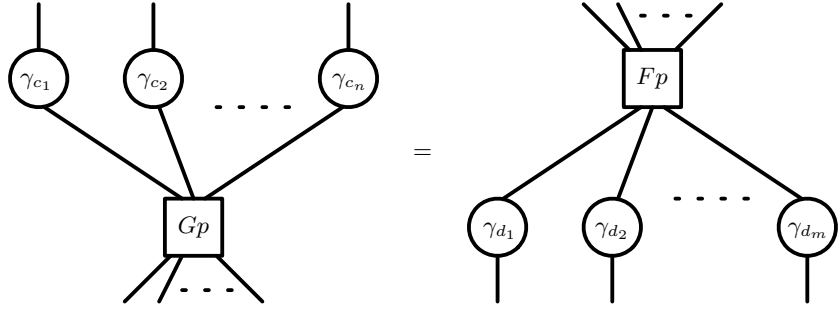
\begin{figure}
\labellist
\small\hair 2pt
 \pinlabel {$\gamma_{c_1}$} [B] at 33 184
 \pinlabel {$\gamma_{c_2}$} [B] at 125 184
 \pinlabel {$\gamma_{c_n}$} [B] at 279 184
 \pinlabel {$Fp$} [B] at 516 184
 \pinlabel {$=$} [B] at 336 126
 \pinlabel {$Gp$} [B] at 156 67
 \pinlabel {$\gamma_{d_1}$} [B] at 397 67
 \pinlabel {$\gamma_{d_2}$} [B] at 485 67
 \pinlabel {$\gamma_{d_m}$} [B] at 639 67
\endlabellist
\centering
\includegraphics[width=0.9\textwidth]{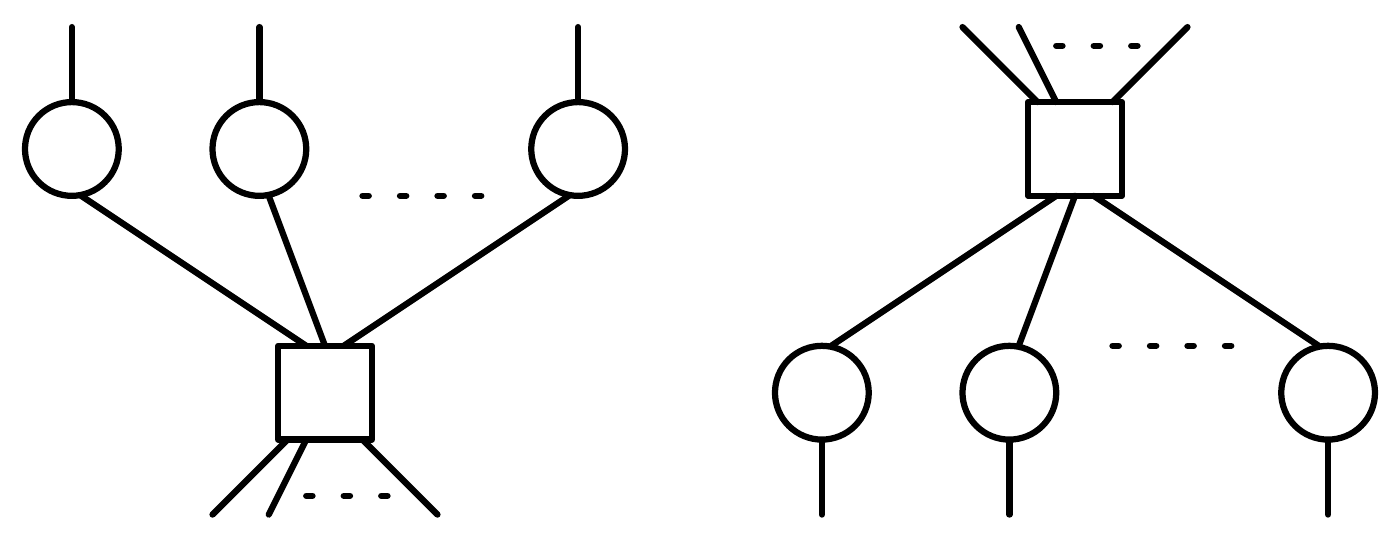}
\caption{Relation satisfied by a properadic natural transformation}
\label{fig: nat trans rel}
\end{figure}

\begin{remark}
In \cite[Section 4.4]{HRYbook}, a closed monoidal structure is given for those properads whose operations all have at least one input and at least one output, so in particular this subcategory becomes a 2-category. 
But this structure is not extended there to the category of all properads. 
\cref{properad natural trans} is the extension of this 2-category to $\properad$.
(See \cite[Definition 6.3]{Duncan:TQC} for a different notion.)
\end{remark}

\subsection{Equivalences of labelled cospan categories}
Let $\lcc$ denote the 2-category of labelled cospan categories, equipped with the 2-cells of \cref{rmk: 2-cat struct}.

\begin{proposition}\label{prop characterization of equivs}
Suppose $(f,\alpha) \colon C \to D$ is a 1-morphism of labelled cospan categories, as depicted below.
\[ \begin{tikzcd}[row sep=small, column sep=tiny]
C \ar[rr,"f"] \ar[ddr,"\pi"'] & & D \ar[ddl,"\mu"] 
\\
\ar[rr,"\text{\scriptsize $\alpha \cong$}", phantom] & & {} 
\\
& \csp 
\end{tikzcd} \]
The map $(f,\alpha)$ is an equivalence in the 2-category of labelled cospan categories if and only if $f\colon C \to D$ is an equivalence of categories.
\end{proposition}
\begin{proof}
Suppose $f\colon C \to D$ is an equivalence of categories.
Then $f$ is an equivalence in the 2-category of symmetric monoidal categories, symmetric monoidal functors, and monoidal natural transformations.
Choose $g\colon D \to C$ along with monoidal natural isomorphisms $\eta \colon \id_C \cong gf$ and $\varepsilon \colon fg \cong \id_D$.

Define a (monoidal) natural transformation $\beta \colon \pi g \Rightarrow \mu$ as the following composite.
\[ \begin{tikzcd}
D \rar{g} \ar[rr, bend right=50, "\id_D"'] \ar[rr, bend right=25, phantom, "\text{\scriptsize $\varepsilon \Downarrow$}"] 
& \ar[rr, bend left=50, "\pi"] \ar[rr, bend left=25, phantom, "\text{\scriptsize $\alpha^{-1} \Downarrow$}"] C \rar{f} & D \rar{\mu} & \csp
\end{tikzcd} \]
Then $(g,\beta) \colon D \to C$ is a 1-morphism of labelled cospan categories.
Since the following pasting composites are isomorphisms,
\[ \begin{tikzcd}
C \ar[rr,"\id_C"', bend right=20] \ar[rr,"gf", bend left=20] \ar[rr,"\text{\scriptsize $\eta^{-1} \Downarrow$}",phantom] \ar[ddr] & & C \ar[ddl]
\\
\ar[rr,"\text{\scriptsize $\id_{\pi} \cong$}", phantom, bend right=10] & & {} 
\\
& \csp & 
\end{tikzcd} 
\qquad
\begin{tikzcd}
C' \ar[rr,"\id_{C'}"', bend right=20] \ar[rr,"fg", bend left=20] \ar[rr,"\text{\scriptsize $\varepsilon \Downarrow$}",phantom] \ar[ddr] & & C' \ar[ddl]
\\
\ar[rr,"\text{\scriptsize $\id_{\mu} \cong$}", phantom, bend right=10] & & {} 
\\
& \csp & 
\end{tikzcd}
\]
by \cref{rmk: 2-cat struct} we have $(g,\beta) \circ (f,\alpha) \cong \id_{\pi}$ and $(f,\alpha) \circ (g,\beta) \cong \id_{\mu}$ in $\lcc$.
\end{proof}

Note that if $f$ is an \emph{isomorphism} of categories, then $(f,\alpha)$ is an isomorphism of labelled cospan categories.

\subsection{Strict labelled cospan categories form a 2-category}
We now endow $\slcc$ with the structure of a 2-category, which will serve as an intermediary between $\properad$ (considered as a 2-category) and $\lcc$.

\begin{definition}\label{def slcc 2cat}
Given two maps  
\[ \begin{tikzcd}[column sep=tiny]
C \ar[rr,"f"] \ar[dr,"\pi"'] & & C' \ar[dl,"\pi'"] \\
& C(\ast)
\end{tikzcd} \qquad \qquad 
\begin{tikzcd}[column sep=tiny]
C \ar[rr,"g"] \ar[dr,"\pi"'] & & C' \ar[dl,"\pi'"] \\
& C(\ast)
\end{tikzcd}
\]
of strict labelled cospan categories (that is, commutative triangles with $f$ and $g$ strict maps of permutative categories), we define a 2-morphism from one to the other to be a monoidal natural transformation $\gamma \colon f \Rightarrow g$ so that the whiskering 
\begin{equation}\label{diag whisker}
\begin{tikzcd}
C \ar[rr,"g"', bend right=20] \ar[rr,"f", bend left=20] \ar[rr,"\text{\scriptsize $\gamma \Downarrow$}",phantom]  & & C' \ar[ddl,"\pi'"]
\\
& &
\\
& C(\ast) & 
\end{tikzcd}
\end{equation}
is the identity natural transformation on $\pi \colon C \to C(\ast)$.
\end{definition}

If the whiskering \eqref{diag whisker} is an isomorphism, then it is automatically an identity. 
This follows from \cref{rmk unique ni} and the fact that the canonical inclusion $C(\ast) \to \csp$ is injective-on-objects and faithful.

There is a functor from the 1-category of strict labelled cospan categories to the 2-category $\lcc$ which is given by composition with $C(\ast) \to \csp$.

\begin{lemma}\label{lem slcc lcc 2-functor}
The usual inclusion $\slcc \to \lcc$ extends to a 2-functor.
\end{lemma}
\begin{proof}
Suppose we are given a 2-morphism $\gamma$ between 1-morphisms $f$ and $g$ in $\slcc$ as in \cref{def slcc 2cat}.
We then obtain the middle equality in the following chain of equalities of symmetric monoidal natural transformations between monoidal functors.
\[ 
\begin{tikzcd}[row sep=small, column sep=tiny]
C \ar[rr, bend right=20] \ar[rr, bend left=20] \ar[rr,"\text{\scriptsize $\Downarrow$}",phantom]  \ar[dddr, bend right] & & C' \ar[dddl, bend left]
\\
\ar[rr,"=",phantom, bend right] & & {}
\\
& \phantom{C(\ast)} & \\
& \csp
\end{tikzcd}
= 
\begin{tikzcd}[row sep=small, column sep=tiny]
C \ar[rr, bend right=20] \ar[rr, bend left=20] \ar[rr,"\text{\scriptsize $\Downarrow$}",phantom] \ar[ddr]  & & C' \ar[ddl] 
\\
\ar[rr,"=",phantom, bend right=10] & & {}
\\
& C(\ast) \dar & \\
& \csp
\end{tikzcd}
=
\begin{tikzcd}[row sep=small, column sep=tiny]
C \ar[rr, bend left=20]  \ar[ddr]  & & C' \ar[ddl] 
\\
\ar[rr,"=",phantom, bend left=10] & & {}
\\
& C(\ast) \dar & \\
& \csp
\end{tikzcd}
=
\begin{tikzcd}[row sep=small, column sep=tiny]
C  \ar[rr, bend left=20]  \ar[dddr, bend right] & & C'  \ar[dddl, bend left]
\\
\ar[rr,"=",phantom, bend right] & & {}
\\
& \phantom{C(\ast)} & \\
& \csp
\end{tikzcd}
\]
Thus $\gamma$ is a 2-morphism between the 1-morphisms of labelled cospan categories $(f,\id)$ and $(g,\id)$ as in \cref{rmk: 2-cat struct}.
\end{proof}

We will return to the relationship between $\slcc$ and $\lcc$ in \cref{sec slcc lcc 2-cat}, but first we will give more detail to facilitate the comparison with natural transformations of properads.
Given a 2-morphism $\gamma$ of $\slcc$ as in \cref{def slcc 2cat}, the underlying data consists of \emph{connected} morphisms
\[
  \gamma_x \colon f(x) \to g(x)
\]
for each connected object $x$ of $C$.
We know these maps must all be connected, as $\pi'(\gamma_x) = (\id_{\pi})_x = \id_{\pi(x)} = \id_{\underline{1}}$ is a connected map in $C(\ast)$.
These maps are also the full extent of the data: if $c$ is an arbitrary object of $C$, it can be uniquely written as $c = x_1 \otimes \dots \otimes x_n$ where the $x_i$ are connected objects.
Since $\gamma$ is a monoidal natural transformation and $f$ and $g$ are strict monoidal functors, the diagram
\[ \begin{tikzcd}[column sep=huge]
f(x_1) \otimes \dots \otimes f(x_n) \dar["="] \rar["\gamma_{x_1} \otimes \dots \otimes \gamma_{x_n}"] & g(x_1) \otimes \dots \otimes g(x_n) \dar["="] \\
f(x_1 \otimes \dots \otimes x_n) \rar["\gamma_{c}"] & g(x_1 \otimes \dots \otimes x_n)
\end{tikzcd} \]
commutes, that is, $\gamma_c = \gamma_{x_1} \otimes \dots \otimes \gamma_{x_n}$.
We now observe that one only needs to check compatibility of $\gamma$ with \emph{connected morphisms}, rather than arbitrary morphisms.

\begin{proposition}\label{proposition reduction}
Suppose $f, g \colon C \to C'$ are two 1-morphisms of strict labelled cospan categories, and for each connected object $x$ of $C$ we have a connected morphism $\gamma_x \colon f(x) \to g(x)$.
If the diagram
\[ \begin{tikzcd}
f(x_1) \otimes \dots \otimes f(x_n) = f(x_1 \otimes \dots \otimes x_n) \rar{f(p)} \dar[shift right=1.5cm, "\gamma_{x_1} \otimes \dots \otimes \gamma_{x_n}"] & f(y_1 \otimes \dots \otimes y_m) \dar[ "\gamma_{y_1} \otimes \dots \otimes \gamma_{y_m}"] \\
g(x_1) \otimes \dots \otimes g(x_n) = g(x_1 \otimes \dots \otimes x_n) \rar{g(p)} & g(y_1 \otimes \dots \otimes y_m) 
\end{tikzcd} \]
commutes for every connected morphism $p\colon x_1 \otimes \dots \otimes x_n \to y_1 \otimes \dots \otimes y_m$ of $C$, then the collection $\{ \gamma_{x_1 \otimes \dots \otimes x_n} \coloneqq \gamma_{x_1} \otimes \dots \otimes \gamma_{x_n} \}$ constitutes a 2-morphism $f \Rightarrow g$.
\end{proposition}
\begin{proof}
Any endmorphism $p$ of the tensor unit $\mathbf{1} \in C$ can be written as a composition of connected morphisms by \cref{def lcc}\eqref{lcc: free gen}.
Hence our assumption implies that $g(p) \gamma_{\mathbf{1}} = \gamma_{\mathbf{1}} f(p)$ for all such endomorphisms; since $\gamma_{\mathbf{1}}$ is the identity on the tensor unit of $C'$, this implies that $f(p) = g(p)$.
Now by \cref{def lcc}\eqref{lcc: reduced and free}, any morphism $p\colon c\to d$ can be written uniquely as $p' \otimes p'' \colon c \otimes \mathbf{1} \to d \otimes \mathbf{1}$ where $p'$ is a reduced morphism and $p''$ is an endomorphism of $\mathbf{1}$, so it remains to show that $\gamma_d f(p') = g(p') \gamma_c$ for reduced morphisms $p'$.

Let $\sigma$ be a bijection on $\un = \{ 1, \dots, n \}$ and $x_1, \dots, x_n$ be connected objects in $C$. 
Naturality of symmetry implies that the diagram 
\[ \begin{tikzcd}[column sep=3cm]
f(x_{\sigma(1)}) \otimes \dots \otimes f(x_{\sigma(n)}) 
\rar[ "\gamma_{x_{\sigma(1)}} \otimes \dots \otimes \gamma_{x_{\sigma(n)}}"]  \dar{\hat{\sigma}}
& g(x_{\sigma(1)}) \otimes \dots \otimes g(x_{\sigma(n)}) \dar{\hat{\sigma}} \\
f(x_1) \otimes \dots \otimes f(x_n) \rar[ "\gamma_{x_1} \otimes \dots \otimes \gamma_{x_n}"] & g(x_1) \otimes \dots \otimes g(x_n)
\end{tikzcd} \]
commutes in the permutative category $C'$.
Since $f$ and $g$ are maps of permutative categories, the preceding diagram is equal to the following,
\[ \begin{tikzcd}[column sep=3cm]
f(x_{\sigma(1)}  \otimes \dots \otimes x_{\sigma(n)}) 
\rar[ "\gamma_{x_{\sigma(1)} \otimes \dots \otimes x_{\sigma(n)}}"]  \dar{f(\hat{\sigma})}
& g(x_{\sigma(1)} \otimes \dots \otimes x_{\sigma(n)}) \dar{g(\hat{\sigma})} \\
f(x_1 \otimes \dots \otimes x_n) \rar[ "\gamma_{x_1 \otimes \dots \otimes x_n}"] & g(x_1 \otimes \dots \otimes x_n)
\end{tikzcd} \]
where $\hat{\sigma} \colon x_{\sigma(1)}  \otimes \dots \otimes x_{\sigma(n)} \to x_1 \otimes \dots \otimes x_n$ is the permuation map in $C$.
Thus the proposed natural transformation is compatible with permutation maps.

Suppose $p \colon x_1 \otimes \dots \otimes x_n \to y_1 \otimes \dots \otimes y_m$ is an arbitrary reduced morphism of $C$ lying over $\un \to \uk \leftarrow \um$ in $C(\ast)$.
By choosing appropriate bijections $\sigma$ of $\un$ and $\rho$ of $\um$, we can arrange things so that $\hat \rho^{-1} \circ p \circ \hat \sigma$ is sent by $\pi$ to a $k$-fold tensor product of connected morphisms in $C(\ast)$ (see \cref{injectivity after quotient,lem mu bar bijective}).
Once we've done this, by \cref{def lcc}\eqref{lcc: pullback} we can find unique connected morphisms $p_1, \dots, p_k$ of $C$ so that the diagram
\[ \begin{tikzcd}[column sep=large]
x_{\sigma(1)} \otimes \dots \otimes x_{\sigma(n)} \rar{p_1 \otimes \dots \otimes p_k} \dar{\hat \sigma} & y_{\rho(1)} \otimes \dots \otimes y_{\rho(m)} \dar{\hat \rho} \\
x_{1} \otimes \dots \otimes x_{n} \rar{p} & y_{1} \otimes \dots \otimes y_{m}
\end{tikzcd} \]
commutes.
By assumption each of the connected morphisms $p_i$ is compatible with $\gamma$, hence so is $p_1 \otimes \dots \otimes p_k$.
Combining this with the previous paragraph, we see that \[p = \hat \rho \circ (p_1 \otimes \dots \otimes p_k) \circ \hat \sigma^{-1} = \hat \rho \circ (p_1 \otimes \dots \otimes p_k) \circ \widehat{\sigma^{-1}}\]  is also compatible with $\gamma$.
\end{proof}

We have a composite of equivalences of 1-categories
\[
  \properad \xrightarrow{\simeq} \seg(\LL) \xrightarrow{\simeq} \slcc
\]
by \cref{prop properads as presheaves} and \cref{equivalence of presheaves and slcc}.
This extends to a strict 2-equivalence between 2-categories.

\begin{theorem}\label{thm properad vs slcc 2-cat}
The 1-functor $\properad \to \slcc$ extends to a 2-functor.
This 2-functor induces isomorphisms on hom-categories, hence is a strict 2-equivalence.
\end{theorem}
The last conclusion uses the previously established fact that the functor is surjective on objects up to isomorphism.
\begin{proof}
Suppose $F,G \colon P \to P'$ are two maps of properads, and $f, g\colon C \to C'$ are the associated maps of strict labelled cospan categories.
The colors of $P$ are precisely the connected objects of $C$, and the operations of $P$ are precisely the connected morphisms of $C$.
Under this correspondence, the diagram from \cref{proposition reduction} is exactly the condition for a family of unary operations in $P'$ to constitute a natural transformation $F\Rightarrow G$ of properad maps (\cref{properad natural trans}).
It follows that 
\[
  \properad(P,P')(F,G) \to \slcc(C,C')(f,g)
\]
is a bijection, that is, $\properad \to \slcc$ is locally fully faithful.
But we already know that $\properad(P,P') \to \slcc(C,C')$ is bijective on objects, hence is an isomorphism of categories.
\end{proof}

\subsection{Comparison with labelled cospan categories}\label{sec slcc lcc 2-cat}

In this section we show that the inclusion of strict labelled cospan categories into all labelled cospan categories is a biequivalence of 2-categories.

\begin{notation}
If $\pi \colon C\to C(\ast)$ is a strict labelled cospan category, we write \[ \tilde \pi \colon C \to C(\ast) \to \csp\] for the composition of $\pi$ with the canonical inclusion $C(\ast) \to \csp$ from \cref{example C star Csp,example C star monoidal}. 
Given a 1-morphism of strict labelled cospan categories, that is a commutative triangle below left
\[
\begin{tikzcd}[row sep=small, column sep=tiny]
C \ar[rr,"f"] \ar[ddr,"\pi"'] & & C' \ar[ddl,"\pi'"] 
  &[+0.5cm]  C \ar[rr,"f"] \ar[ddr,"\tilde \pi"'] & & C' \ar[ddl,"\tilde \pi'"] \\
& &  & 
  \ar[rr,"\text{\scriptsize $\id_{\tilde \pi}$}", phantom] & & {} \\
& C(\ast) & &  & \csp & 
\end{tikzcd} 
\]
where $f$ is a strict permutative functor, we have the 1-morphism $\tilde f \coloneqq (f,\id_{\tilde \pi})$ of labelled cospan categories above right. 
Finally, if $\gamma \colon f \Rightarrow g$ is a 2-morphism as in \cref{def slcc 2cat}, we write $\tilde \gamma \colon (f,\id_{\tilde \pi}) \Rightarrow (g,\id_{\tilde \pi})$ for the same natural transformation, but now thought of as a 2-morphism in $\lcc$.
\end{notation}

We already observed in \cref{lem slcc lcc 2-functor} that $\slcc \to \lcc$ is a 2-functor under these assignments.

\begin{lemma}
The 2-functor $\slcc \to \lcc$ is locally fully faithful.
\end{lemma}
\begin{proof}
Suppose $f,g \colon C \to C'$ are two 1-morphisms of labelled cospan categories.
We wish to show that
\[
  \slcc(C,C')(f,g) \to \lcc(C,C')(\tilde f, \tilde g)
\]
is a bijection.
It is automatically injective since the elements on both sides are just certain natural transformations $f\Rightarrow g$.
Suppose we have a 2-morphism $\tilde f \Rightarrow \tilde g$ as in \cref{rmk: 2-cat struct}, that is a monoidal natural transformation $\gamma \colon f \Rightarrow g$ so that the composite natural transformation
\[ \begin{tikzcd}
C \ar[rr,"g"', bend right=20] \ar[rr,"f", bend left=20] \ar[rr,"\text{\scriptsize $\gamma \Downarrow$}",phantom] \ar[ddr,"\tilde \pi"'] & & C' \ar[ddl,"\tilde \pi'"]
\\
\ar[rr,"\text{\scriptsize $\id_{\tilde \pi} \cong$}", phantom, bend right=10] & & {} 
\\
& \csp & 
\end{tikzcd} \]
is $\id_{\tilde \pi}$.
This composite is just the whiskering of $\tilde \pi'$ with $\gamma$.
To show that $\gamma$ is a 2-morphism of $\slcc$, we need to show that its whiskering with $\pi'$ is the identity on $\pi$
\[ \begin{tikzcd}
C \ar[r,"g"', bend right] \ar[r,"f", bend left] \ar[r,"\text{\scriptsize $\gamma \Downarrow$}",phantom]  & C' \rar["\pi'"] & C(\ast) \rar & \csp
\end{tikzcd} \]
which follows since $C(\ast) \to \csp$ is faithful.
\end{proof}

\begin{lemma}\label{lem slcc lcc loc equiv cat}
The 2-functor $\slcc \to \lcc$ is locally an equivalence of categories.
\end{lemma}
\begin{proof}
Let $\pi\colon C \to C(\ast)$ and $\pi' \colon C' \to C(\ast)$ be strict labelled cospan categories.
From the previous lemma we know that $\slcc(C,C') \to \lcc(C,C')$ is fully faithful, so it remains to prove that it is essentially surjective.
Suppose $(f,\alpha)$ as depicted below left is 1-morphism in $\lcc$.
\[ \begin{tikzcd}[row sep=small, column sep=tiny]
C \ar[rr,"f"] \ar[ddr,"\tilde \pi"'] & & C' \ar[ddl,"\tilde \pi'"] \\
\ar[rr,"\text{\scriptsize $\alpha\cong$}", phantom] & & {} \\
& \csp & 
\end{tikzcd} \qquad 
\begin{tikzcd}[row sep=small, column sep=tiny]
C \ar[rr,"f"] \ar[ddr,"\pi"'] & & C' \ar[ddl,"\pi'"] \\
\ar[rr,"\text{\scriptsize $\beta\cong$}", phantom] & & {} \\
& C(\ast) & 
\end{tikzcd}
\]
Here $f$ is a symmetric monoidal functor (which does not need to be a strict map of permutative categories) and $\alpha$ is a monoidal natural isomorphism.
Since $C(\ast) \to \csp$ is fully faithful, there is a (unique) monoidal natural isomorphism $\beta$ as displayed above right, whose whiskering with $C(\ast) \to \csp$ is $\alpha$. 
Generically write $\Phi \colon f(c_1) \otimes \dots \otimes f(c_n) \to f(c_1 \otimes \dots \otimes c_n)$ for the structural isomorphisms of $f$ (including the case $n=0$ for the tensor unit).

We define a new functor of permutative categories $g \colon C \to C'$.
On a general object $x_1 \otimes \dots \otimes x_n$ of $C$ (with each $x_i$ connected), $g$ is given by
\[
  g(x_1 \otimes \dots \otimes x_n) \coloneqq f(x_1) \otimes \dots \otimes f(x_n).
\]
If $p\colon x_1 \otimes \dots \otimes x_n \to y_1 \otimes \dots \otimes y_m$ is any morphism of $C$, define $g(p)$ to sit in the following commutative square.
\[ \begin{tikzcd}
f(x_1)\otimes \dots \otimes f(x_n) \rar{g(p)} \dar["\Phi"', "\cong"] & f(y_1) \otimes \dots \otimes f(y_m) \dar["\Phi", "\cong"'] \\
f(x_1 \otimes \dots \otimes x_n) \rar{f(p)} & f(y_1 \otimes \dots \otimes y_m)
\end{tikzcd} \]
This $g$ is automatically a functor, and the commutativity of the diagram 
\[ \begin{tikzcd}
\left(\bigotimes\limits_{i=1}^k f(x_i) \right) \otimes \left(\bigotimes\limits_{i=k+1}^n f(x_i) \right) \rar[dashed] \dar["\Phi \otimes \Phi"'] \ar[dd, "\Phi "', bend right=50, start anchor=west, end anchor = west] & \left(\bigotimes\limits_{i=1}^j f(y_i) \right) \otimes \left(\bigotimes\limits_{i=j+1}^m f(y_i) \right) \dar["\Phi \otimes \Phi"] \ar[dd, "\Phi", bend left=50, start anchor=east, end anchor = east] \\
f\left(\bigotimes\limits_{i=1}^k x_i \right) \otimes f\left(\bigotimes\limits_{i=k+1}^n x_i \right) \rar{f(p) \otimes f(p')} \dar["\Phi"'] & f\left(\bigotimes\limits_{i=1}^j y_i \right) \otimes f\left(\bigotimes\limits_{i=j+1}^m y_i \right) \dar["\Phi"] \\
f(x_1 \otimes \dots \otimes x_n) \rar{f(p \otimes p')} & f(y_1 \otimes \dots \otimes y_m)
\end{tikzcd} \]
implies that the unique dashed map must be both $g(p\otimes p')$ and $g(p) \otimes g(p')$, hence $g$ is a strict monoidal functor.
A similar diagram shows how to infer from $f$ being symmetric monoidal functor that the same is true for $g$.

We will return to checking that $g$ is a 1-morphism in $\slcc$ in a moment.
First, we define $\gamma \colon g \Rightarrow f$ to be the monoidal natural isomorphism given by $\Phi$.
That is, if $c = x_1 \otimes \dots \otimes x_n$, then $\gamma_c$ is the map
\[
  g(c) = f(x_1) \otimes \dots \otimes f(x_n) \to f(x_1 \otimes \dots \otimes x_n) = f(c).
\]
We now calculate the following composite.
\begin{equation}\label{2-cell to compute} \begin{tikzcd}
C \ar[r,"f"', bend right] \ar[r,"g", bend left] \ar[r,"\text{\scriptsize $\gamma \Downarrow$}",phantom] \ar[rr,bend right=55,"\pi"'] \ar[rr,bend right, phantom, "\text{\scriptsize $\beta \Downarrow$}" near end] & C' \rar["\pi'"] & C(\ast)
\end{tikzcd} \end{equation}
Since $\pi$ and $\pi'$ are strict monoidal functors and $\beta_{x_i}$ is the identity on $\underline{1}$ for a connected object $x_i$, commutativity of the following square
\[ \begin{tikzcd}[column sep=huge]
\pi'f(x_1) \otimes \dots \otimes \pi'f(x_n) \dar{=} \rar{\beta_{x_1} \otimes \dots \otimes \beta_{x_n}} &
  \pi(x_1) \otimes \dots \otimes \pi(x_n) \ar[dd,"="] \\
\pi'(f(x_1) \otimes \dots \otimes f(x_n)) \dar{\pi'(\Phi)}  \\
\pi'f(x_1 \otimes \dots \otimes x_n) \rar{\beta_{x_1 \otimes \dots \otimes x_n}} &
  \pi(x_1 \otimes \dots \otimes x_n)
\end{tikzcd} \]
tells us that $\beta_{x_1 \otimes \dots \otimes x_n}$ is the inverse of $\pi'(\Phi_{x_1, \dots, x_n})$.
Thus $\beta_c \circ \pi'(\gamma_c) = \id_{\pi c}$ for all objects $c\in C$, so it follows that the composite of \eqref{2-cell to compute} is the identity 2-morphism on the 1-morphism $\pi$.
We then have $\pi' g = \pi$ and so $g$ is a 1-morphism in $\slcc$, and $\gamma$ is an isomorphism in $\lcc(C,C')$ between $(g,\id_{\tilde \pi})$ and $(f,\alpha)$.
\end{proof}

It remains to prove that $\slcc \to \lcc$ is surjective up to equivalence.
We begin with a special case.

\begin{lemma}\label{lem strictify lcc}
Suppose that $C$ is a permutative category whose set of objects is a free monoid on the set $S$ under the monoidal product, and that $\pi \colon C \to \csp$ is a symmetric monoidal functor.
If $\pi$ is a labelled cospan category such that the set of connected objects is precisely $S\subset \ob C$, then there is a strict labelled cospan category $\pi' \colon C \to C(\ast)$ such that the labelled cospan category $\tilde \pi'$ is isomorphic to $\pi$.
\end{lemma}
\begin{proof}
In this proof we write $i\colon C(\ast) \to \csp$ for the canonical inclusion, and, whenever $k + j = n$, write 
\[ \Psi_{k,j} \colon i\uk \otimes i\uj \to i(\uk + \uj) = i \un \]
for the natural monoidal structure isomorphism of $i$.
We similarly write $\Phi_{c,d} \colon \pi c \otimes \pi d \to \pi(c\otimes d)$ for the structure isomorphism of the symmetric monoidal functor $\pi$.
We now aim to simultaneously define a more strict version of $\pi$ called $\kappa \colon C \to \csp$, along with a natural isomorphism $\beta \colon \pi \cong \kappa$.
On objects, $\kappa$ takes a word $c = [x_1 x_2 \dots x_n]$ of length $n$ (where each $x_i \in S$) to $i\un$.
We now inductively, based on word length, define isomorphisms $\beta_c \colon \pi c \to \kappa c$. 
For length zero and length one words, we declare that $\beta_{[\,]} \colon \pi[\,] \to \kappa [\, ] = \varnothing$ is the identity, and $\beta_{[x]} \colon \pi[x] \to \kappa[x] = i\underline{1}$ is the unique isomorphism.
Suppose $c$ has length $k > 0$ and $d$ has length $j>0$, and that $\beta_c$ and $\beta_d$ have been defined.
We then define $\beta_{c\otimes d}$ as the unique map fitting into the following square
\[ \begin{tikzcd}
\pi c \otimes \pi d \rar{\beta_c \otimes \beta_d}  \dar["\Phi_{c,d}","\cong"']& i\uk \otimes i\uj \dar["\Psi_{\uk,\uj}","\cong"'] \\
\pi(c \otimes d) \rar{\beta_{c\otimes d}} & i(\uk + \uj).
\end{tikzcd} \]
One must check that this is well-defined, that is, for any positive-length words $c,d,e \in \ob C$ that $\beta_{(c\otimes d) \otimes e} = \beta_{c\otimes (d\otimes e)}$.
This follows by using the hexagon constraints for $\Phi$ and $\Psi$ and the fact that $C$ is strict monoidal. 

We define $\kappa$ on morphisms $p\colon c\to d$ by declaring
\[
  \kappa(p) \coloneqq \beta_d \circ \pi(p) \circ \beta_c^{-1}.
\]
So defined, $\kappa$ is automatically a functor and $\beta$ is a natural transformation.
Since $i$ is fully faithful, there is a unique functor $\pi' \colon C \to C(\ast)$ so that $i\pi' = \kappa$.
Using the defining equations, one can show 
that $\pi'$ is a strict monoidal functor, which is also symmetric monoidal.
From the diagram above one concludes that $\beta$ is a monoidal natural transformation $\pi \Rightarrow i\pi'$.

The existence of the monoidal natural isomorphism $\beta$ allows one to check that $\kappa$ is a labelled cospan category.
Further, $\pi'$ is a strict labelled cospan category and $(\id_C, \beta)$ is an isomorphism $\tilde \pi' = \kappa \cong \pi$ of labelled cospan categories.
\end{proof}

\begin{lemma}\label{lem slcc lcc surj equiv}
The 2-functor $\slcc \to \lcc$ is surjective up to equivalence.
\end{lemma}
\begin{proof}
Suppose $\pi \colon C \to \csp$ is a labelled cospan category. 
A variation on the proof of \cite[Proposition 4.2]{May74} lets us define a permutative category $D$ and a symmetric monoidal equivalence $f\colon D \to C$ so that $\ob D$ is the free monoid on the set $\ob^{\mathrm{c}} C$ of connected objects of the labelled cospan category $C$.
The functor $f$ takes an object $[x_1 x_2 \cdots x_{n-1} x_n]$ (where each $x_i \in \ob^{\mathrm{c}} C$) to $x_1 \otimes (x_2 \otimes ( \cdots \otimes (x_{n-1} \otimes x_n)\cdots ))$.
One can check that $\pi f \colon D \to \csp$ is a labelled cospan category.
By \cref{prop characterization of equivs}, $\pi$ and $\pi f$ are equivalent objects in $\lcc$, and by \cref{lem strictify lcc}, $\pi f$ is isomorphic to a strict labelled cospan category.
Hence $\pi$ is equivalent to a strict labelled cospan category.
\end{proof}

We now have enough machinery to establish the following uniqueness result for (strict) labelled cospan categories with a given domain. 

\begin{proposition}[Barkan--Steinebrunner]\label{prop uniqueness of lcc structure}
If $D$ is a permutative category, there is at most one strict labelled cospan category $D \to C(\ast)$.
If $C$ is a symmetric monoidal category, then any two labelled cospan categories $C \to \csp$ are equivalent in $\lcc$.
\end{proposition}
\begin{proof}
Suppose $\pi, \mu \colon C \to \csp$ are two labelled cospan categories with the same source category; we will reduce the existence of an equivalence $\pi \simeq \mu$ to the first statement.
Note that $\pi(c)$ has cardinality zero if and only if $\mu(c)$ does by \cref{def lcc}\eqref{lcc: obj decomp}; then if a $\pi$-connected object $c$ decomposes as $c \cong c_1 \otimes \dots \otimes c_n$ for $\mu$-connected objects $c_i$, we must have $n=1$.
It follows that the permuative category $D$ and the functor $f$ from the proof of \cref{lem slcc lcc surj equiv} are the same for both $\pi$ and $\mu$. 
Applying \cref{lem strictify lcc} we have $\pi \simeq \pi f \cong \tilde \pi'$ and $\mu \simeq \mu f \cong \tilde \mu'$ for strict labelled cospan categories $\pi', \mu' \colon D \to C(\ast)$.
These last two maps will turn out to be equal, establishing the second claim.

Given two strict labelled cospan categories $\pi', \mu' \colon D \to C(\ast)$ with the same domain permutative category $D$, we now show that $\pi' = \mu'$. 
Since free monoids have unique generating sets and the set of objects of $D$ is the free monoid on the set of connected objects, the connected objects of $\pi'$ and $\mu'$ are the same.
Hence, on objects, the maps $\pi'$ and $\mu'$ coincide; further, since $\pi'$ and $\mu'$ are strict maps of permutative categories, they agree on the maps $\hat \sigma \colon x_{\sigma(1)}\otimes \dots \otimes x_{\sigma(n)} \to x_1 \otimes \dots \otimes x_n$ for each permutation $\sigma$ of $\un$.
A free abelian monoid has a unique generating set, so the connected endomorphisms of $\mathbf{1}$ coincide for $\pi'$ and $\mu'$ by \cref{def lcc}\eqref{lcc: free gen}.
Since $\pi',\mu' \colon \hom(\mathbf{1},\mathbf{1}) \to \hom(\underline{0}, \underline{0})$ send the generators to the unique generator $\underline{0} \to \underline{1} \leftarrow \underline{0}$, these maps coincide on endomorphisms of the tensor unit.
Suppose that $p \colon c\to d$ is a morphism of $D$ which is connected with respect to $\pi'$.
We can write $p = \hat \sigma \circ (p_1 \otimes \dots \otimes p_k) \circ \hat \rho$ where each $p_i$ is connected with respect to $\mu'$.
Then since no $\pi'(p_i)$ is the identity on $\underline{0}$, we must have $k=1$, so $p$ is also connected with respect to $\mu'$.
Since $\pi'$ and $\mu'$ agree on objects and have the same connected morphisms, they agree on connected morphisms.
But then $\pi'(p) = \mu'(p)$ for arbitrary morphisms $p$, since we can write any such $p$ as $\hat \sigma \circ (p_1 \otimes \dots \otimes p_k) \circ \hat \rho$ with the $p_i$ connected, and $\pi'$ and $\mu'$ are strict maps of permutative categories.
Thus $\pi' = \mu'$.
\end{proof}

\begin{remark}
Suppose $f\colon C \to C'$ is a symmetric monoidal equivalence and $\pi \colon C \to \csp$, $\pi' \colon C' \to \csp$ are labelled cospan categories.
Then $f$ automatically determines a 1-morphism in $\lcc$.
Indeed, in this situation it is automatic that $\pi' f$ is a labelled cospan category, so unraveling the proof of \cref{prop uniqueness of lcc structure} we obtain a natural isomorphism $\pi' f \cong \pi$ as the following pasting composite.
\[ \begin{tikzcd}
& &[-0.5cm] C \ar[rr,"f"] &[-0.5cm] &[-0.5cm] C'  \ar[dr,"\pi'"] &[-0.5cm] \\
C \rar{\simeq} 
  \ar[urr,bend left,"\id"] \ar[drr,bend right,"\id"'] 
  \ar[urr,bend left=6, phantom,"\cong"] \ar[drr,bend right=6, phantom,"\cong"] 
  & D \ar[ur,"\simeq"] \ar[dr,"\simeq"] \ar[rr]
    \ar[rrrr, bend left=25,"\cong", phantom] \ar[rrrr, bend right=25,"\cong", phantom]
   & & C(\ast)  \ar[rr,"\simeq"] & & \csp \\
& & C \ar[urrr, bend right=20,"\pi"' near end]
\end{tikzcd} \]
Here, the functor $D\to C(\ast)$ is the strict labelled cospan category constructed in the first paragraph of the proof.
\end{remark}

\begin{theorem}\label{theorem slcc lcc biequiv}
The 2-functor $\slcc \to \lcc$ is a biequivalence.
\end{theorem}
\begin{proof}
This follows from \cref{lem slcc lcc loc equiv cat} and \cref{lem slcc lcc surj equiv}.
\end{proof}

We conclude by combining this theorem with the biequivalence of \cref{thm properad vs slcc 2-cat}.

\begin{corollary}\label{main theorem as corollary}
The composite 2-functor $\properad \to \lcc$ is a biequivalence. 
\qed
\end{corollary}

\bibliographystyle{amsalpha}
\bibliography{refs}

\providecommand{\bysame}{\leavevmode\hbox to3em{\hrulefill}\thinspace}
\providecommand{\MR}{\relax\ifhmode\unskip\space\fi MR }
\providecommand{\MRhref}[2]{%
  \href{http://www.ams.org/mathscinet-getitem?mr=#1}{#2}
}
\providecommand{\href}[2]{#2}
\begin{thebibliography}{RSW06}

\bibitem[BB17]{BataninBerger:HTAPM}
M.~A. Batanin and C.~Berger, \emph{Homotopy theory for algebras over polynomial
  monads}, Theory Appl. Categ. \textbf{32} (2017), Paper No. 6, 148--253.
  \MR{3607212}

\bibitem[BHS]{BarkanHaugsengSteinebrunner:EAP}
Shaul Barkan, Rune Haugseng, and Jan Steinebrunner, \emph{Envelopes for
  algebraic patterns}, preprint,
  \href{https://arxiv.org/abs/2208.07183}{arXiv:2208.07183} [math.CT].

\bibitem[BS]{Barkan-Steinebrunner}
Shaul Barkan and Jan Steinebrunner, \emph{The equifibered approach to
  $\infty$-properads}, preprint,
  \href{https://arxiv.org/abs/2211.02576}{arXiv:2211.02576} [math.AT].

\bibitem[Car91]{Carboni:MRGR}
Aurelio Carboni, \emph{Matrices, relations, and group representations}, J.
  Algebra \textbf{136} (1991), no.~2, 497--529. \MR{1089310}

\bibitem[CF17]{CoyaFong:CPECFM}
Brandon Coya and Brendan Fong, \emph{Corelations are the prop for extraspecial
  commutative {F}robenius monoids}, Theory Appl. Categ. \textbf{32} (2017),
  Paper No. 11, 380--395. \MR{3633708}

\bibitem[CH21]{ChuHaugseng}
Hongyi Chu and Rune Haugseng, \emph{Homotopy-coherent algebra via {S}egal
  conditions}, Adv. Math. \textbf{385} (2021), Paper No. 107733, 95.
  \MR{4256131}

\bibitem[CH22]{ChuHackney}
Hongyi Chu and Philip Hackney, \emph{On rectification and enrichment of
  infinity properads}, J. Lond. Math. Soc. (2) \textbf{105} (2022), no.~1,
  1418--1517. \MR{4389306}

\bibitem[Dun06]{Duncan:TQC}
Ross Duncan, \emph{Types for quantum computing}, Ph.D. thesis, Oxford
  University, 2006.

\bibitem[Gra66]{Gray:FCC}
John~W. Gray, \emph{Fibred and cofibred categories}, Proc. {C}onf.
  {C}ategorical {A}lgebra ({L}a {J}olla, {C}alif., 1965), Springer, New York,
  1966, pp.~21--83. \MR{0213413}

\bibitem[Hac]{Hackney:Properad2Cat}
Philip Hackney, \emph{The 2-category of properads}, (working title), in
  preparation.

\bibitem[HK]{HaugsengKock:IOSMIC}
Rune Haugseng and Joachim Kock, \emph{$\infty$-operads as symmetric monoidal
  $\infty$-categories}, Forthcoming in Publ.\ Mat.,
  \href{https://arxiv.org/abs/2106.12975}{arXiv:2106.12975} [math.CT].

\bibitem[HR15]{HackneyRobertson:OCP}
Philip Hackney and Marcy Robertson, \emph{On the category of props}, Appl.
  Categ. Structures \textbf{23} (2015), no.~4, 543--573. \MR{3367131}

\bibitem[HR17]{HackneyRobertson:HTSP}
\bysame, \emph{The homotopy theory of simplicial props}, Israel J. Math.
  \textbf{219} (2017), no.~2, 835--902. \MR{3649609}

\bibitem[HRY15]{HRYbook}
Philip Hackney, Marcy Robertson, and Donald Yau, \emph{Infinity properads and
  infinity wheeled properads}, Lecture Notes in Mathematics, vol. 2147,
  Springer, Cham, 2015. \MR{3408444}

\bibitem[HRY17]{HackneyRobertsonYau:SMIP}
\bysame, \emph{A simplicial model for infinity properads}, High. Struct.
  \textbf{1} (2017), no.~1, 1--21. \MR{3912049}

\bibitem[JY21]{JohnsonYau:2DC}
Niles Johnson and Donald Yau, \emph{2-dimensional categories}, Oxford
  University Press, Oxford, 2021. \MR{4261588}

\bibitem[KM]{KaufmannMonaco:PCPUFC}
Ralph~M. Kaufmann and Michael Monaco, \emph{Plus constructions, plethysm, and
  unique factorization categories with applications to graphs and operad-like
  theories}, preprint,
  \href{https://arxiv.org/abs/2209.06121}{arXiv:2209.06121} [math.CT].

\bibitem[Lac04]{Lack:CP}
Stephen Lack, \emph{Composing {PROPS}}, Theory Appl. Categ. \textbf{13} (2004),
  No. 9, 147--163. \MR{2116328}

\bibitem[LR20]{LoregianRiehl:CNF}
Fosco Loregian and Emily Riehl, \emph{Categorical notions of fibration}, Expo.
  Math. \textbf{38} (2020), no.~4, 496--514. \MR{4177953}

\bibitem[Lur09]{LurieHTT}
Jacob Lurie, \emph{Higher topos theory}, Annals of Mathematics Studies, vol.
  170, Princeton University Press, Princeton, NJ, 2009. \MR{2522659}

\bibitem[LV12]{LodayVallette:AO}
Jean-Louis Loday and Bruno Vallette, \emph{Algebraic operads}, Grundlehren der
  mathematischen Wissenschaften, vol. 346, Springer, Heidelberg, 2012.
  \MR{2954392}

\bibitem[{Mac}65]{MacLane:CA}
Saunders {Mac Lane}, \emph{Categorical algebra}, Bull. Amer. Math. Soc.
  \textbf{71} (1965), 40--106. \MR{171826}

\bibitem[May74]{May74}
J.~P. May, \emph{{$E_{\infty}$} spaces, group completions, and permutative
  categories}, New developments in topology ({P}roc. {S}ympos. {A}lgebraic
  {T}opology, {O}xford, 1972), Cambridge, 1974, London Math. Soc. Lecture Note
  Ser., No. 11, pp.~61--93.

\bibitem[na]{nlabKan}
nLab authors, \emph{Kan extension along (op)fibration},
  \url{https://ncatlab.org/nlab/show/Kan+extension#AlongFibrations}.

\bibitem[RSW06]{RSW:GCSACG}
R.~Rosebrugh, N.~Sabadini, and R.~F.~C. Walters, \emph{Generic commutative
  separable algebras and cospans of graphs}, Theory Appl. Categ. \textbf{15}
  (2005/06), No. 6, 164--177. \MR{2210579}

\bibitem[RV20]{RiehlVerity:construction}
Emily Riehl and Dominic Verity, \emph{On the construction of limits and
  colimits in {$\infty$}-categories}, Theory Appl. Categ. \textbf{35} (2020),
  Paper No. 30, 1101--1158. \MR{4127725}

\bibitem[RV22]{RiehlVerity:EOICT}
\bysame, \emph{Elements of {$\infty$}-category theory}, Cambridge Studies in
  Advanced Mathematics, vol. 194, Cambridge University Press, Cambridge, 2022.
  \MR{4354541}

\bibitem[Ste]{Steinebrunner}
Jan Steinebrunner, \emph{The surface category and tropical curves}, preprint,
  \href{https://arxiv.org/abs/2111.14757v1}{arXiv:2111.14757v1} [math.AT].

\bibitem[Val07]{Vallette:KDP}
Bruno Vallette, \emph{A {K}oszul duality for {PROP}s}, Trans. Amer. Math. Soc.
  \textbf{359} (2007), no.~10, 4865--4943. \MR{2320654 (2008e:18020)}

\bibitem[YJ15]{YauJohnson:FPAM}
Donald Yau and Mark~W. Johnson, \emph{A foundation for {PROP}s, algebras, and
  modules}, Mathematical Surveys and Monographs, vol. 203, American
  Mathematical Society, Providence, RI, 2015. \MR{3329226}

\end{thebibliography}
\end{document}